\documentclass[11pt]{article}

\usepackage[a4paper,margin=1in]{geometry}
\usepackage{amsmath,amssymb,amsthm,mathtools}
\usepackage{microtype}
\usepackage{enumitem}
\usepackage{cite}
\usepackage[hidelinks]{hyperref}
\hypersetup{
    pdftitle={Sharp Propagation and the Local-to-Nonlocal Transition for Fractional Fisher--KPP Equations on Integer Lattices},
    pdfauthor={Yuanyang Hu},
    pdfkeywords={fractional graph Laplacian, Fisher--KPP equation, integer lattice, Wulff shape, accelerating propagation, singular transition}
}

\numberwithin{equation}{section}
\allowdisplaybreaks

\newtheorem{theorem}{Theorem}[section]
\newtheorem{proposition}[theorem]{Proposition}
\newtheorem{lemma}[theorem]{Lemma}
\newtheorem{corollary}[theorem]{Corollary}
\theoremstyle{definition}
\newtheorem{definition}[theorem]{Definition}
\theoremstyle{remark}
\newtheorem{remark}[theorem]{Remark}

\newcommand{\Z}{\mathbb Z}
\newcommand{\R}{\mathbb R}
\newcommand{\N}{\mathbb N}
\newcommand{\one}{\mathbf 1}
\newcommand{\ellinf}{\ell^\infty(\Z^d)}

\newcommand{\supp}{\operatorname{supp}}
\newcommand{\ee}{\mathrm e}

\title{Sharp Propagation and the Local-to-Nonlocal Transition for\\
Fractional Fisher--KPP Equations on Integer Lattices}
\author{Yuanyang Hu\thanks{School of Mathematics and Statistics, Henan University,
Kaifeng 475004, China.}}
\date{}

\begin{document}

\maketitle

\begin{abstract}
Let $P$ denote normalized nearest-neighbor averaging on $\Z^d$ and set
$A=I-P$. We consider
$
    \partial_tu+A^su=f(u),
    ~~ t>0,~~ x\in\Z^d, 0<s<1.
$
For a Fisher--KPP reaction with $a=f'(0)>0$, compactly supported data have
logarithmic propagation rate $a/(d+2s)$. In the logistic case, each fixed
level set lies between two constant multiples of
$\exp(at/(d+2s))$; this localization remains valid for data with the
critical algebraic tail. Nondecreasing front-like data in one dimension
instead propagate at rate $a/(2s)$, and the resulting acceleration rules
out planar traveling fronts of finite speed in every dimension. We also
analyze the singular limit $s\uparrow1$. If
$\tau_s=-\log(1-s)$, then throughout $at<\tau_s$ compactly supported
solutions follow the Wulff shape of the local nearest-neighbor equation.
After a constant multiple of $\tau_s$, their propagation is exponential.
Half-space data exhibit the corresponding transition from the directional
local speed to the exponent $a/(2s)$. The analysis rests on sharp bounds for
the subordinated kernel, algebraic barriers, a compactness--Liouville
argument, and a decomposition of the fractional walk into its
nearest-neighbor and rare long-jump parts.
\end{abstract}


\medskip


\section{Introduction}

Let
\[
    (P\phi)(x)=\frac1{2d}\sum_{|e|=1}\phi(x+e),
    \qquad A=I-P=-\frac1{2d}\Delta_d,
\]
on $\Z^d$. This paper concerns the Fisher--KPP equation
\begin{equation}\label{eq:intro-equation}
    \partial_tu(t,x)+A^su(t,x)=f(u(t,x)),
    \qquad t>0,\quad x\in\Z^d.
\end{equation}
Here $0<s<1$, so $A^s=(2d)^{-s}(-\Delta_d)^s$. Although $A$ only permits
nearest-neighbor motion, its fractional power generates jumps of every
length. The corresponding kernel decays algebraically, and this heavy tail
changes the propagation scale from linear to exponential.

For the classical Fisher--KPP equation, compactly supported populations
invade at a linear speed characterized by traveling fronts
\cite{Fisher1937,Kolmogorov1937,AronsonWeinberger1978,LiangZhao2010}.
Nearest-neighbor lattice diffusion has the same finite-speed character
\cite{ZinnerHarrisHudson1993,LiangZhou2020}, though the lattice affects the
directional minimal speed \cite{FangLiLouWang2026} and the finer front
asymptotics \cite{BesseFayeRoquejoffreZhang2023}. By contrast, Cabr\'e and
Roquejoffre showed that for fractional diffusion in $\R^d$, localized data
have logarithmic propagation rate $f'(0)/(d+2s)$, while one-dimensional
front-like data have rate $f'(0)/(2s)$
\cite{CabreRoquejoffre2013}. The periodic and singular-limit counterparts
were developed in \cite{CabreCoulonRoquejoffre2012,CoulonRoquejoffre2012},
and slowly decaying initial data were treated in \cite{FelmerYangari2013}.

Equation \eqref{eq:intro-equation} also lies within the analysis of evolution
equations on graphs. On a weighted graph, heat-kernel bounds reflect volume
growth, Poincar\'e inequalities, and curvature properties
\cite{Dodziuk1984,Delmotte1999,BauerEtAl2015,HornLinLiuYau2019}; these
estimates enter nonlinear questions such as the Fujita alternative
\cite{LinWu2017}. Variational graph methods have likewise been developed for
Kazdan--Warner and Yamabe-type equations
\cite{GrigoryanLinYang2016KW,GrigoryanLinYang2016Yamabe}. Fractional powers
of graph Laplacians are defined by spectral calculus or heat-semigroup
subordination and have been studied from operator-theoretic and evolutionary
viewpoints in
\cite{BendikovSaloffCoste2012,CiaurriEtAl2018,LizamaRoncal2018,KeyantuoLizamaWarma2019,BenziEtAl2020};
see also \cite{KellerLenzWojciechowski2021}. Our concern is the nonlinear
propagation law produced when this subordinated lattice walk is coupled to a
KPP reaction.

Heavy-tailed dispersal does not by itself determine a universal rate. The
behavior near the unstable state, the decay of the initial datum, and the
geometry of the occupied region all matter. Nonlinear fractional diffusion,
cooperative systems, and periodic media retain exponential propagation in
appropriate KPP regimes
\cite{StanVazquez2014,CoulonYangari2017,MeleardMirrahimi2015,Leculier2019,SouganidisTarfulea2019,LeculierMirrahimiRoquejoffre2021}.
For general monostable, ignition, or weak-Allee reactions, the rate may
instead be algebraic or finite
\cite{CovilleGuiZhao2021,ZhangZlatos2023,BouinCovilleLegendre2025}.
Related distinctions occur for convolution equations, slowly decaying data,
and nonlocal free-boundary problems
\cite{Garnier2011,HamelRoques2010,FinkelshteinKondratievTkachov2019,XuLiRuan2021,CaoDuLiLi2019,DuLiZhou2021}.

Our first aim is to obtain the sharp fixed-order propagation laws on
$\Z^d$. For compactly supported data, the jump kernel satisfies
\[
    K_s(z)\asymp |z|^{-d-2s}.
\]
Balancing this tail against the linear KPP growth $\ee^{f'(0)t}$ suggests
the radius $\exp(f'(0)t/(d+2s))$. Theorem~\ref{thm:main} confirms this
logarithmic rate under the KPP upper bound $f(r)\leq f'(0)r$, without a
concavity assumption. For the logistic reaction, the conclusion can be
sharpened to
\[
    R_\theta(t)\asymp\exp\left(\frac{at}{d+2s}\right),
    \qquad 0<\theta<1,
\]
as stated in Theorem~\ref{thm:logistic-localization}. The same
constant-factor localization holds for nonzero data bounded above by the
critical tail $(1+|x|)^{-d-2s}$; see
Theorem~\ref{thm:critical-tail-localization}.

The exponent is different for an initially occupied half-line. When $d=1$,
the pointwise decay $K_s(n)\asymp n^{-1-2s}$ has the cumulative tail
$\sum_{n\geq |j|}K_s(n)\asymp |j|^{-2s}$. Accordingly,
Theorem~\ref{thm:monotone-sharp-propagation} gives the rate $f'(0)/(2s)$
for nondecreasing front-like data. The acceleration also precludes planar
traveling fronts of finite speed in every direction. This agrees with the
general obstruction caused by the absence of positive exponential moments
\cite{Yagisita2009}, but our proof follows directly from the spreading
estimates. In the continuous fractional problem, acceleration is also tied
to stretching and asymptotic changes in front geometry
\cite{GarnierHamelRoques2017,RoquejoffreTarfulea2017}; related lattice models
with algebraically decaying dispersal likewise exhibit infinite invasion
speed \cite{LiXinZhao2025}.

Our second aim is to describe how these exponential laws emerge as
$s\uparrow1$. The fractional jump kernel splits into a nearest-neighbor part
and a long-jump part of total mass comparable to $1-s$. The reaction
amplifies the latter by $\ee^{at}$, so the balance
\[
    (1-s)\ee^{at}\asymp1
\]
selects the time $\tau_s/a$, where $\tau_s=-\log(1-s)$. Before this time,
compactly supported solutions follow the anisotropic Wulff shape of the
local lattice equation. After a constant multiple of $\tau_s$, the
fixed-order exponential law holds uniformly near $s=1$. Half-space data
undergo the analogous transition from the directional local speed to the
exponent $a/(2s)$. This parallels the continuous transition in
\cite{CoulonRoquejoffre2012}, with the Euclidean ball replaced by the Wulff
shape determined by the nearest-neighbor walk.

Two lattice features require separate arguments. There is no exact stable
scaling or explicit fractional heat kernel, and a prescribed level set may
contain no lattice point. We therefore work with superlevel sets and derive
the necessary kernel estimates from subordination and a compound-Poisson
representation. One large jump gives the lower tail, while a convolution
estimate controls repeated jumps. Expanding algebraic barriers create a
positive plateau on subcritical exponential regions, and a
compactness--Liouville argument upgrades that plateau to convergence to one.
In the logistic case the profile
$\Theta_R(x)=(1+(|x|/R)^{d+2s})^{-1}$ yields the sharper bounded-error
localization. Near $s=1$, separating the nearest-neighbor walk from the rare
long jumps makes the factor $1-s$ explicit and permits estimates up to the
full early range $at<\tau_s$.

The paper is organized as follows. Section~\ref{sec:setting} introduces the
fractional lattice operator, fixes the hypotheses and notation, establishes
well-posedness, and states both the fixed-order and transition results. The
fixed-order kernel and semigroup estimates are established in
Section~\ref{sec:kernel}, followed by the comparison profiles in
Section~\ref{sec:barrier}. Sections~\ref{sec:lower}--\ref{sec:monotone-data}
contain the spreading proofs, and Section~\ref{sec:traveling-fronts} treats
planar traveling fronts. The local Wulff regime is proved in
Section~\ref{tr:sec:local}. The estimates uniform near $s=1$ and the onset of
exponential propagation are developed in Sections~\ref{tr:sec:uniform} and
\ref{tr:sec:acceleration}.

\section{Setting and main results}\label{sec:setting}

Throughout the paper, $d\in\N$ and $d\geq1$. The order $s\in(0,1)$ is
fixed in the first group of results and varies only in the singular-transition
results stated below. The symbols $|x|$ and $|x|_1$ denote, respectively,
the Euclidean and $\ell^1$ norms of $x\in\Z^d$, and
$B_R:=\{x\in\Z^d:|x|\leq R\}$.
Since all norms on $\R^d$ are equivalent, the propagation exponent is independent of the chosen norm. For $1\leq p<\infty$, $\ell^p(\Z^d)$ is equipped with the norm
$\|v\|_{\ell^p}:=(\sum_{x\in\Z^d}|v(x)|^p)^{1/p}$,
while $\|v\|_{\ell^\infty}:=\sup_{x\in\Z^d}|v(x)|$. For functions $g,h$ on $\Z^d$, whenever the sum is absolutely convergent, their convolution is
$(g*h)(x):=\sum_{y\in\Z^d}g(x-y)h(y)$.
We write $\delta_0$ for the unit mass at the origin and adopt the convention $g^{*0}=\delta_0$.

For two nonnegative quantities $X$ and $Y$, the notation $X\lesssim Y$
means that $X\leq CY$ for a constant $C>0$ independent of the displayed
variables; $X\asymp Y$ means that both inequalities hold. Thus
$K_s(z)\asymp |z|^{-d-2s}$ as $|z|\to\infty$ means that, for suitable
$c,C,R_0>0$,
$c|z|^{-d-2s}\leq K_s(z)\leq C|z|^{-d-2s}$ whenever $|z|\geq R_0$.
The letters $c$ and $C$ denote positive constants that may change from line
to line. Unless stated otherwise, they are independent of $t,x,z$ and the
scale $R$, but may depend on fixed data such as $d,s,f$, and $u_0$. We set
$q:=d+2s$.

The operator $P$ is a positive contraction on every $\ell^p(\Z^d)$, $1\leq p\leq\infty$. Hence $A=I-P$ is bounded on these spaces, and it is nonnegative and self-adjoint on $\ell^2(\Z^d)$. Its fractional power on $\ell^2$ can be defined by spectral calculus. On bounded functions, we use the equivalent Balakrishnan formula \cite{Balakrishnan1960,SchillingSongVondracek2012}
\begin{equation}\label{eq:balakrishnan}
    A^s\phi
    =c_s\int_0^\infty
    \bigl(\phi-\ee^{-rA}\phi\bigr)
    \frac{dr}{r^{1+s}},
    \qquad
    c_s:=\frac{s}{\Gamma(1-s)}.
\end{equation}
Indeed,
\begin{equation}\label{eq:ordinary-semigroup-series}
    \ee^{-rA}
    =\ee^{-r}\ee^{rP}
    =\ee^{-r}\sum_{n=0}^{\infty}\frac{r^n}{n!}P^n
\end{equation}
is a positive contraction on $\ell^\infty$. Moreover,
$\|I-\ee^{-rA}\|_{\ell^\infty\to\ell^\infty}
\leq r\|A\|_{\ell^\infty\to\ell^\infty}$ for $0<r\leq1$,
whereas the same norm is at most $2$ for $r\geq1$. Since $0<s<1$, the integral in \eqref{eq:balakrishnan} therefore converges in $\ell^\infty$. On $\ell^2\cap\ell^\infty$, it agrees with the spectral definition by the scalar identity
$\lambda^s=c_s\int_0^\infty(1-\ee^{-r\lambda})r^{-1-s}\,dr$ for
$\lambda\geq0$.

We impose the following hypotheses.

\medskip

\noindent\textbf{Assumption (F).}
\begin{enumerate}
\renewcommand{\labelenumi}{(F\arabic{enumi})}
\item The function $f\in C^1([0,1])$ satisfies
\begin{equation}\label{eq:F-hypotheses}
    f(0)=f(1)=0,
    \qquad
    f(r)>0\quad\text{for }0<r<1.
\end{equation}
\item With $a:=f'(0)$, one has $a>0$ and
\begin{equation}\label{eq:KPP-inequality}
    f(r)\leq ar
    \quad\text{for }0\leq r\leq1.
\end{equation}
\end{enumerate}
When only assumption \emph{(F)} is in force, we use an arbitrary locally
Lipschitz extension of $f$ to $\R$ to construct the local Banach-space
solution. The invariant-interval argument below shows that the
resulting solution takes values in $[0,1]$ and is independent of this extension.

\medskip

\noindent\textbf{Assumption (I).}
The initial datum $u_0\in\ellinf$ is finitely supported and satisfies
\begin{equation}\label{eq:initial-hypotheses}
    0\leq u_0\leq1,
    \qquad
    u_0\not\equiv0.
\end{equation}

\begin{definition}\label{def:classical-solution}
Let $T\in(0,\infty]$. A function $u$ is a classical solution of \eqref{eq:intro-equation} on $[0,T)$ with initial datum $u_0$ if
${u\in C^1([0,T);\ellinf)}$, $u(0)=u_0$,
and the equation holds in $\ellinf$, equivalently at every $x\in\Z^d$, for all $t\in(0,T)$.
\end{definition}
Here $C^1([0,\infty);\ellinf)$ is understood locally on every compact time
interval when $T=\infty$.

\begin{proposition}\label{prop:well-posedness}
Under assumption \emph{(F)}, every $u_0\in\ellinf$ satisfying
$0\leq u_0\leq1$ generates a unique global classical solution of
\eqref{eq:intro-equation}. This solution satisfies $0\leq u(t,x)\leq1$ for
$t\geq0$ and $x\in\Z^d$.
\end{proposition}

\begin{proof}
Extend $f$ from $[0,1]$ to a locally Lipschitz function on $\R$. By the estimate following \eqref{eq:balakrishnan}, $A^s$ is bounded on $\ellinf$. Hence the map
$v\mapsto-A^sv+f(v)$
is locally Lipschitz on $\ellinf$, with $f(v)$ understood pointwise. The Banach-space Picard theorem therefore gives a unique maximal classical solution.

Let $[0,T]$ be a compact subinterval of the maximal interval. The range of
the solution on $[0,T]$, together with $0$ and $1$, lies in a compact
interval on which the chosen extension of $f$ is Lipschitz. We may therefore
apply Lemma~\ref{lem:comparison} on $[0,T]$. Since the constant functions
$0$ and $1$ are solutions, comparison yields
$0\leq u(t,x)\leq1$ for $0\leq t\leq T$ and $x\in\Z^d$.
Thus the solution remains in a fixed bounded set throughout its maximal
interval. The continuation criterion for ODEs in $\ellinf$ gives global
existence, and the same invariant bound shows that the solution is independent
of the chosen extension of $f$.
\end{proof}

For $0<\theta<1$, define the outer radius of the $\theta$-superlevel set by
\begin{equation}\label{eq:level-radius}
    R_\theta(t)
    :=\sup\bigl\{|x|:u(t,x)\geq\theta\bigr\}.
\end{equation}
We use the convention $\sup\varnothing=0$. The pointwise tail bound
\eqref{eq:nonlinear-tail-bound} below implies that $R_\theta(t)<\infty$ for
every $t>0$, and the lower estimate implies that $R_\theta(t)>0$ for all
sufficiently large $t$.

We begin with the fixed-order law for compactly supported data.

\begin{theorem}\label{thm:main}
Assume \emph{(F)} and \emph{(I)}. Then, for every $\varepsilon>0$,
\begin{equation}\label{eq:main-lower}
    \lim_{t\to\infty}
    \inf_{\substack{x\in\Z^d\\ |x|\leq
    \exp((a-\varepsilon)t/(d+2s))}}
    u(t,x)=1,
\end{equation}
whereas
\begin{equation}\label{eq:main-upper}
    \lim_{t\to\infty}
    \sup_{\substack{x\in\Z^d\\ |x|\geq
    \exp((a+\varepsilon)t/(d+2s))}}
    u(t,x)=0.
\end{equation}
Consequently,
$\lim_{t\to\infty}t^{-1}\log R_\theta(t)=f'(0)/(d+2s)$ for every
$0<\theta<1$.
\end{theorem}

The first-order limit in Theorem~\ref{thm:main} does not imply a bounded
error in $\log R_\theta(t)$. For the logistic reaction we obtain the stronger
constant-factor localization below.

\medskip

\noindent\textbf{Assumption (L).}
For some $a>0$, the reaction is the globally defined logistic polynomial
\begin{equation}\label{eq:logistic-reaction}
    f(r)=ar(1-r),
    \qquad r\in\R.
\end{equation}

\begin{theorem}\label{thm:logistic-localization}
Assume \emph{(I)} and \emph{(L)}. For every $0<\theta<1$, there exist
$c_\theta,C_\theta>0$ and $T_\theta>0$ such that
\begin{equation}\label{eq:logistic-level-inclusions}
    B_{c_\theta\exp(at/q)}
    \subseteq
    \{x\in\Z^d:u(t,x)\geq\theta\}
    \subseteq
    B_{C_\theta\exp(at/q)}
\end{equation}
for every $t\geq T_\theta$. In particular,
\begin{equation}\label{eq:logistic-radius-O-one}
    \log R_\theta(t)=\frac{a}{d+2s}t+O(1)
    \qquad\text{as }t\to\infty.
\end{equation}
\end{theorem}

The compact-support assumption in Theorem~\ref{thm:logistic-localization}
can be relaxed to the critical algebraic decay dictated by the fractional
jump kernel.

\begin{theorem}
\label{thm:critical-tail-localization}
Assume \emph{(L)}. Let $u_0\in\ellinf$ satisfy
\begin{equation}\label{eq:critical-tail-initial-data}
    0\leq u_0\leq1,
    \qquad
    u_0\not\equiv0,
\end{equation}
and suppose that, for some $C_0>0$,
\begin{equation}\label{eq:critical-tail-upper-bound}
    u_0(x)\leq C_0(1+|x|)^{-q},
    \qquad x\in\Z^d.
\end{equation}
Then, for every $0<\theta<1$, there exist
$c_\theta,C_\theta>0$ and $T_\theta>0$ such that
\begin{equation}\label{eq:critical-tail-level-inclusions}
    B_{c_\theta\exp(at/q)}
    \subseteq
    \{x\in\Z^d:u(t,x)\geq\theta\}
    \subseteq
    B_{C_\theta\exp(at/q)}
\end{equation}
for every $t\geq T_\theta$. Consequently,
\begin{equation}\label{eq:critical-tail-radius}
    \log R_\theta(t)
    =\frac{a}{d+2s}t+O(1)
    \qquad\text{as }t\to\infty.
\end{equation}
\end{theorem}

For the lattice problem it is preferable to formulate the location of the
interface through superlevel sets: for a prescribed $\theta$, the exact level
set $\{u(t,\cdot)=\theta\}$ may contain no lattice point. Define
\begin{equation}\label{eq:inner-level-radius}
    r_\theta(t)
    :=\sup\{r\geq0:B_r\subseteq\{x\in\Z^d:u(t,x)\geq\theta\}\}.
\end{equation}
For a set $E\subseteq\Z^d$, let
\[
    \partial_{\rm d}E
    :=\{x\in E:\text{there exists }y\notin E
    \text{ with }|x-y|_1=1\}
\]
denote its inner nearest-neighbor boundary.

\begin{corollary}
\label{cor:critical-tail-interface}
Under the assumptions of Theorem~\ref{thm:critical-tail-localization}, for
every $0<\theta<1$,
\begin{equation}\label{eq:inner-outer-radius-asymptotics}
    \log r_\theta(t)
    =\frac{a}{d+2s}t+O(1),
    \qquad
    \log R_\theta(t)
    =\frac{a}{d+2s}t+O(1).
\end{equation}
Moreover, if
\[
    E_\theta(t):=\{x\in\Z^d:u(t,x)\geq\theta\},
\]
then there exist $\underline C_\theta,\overline C_\theta>0$ and
$T_\theta'>0$ such that
\begin{equation}\label{eq:discrete-interface-annulus}
    \partial_{\rm d}E_\theta(t)
    \subseteq
    \left\{x\in\Z^d:
    \underline C_\theta\ee^{at/q}
    \leq |x|\leq
    \overline C_\theta\ee^{at/q}\right\}
\end{equation}
for every $t\geq T_\theta'$.
\end{corollary}

We also consider an initially occupied half-lattice. In this case the
relevant tail is cumulative. The outer estimate uses the critical left-tail
bound below; the inner estimate does not.

\begin{theorem}\label{thm:monotone-sharp-propagation}
Let $d=1$ and assume \emph{(F)}. Suppose that $u_0\in\ell^\infty(\Z)$ is
nonzero and satisfies
\begin{equation}\label{eq:monotone-initial-data}
    0\leq u_0\leq1,
    \qquad
    u_0(j)\leq u_0(j+1)
    \quad (j\in\Z),
\end{equation}
together with
\begin{equation}\label{eq:monotone-left-tail}
    u_0(j)\leq C_0(-j)^{-2s},
    \qquad j\leq-1,
\end{equation}
for some $C_0>0$. Then
\begin{align}
    \lim_{t\to\infty}
    \sup_{\substack{j\in\Z\\j\leq-\exp(\sigma t)}}u(t,j)&=0
    &&\text{if }\sigma>\frac{a}{2s},
    \label{eq:monotone-outer-conclusion}\\
    \lim_{t\to\infty}
    \inf_{\substack{j\in\Z\\j\geq-\exp(\sigma t)}}u(t,j)&=1
    &&\text{if }\sigma<\frac{a}{2s}.
    \label{eq:monotone-inner-conclusion}
\end{align}
For $0<\theta<1$, let
\begin{equation}\label{eq:monotone-interface-definition}
    X_\theta(t):=\inf\{j\in\Z:u(t,j)\geq\theta\}.
\end{equation}
Then $X_\theta(t)$ is a finite negative integer for all sufficiently large
$t$, and
\begin{equation}\label{eq:monotone-interface-law}
    \lim_{t\to\infty}
    \frac{1}{t}\log\bigl(-X_\theta(t)\bigr)
    =\frac{f'(0)}{2s}.
\end{equation}
\end{theorem}

\subsection{The singular transition as \texorpdfstring{$s\uparrow1$}{s to 1}}

We next let the order vary and specialize to assumption \emph{(L)}. For each
$s\in(0,1)$, let $u_s$ denote the solution of \eqref{eq:intro-equation} with
the same initial datum $u_0$. The small mass of the non-nearest-neighbor part
of the fractional kernel determines the transition time. We first record the
local lattice problem that governs the limit $s\uparrow1$.

The early propagation set is determined by the local equation
\begin{equation}\label{tr:eq:local-equation}
  \partial_tu_1+Au_1=au_1(1-u_1).
\end{equation}
Set
\begin{align}
  \Lambda_a(\xi)
  &:=a-1+\frac1d\sum_{j=1}^d\cosh\xi_j,
  \label{tr:eq:cumulant}\\
  I_a(v)&:=\sup_{\xi\in\R^d}\{\xi\cdot v-\Lambda_a(\xi)\},
  \qquad
  \mathcal W_a:=\{v\in\R^d:I_a(v)\leq0\}.
  \label{tr:eq:wulff}
\end{align}
The set $\mathcal W_a$ is compact, convex, symmetric under coordinate
reflections, and has nonempty interior. For $e\in\mathbb S^{d-1}$, define
\begin{equation}\label{tr:eq:directional-speed}
  c_a(e):=\inf_{\lambda>0}\frac{\Lambda_a(\lambda e)}{\lambda}.
\end{equation}

We write
\begin{equation}\label{tr:eq:transition-time}
  \tau_s:=-\log(1-s).
\end{equation}
All limits below are joint limits subject to the displayed restrictions. For
example, \eqref{tr:eq:late-localized-inner} is taken as $s\uparrow1$ and
$t\to\infty$ under the condition
$t\geq C_{\sigma_-,\sigma_+}\tau_s$.

\begin{theorem}\label{tr:thm:localized-transition}
Let $u_0\in\ell^\infty(\Z^d)$ be nonzero and finitely supported, with
$0\leq u_0\leq1$. For $s\in(0,1)$, let $u_s$ be the solution of
\eqref{eq:intro-equation} with $u_s(0)=u_0$.

If $s\uparrow1$, $t\to\infty$, and
\begin{equation}\label{tr:eq:early-condition}
  at<\tau_s,
\end{equation}
then, for every $\delta\in(0,1)$,
\begin{align}
  \inf_{x\in(1-\delta)t\mathcal W_a\cap\Z^d}
  u_s(t,x)&\to1,
  \label{tr:eq:early-localized-inner}\\
  \sup_{\substack{x\in\Z^d\\
  x\notin(1+\delta)t\mathcal W_a}}
  u_s(t,x)&\to0.
  \label{tr:eq:early-localized-outer}
\end{align}

Let
\[
  0<\sigma_-<\frac{a}{d+2}<\sigma_+.
\]
Equivalently, $\sigma_-<a/(d+2s)<\sigma_+$ for all $s$ sufficiently close
to one. There is $C_{\sigma_-,\sigma_+}>0$, independent of $s$, such that
\begin{align}
  \lim_{\substack{s\uparrow1,\ t\to\infty\\
  t\geq C_{\sigma_-,\sigma_+}\tau_s}}
  \inf_{\substack{x\in\Z^d\\ |x|\leq\exp(\sigma_-t)}}
  u_s(t,x)&=1,
  \label{tr:eq:late-localized-inner}\\
  \lim_{\substack{s\uparrow1,\ t\to\infty\\
  t\geq C_{\sigma_-,\sigma_+}\tau_s}}
  \sup_{\substack{x\in\Z^d\\ |x|\geq\exp(\sigma_+t)}}
  u_s(t,x)&=0.
  \label{tr:eq:late-localized-outer}
\end{align}
\end{theorem}

For $e\in\mathbb S^{d-1}$ and $\kappa\in\R$, put
\[
  H_{e,\kappa}(x):=\one_{\{x\cdot e\geq\kappa\}},
  \qquad x\in\Z^d.
\]

\begin{theorem}\label{tr:thm:front-transition}
Fix $e\in\mathbb S^{d-1}$. Suppose that $u_0$ is independent of $s$ and
that, for some $\mu\in(0,1]$ and $\kappa_-,\kappa_+\in\R$,
\begin{equation}\label{tr:eq:front-data}
  \mu H_{e,\kappa_+}\leq u_0\leq H_{e,\kappa_-},
  \qquad 0\leq u_0\leq1.
\end{equation}
Let $u_s$ solve \eqref{eq:intro-equation}. If $s\uparrow1$,
$t\to\infty$, and $at<\tau_s$, then, for every
$\delta\in(0,c_a(e))$,
\begin{align}
  \sup_{\substack{x\in\Z^d\\ x\cdot e\leq-(c_a(e)+\delta)t}}
  u_s(t,x)&\to0,
  \label{tr:eq:early-front-outer}\\
  \inf_{\substack{x\in\Z^d\\ x\cdot e\geq-(c_a(e)-\delta)t}}
  u_s(t,x)&\to1.
  \label{tr:eq:early-front-inner}
\end{align}
Let
\[
  0<\sigma_-<\frac a2<\sigma_+.
\]
Equivalently, $\sigma_-<a/(2s)<\sigma_+$ for all $s$ sufficiently close
to one. There is $\widehat C_{\sigma_-,\sigma_+}>0$, independent of $s$,
such that
\begin{align}
  \lim_{\substack{s\uparrow1,\ t\to\infty\\
  t\geq \widehat C_{\sigma_-,\sigma_+}\tau_s}}
  \sup_{\substack{x\in\Z^d\\ x\cdot e\leq-\exp(\sigma_+t)}}
  u_s(t,x)&=0,
  \label{tr:eq:late-front-outer}\\
  \lim_{\substack{s\uparrow1,\ t\to\infty\\
  t\geq \widehat C_{\sigma_-,\sigma_+}\tau_s}}
  \inf_{\substack{x\in\Z^d\\ x\cdot e\geq-\exp(\sigma_-t)}}
  u_s(t,x)&=1.
  \label{tr:eq:late-front-inner}
\end{align}
\end{theorem}

\begin{remark}\label{tr:rem:normalization}
The local speed in \eqref{tr:eq:directional-speed} reflects the normalization
$A=I-P$. Even in one dimension it is not the Euclidean value $2\sqrt a$.
Thus replacing the early Wulff shape by a ball of radius $2\sqrt a\,t$
would give an incorrect lattice statement.
\end{remark}

\begin{remark}\label{tr:rem:sigma-quantifiers}
The late conclusions are often written in the shorter form
\[
  \sigma>\frac{a}{d+2s}
  \quad\text{or}\quad
  0<\sigma<\frac{a}{d+2s}
\]
for localized data, and with $a/(2s)$ for front-like data. In a joint limit
$s\uparrow1$, this notation means that the fixed rate $\sigma$ remains on
the indicated side of the critical rate with a positive margin. The
constant after which the exponential regime is valid may depend on this
margin, and hence on $\sigma$, but it is independent of $s$. No single
transition constant is asserted to work uniformly for rates converging to
the critical one.
\end{remark}

\begin{remark}\label{rem:initial-data}
Suppose that $f$ is defined on $[0,\infty)$, is locally Lipschitz, satisfies
$f(r)<0$ for $r>1$, and obeys $f(r)\leq ar$ for every $r\geq0$. Then
Theorem~\ref{thm:main} extends to every bounded, nonzero, nonnegative,
finitely supported initial datum. The scalar equation $y'=f(y)$, started
above $\max\{1,\|u_0\|_{\ell^\infty}\}$, gives a spatially constant upper
solution that converges to one. For the lower estimate, compare with the
solution arising from a nonzero truncation of $u_0$ with values in $[0,1]$.
\end{remark}

\section{The fractional jump kernel and its semigroup}\label{sec:kernel}

Let $h_r(z)$ be the convolution kernel of $\ee^{-rA}$. By
\eqref{eq:ordinary-semigroup-series}, it is the transition kernel of the
continuous-time normalized nearest-neighbor random walk with total jump rate
one. We use it to estimate the fractional jump kernel and its heat kernel;
background on random-walk kernels can be found in \cite{LawlerLimic2010}.

\subsection{A one-dimensional estimate}

Let $g_\tau(k)$ be the transition probability at $k\in\Z$ for the continuous-time one-dimensional symmetric random walk with total jump rate one. Its two oriented jump counts are independent Poisson random variables with mean $\tau/2$. Conditioning on the total number of jumps gives
\begin{equation}\label{eq:one-dimensional-kernel}
    g_\tau(k)
    =\ee^{-\tau}
    \sum_{\substack{m\geq |k|\\m\equiv k\ ({\rm mod}\ 2)}}
    \frac{\tau^m}{m!}\,2^{-m}
    \binom{m}{(m+k)/2}.
\end{equation}

\begin{lemma}\label{lem:one-dimensional-walk}
There exist constants $c,C>0$ such that
\begin{equation}\label{eq:one-dimensional-lower}
    g_\tau(k)\geq c\tau^{-1/2}
    \quad\text{whenever}\quad
    \tau\geq C(1+k^2).
\end{equation}
Moreover, for every fixed $b>0$, there are constants $c_b,C_b>0$ such that
\begin{equation}\label{eq:one-dimensional-upper}
    g_\tau(k)
    \leq C_b\tau^{-1/2}
    \exp\left(-c_b\frac{k^2}{\tau}\right)
\end{equation}
whenever $\tau>0$ and $\tau\geq b|k|$.
\end{lemma}

\begin{proof}
For the lower estimate, choose $\tau_0\geq16$ and then choose
$C_0\geq\tau_0$ sufficiently large. Assume that
\[
    \tau\geq C_0(1+k^2).
\]
Let $\mathcal M_\tau$ consist of the integers $m$ satisfying
\[
    |m-\tau|\leq\sqrt\tau,
    \qquad m\equiv k\pmod 2.
\]
The interval $[\tau-\sqrt\tau,\tau+\sqrt\tau]$ contains at least
$c\sqrt\tau$ integers of either prescribed parity. Hence
\[
    |\mathcal M_\tau|\geq c\sqrt\tau.
\]
For every $m\in\mathcal M_\tau$, we have
\[
    \frac{\tau}{2}\leq m\leq2\tau,
    \qquad
    |k|\leq\sqrt{\frac{\tau}{C_0}}\leq\sqrt m.
\]
In particular, $m\geq|k|$, so every $m\in\mathcal M_\tau$ is an
admissible index in \eqref{eq:one-dimensional-kernel}. Stirling's two-sided
inequalities give, uniformly for $m\in\mathcal M_\tau$,
\[
    \ee^{-\tau}\frac{\tau^m}{m!}
    \geq c\tau^{-1/2}
    \exp\left(-C\frac{(m-\tau)^2}{\tau}\right)
    \geq c\tau^{-1/2}.
\]
The same inequalities, together with $m\equiv k\pmod 2$, yield
\[
    2^{-m}\binom{m}{(m+k)/2}
    \geq cm^{-1/2}
    \exp\left(-C\frac{k^2}{m}\right)
    \geq c\tau^{-1/2},
\]
because $k^2/m\leq2/C_0$. Thus every term indexed by
$\mathcal M_\tau$ is bounded below by $c\tau^{-1}$. Summing these terms
gives
\[
    g_\tau(k)\geq c|\mathcal M_\tau|\tau^{-1}
    \geq c\tau^{-1/2},
\]
which proves \eqref{eq:one-dimensional-lower}.

For the upper estimate, assume by symmetry that $k\geq0$ and set
\[
    \lambda:=\operatorname{arsinh}(k/\tau).
\]
Exponential tilting of the law of the walk gives
\begin{equation}\label{eq:tilted-point-mass}
    g_\tau(k)
    =\exp\bigl(-\lambda k+\tau(\cosh\lambda-1)\bigr)
    g_{\tau,\lambda}(k),
\end{equation}
where $g_{\tau,\lambda}$ is the point-mass function under the tilted
law. Its characteristic function has modulus
\[
    \exp\bigl(\tau\cosh\lambda(\cos\xi-1)\bigr).
\]
Because $\tau\geq bk$, the parameter $\lambda$ remains in a compact
interval depending only on $b$. Fourier inversion and
\[
    1-\cos\xi\geq c\xi^2,
    \qquad |\xi|\leq\pi,
\]
therefore give
\[
    g_{\tau,\lambda}(k)
    \leq\frac{1}{2\pi}
    \int_{-\pi}^{\pi}
    \exp\bigl(-c_b\tau\xi^2\bigr)\,d\xi
    \leq C_b\tau^{-1/2}.
\]
Writing $v=k/\tau$, the exponential factor in
\eqref{eq:tilted-point-mass} equals
\[
    \exp\bigl(-\tau I(v)\bigr),
    \qquad
    I(v):=v\operatorname{arsinh}v-\sqrt{1+v^2}+1.
\]
On the compact interval $0\leq v\leq1/b$, one has
\[
    I(v)\geq c_bv^2,
\]
because $I(0)=I'(0)=0$ and
\[
    I''(v)=(1+v^2)^{-1/2}>0.
\]
Substitution into \eqref{eq:tilted-point-mass} proves
\eqref{eq:one-dimensional-upper}.
\end{proof}

By Poisson thinning, the $2d$ oriented jump counts of the $d$-dimensional walk are independent. Therefore
\begin{equation}\label{eq:product-kernel}
    h_r(z)=\prod_{j=1}^d g_{r/d}(z_j),
    \qquad z=(z_1,\ldots,z_d)\in\Z^d.
\end{equation}

\subsection{The L\'evy kernel}

The fractional discrete-Laplacian formulas in \cite{CiaurriEtAl2018} have the
following multidimensional analogue. We include the proof because the later
arguments require a two-sided bound uniform in the lattice point.

\begin{proposition}\label{prop:jump-kernel}
There exists a symmetric function
\[
    K_s:\Z^d\setminus\{0\}\to(0,\infty)
\]
such that, for every bounded function $\phi$,
\begin{equation}\label{eq:jump-representation}
    A^s\phi(x)
    =\sum_{z\neq0}K_s(z)
    \bigl(\phi(x)-\phi(x+z)\bigr).
\end{equation}
The series in \eqref{eq:jump-representation} is absolutely convergent for
all $x$, uniformly in $x$ when $\phi$ ranges over a bounded subset of
$\ell^\infty(\Z^d)$.
Moreover, there exist constants $c_1,C_1>0$ such that
\begin{equation}\label{eq:jump-kernel-bounds}
    c_1(1+|z|)^{-d-2s}
    \leq K_s(z)
    \leq C_1(1+|z|)^{-d-2s},
    \qquad z\neq0.
\end{equation}
In particular,
\begin{equation}\label{eq:kappa-definition}
    \kappa_s:=\sum_{z\neq0}K_s(z)<\infty.
\end{equation}
\end{proposition}

\begin{proof}
Define
\begin{equation}\label{eq:jump-kernel-integral}
    K_s(z)
    =c_s\int_0^\infty h_r(z)\frac{dr}{r^{1+s}},
    \qquad z\neq0.
\end{equation}
We first prove \eqref{eq:jump-kernel-bounds}. Once this estimate is known,
$K_s\in\ell^1(\Z^d)$, so Fubini's theorem justifies substituting the Markov
kernel $h_r$ into \eqref{eq:balakrishnan} and gives
\eqref{eq:jump-representation}.

Set $n:=|z|_1\geq1$. From Lemma~\ref{lem:one-dimensional-walk} and \eqref{eq:product-kernel},
\[
    h_r(z)\geq cr^{-d/2}
\]
whenever
\[
    C(1+n^2)\leq r\leq2C(1+n^2).
\]
Here $C$ is chosen so that $r/d\geq C_0(1+z_j^2)$ for every coordinate $j$, where $C_0$ is the constant in \eqref{eq:one-dimensional-lower}.
Consequently,
\begin{equation}\label{eq:jump-kernel-lower-proof}
    K_s(z)\geq c\int_{C(1+n^2)}^{2C(1+n^2)} r^{-1-s-d/2}\,dr\geq c(1+n)^{-d-2s}.
\end{equation}

For the upper bound, choose $\alpha>0$ so small that
$\ee\alpha<1$. If $0<r\leq\alpha n$, reaching $z$ requires at
least $n$ jumps. The total number of jumps is Poisson with mean $r$,
and therefore
\begin{equation}\label{eq:small-time-poisson}
    h_r(z)
    \leq\mathbb P\{\operatorname{Poisson}(r)\geq n\}
    \leq\left(\frac{\ee r}{n}\right)^n.
\end{equation}
It follows that
\begin{align}
    \int_0^{\alpha n}h_r(z)\frac{dr}{r^{1+s}}
    &\leq
    \left(\frac{\ee}{n}\right)^n
    \int_0^{\alpha n}r^{n-1-s}\,dr \notag\\
    &=\alpha^{-s}(\ee\alpha)^n\frac{n^{-s}}{n-s} \notag\\
    &\leq C(\ee\alpha)^n n^{-1-s}
    \leq C(1+n)^{-d-2s}.
    \label{eq:jump-small-time}
\end{align}
Here $n-s\geq(1-s)n$, and the exponential factor $(\ee\alpha)^n$
dominates every fixed negative power of $n$.
For $r\geq\alpha n$, Lemma~\ref{lem:one-dimensional-walk} and \eqref{eq:product-kernel} imply
\[
    h_r(z)
    \leq Cr^{-d/2}
    \exp\left(-c\frac{n^2}{r}\right).
\]
Here $r/d\geq(\alpha/d)|z_j|$ for every $j$, so
\eqref{eq:one-dimensional-upper} applies to each factor in
\eqref{eq:product-kernel}; equivalence of $|z|$ and $|z|_1$ gives the
displayed exponent. Hence, with the change of variables $\rho=n^2/r$,
\begin{align}
    \int_{\alpha n}^\infty h_r(z)\frac{dr}{r^{1+s}}
    &\leq C\int_{\alpha n}^\infty
    r^{-1-s-d/2}\exp\left(-c\frac{n^2}{r}\right)\,dr \notag\\
    &=Cn^{-d-2s}\int_0^{n/\alpha}
    \rho^{s+d/2-1}\ee^{-c\rho}\,d\rho
    \leq Cn^{-d-2s}.
    \label{eq:jump-large-time}
\end{align}
Equivalence of the Euclidean and $\ell^1$ norms proves
\eqref{eq:jump-kernel-bounds}.

Since $d+2s>d$, the upper bound in \eqref{eq:jump-kernel-bounds} gives \eqref{eq:kappa-definition}. In particular, for every bounded $\phi$,
\[
    \sum_{z\neq0}K_s(z)
    |\phi(x)-\phi(x+z)|
    \leq2\kappa_s\|\phi\|_{\ell^\infty},
\]
uniformly in $x$. The series in \eqref{eq:jump-representation} is therefore absolutely convergent, and Fubini's theorem applies in the Balakrishnan formula.
More explicitly, Tonelli's theorem and \eqref{eq:jump-kernel-integral} give
\[
    c_s\int_0^\infty\sum_{z\neq0}h_r(z)
    |\phi(x)-\phi(x+z)|\frac{dr}{r^{1+s}}
    =\sum_{z\neq0}K_s(z)|\phi(x)-\phi(x+z)|,
\]
and the right-hand side is bounded by
$2\kappa_s\|\phi\|_{\ell^\infty}$.
\end{proof}

\begin{remark}\label{rem:bounded-generator}
Proposition~\ref{prop:jump-kernel} implies
\begin{equation}\label{eq:operator-bound}
    \|A^s\phi\|_{\ell^\infty}
    \leq2\kappa_s\|\phi\|_{\ell^\infty}.
\end{equation}
Unlike the continuous fractional Laplacian, $A^s$ is therefore bounded on $\ellinf$. This discrete feature will be used in the compactness argument below.
\end{remark}

\subsection{The compound-Poisson heat kernel}

Let $G_t(z)$ denote the convolution kernel of
\[
    S(t):=\ee^{-tA^s}.
\]

\begin{proposition}\label{prop:heat-kernel}
There exist constants $c_2,C_2>0$ and an integer $m\in\N$ satisfying $m\geq d+2s+1$ such that
\begin{equation}\label{eq:fixed-time-heat-lower}
    G_1(z)\geq c_2(1+|z|)^{-d-2s},
    \qquad z\in\Z^d,
\end{equation}
and
\begin{equation}\label{eq:heat-upper}
    G_t(z)
    \leq C_2(t+t^m)(1+|z|)^{-d-2s},
    \qquad t>0,\quad z\neq0.
\end{equation}
\end{proposition}
Fix, for the rest of the paper, one admissible $m\in\N$ supplied by
Proposition~\ref{prop:heat-kernel}.

\begin{proof}
Define
\[
    j_s(z):=
    \begin{cases}
        \kappa_s^{-1}K_s(z),&z\neq0,\\
        0,&z=0.
    \end{cases}
\]
Then $j_s$ is a probability mass function and
\[
    -A^s\phi(x)
    =\kappa_s\sum_zj_s(z)\phi(x+z)-\kappa_s\phi(x).
\]
Therefore
\begin{equation}\label{eq:compound-poisson}
    G_t
    =\ee^{-\kappa_st}
    \sum_{\ell=0}^\infty
    \frac{(\kappa_st)^\ell}{\ell!}\,j_s^{*\ell}.
\end{equation}
For $z\neq0$, the one-jump contribution gives
\[
    G_1(z)\geq\ee^{-\kappa_s}K_s(z),
\]
while $G_1(0)\geq\ee^{-\kappa_s}$ follows from the zero-jump term. Proposition~\ref{prop:jump-kernel} proves \eqref{eq:fixed-time-heat-lower}.

We claim that, for $\ell\geq1$ and $z\neq0$,
\begin{equation}\label{eq:convolution-bound}
    j_s^{*\ell}(z)
    \leq C\ell^{d+2s+1}(1+|z|)^{-d-2s}.
\end{equation}
In each representation
\[
    z=z_1+\cdots+z_\ell,
\]
at least one increment satisfies $|z_i|\geq|z|/\ell$. For fixed $i$, fixing the other $\ell-1$ increments uniquely determines $z_i$. Hence
\begin{align*}
    j_s^{*\ell}(z)
    &\leq
    \sum_{i=1}^{\ell}
    \sum_{\substack{z_1+\cdots+z_\ell=z\\
    |z_i|\geq |z|/\ell}}
    \prod_{k=1}^{\ell}j_s(z_k)\\
    &\leq C\ell^{d+2s+1}(1+|z|)^{-d-2s},
\end{align*}
where we used
\[
    j_s(z_i)
    \leq C(1+|z|/\ell)^{-d-2s}
    \leq C\ell^{d+2s}(1+|z|)^{-d-2s}
\]
and summed the remaining $\ell-1$ jump variables, whose total mass is one. This proves \eqref{eq:convolution-bound}.

Using \eqref{eq:compound-poisson} and \eqref{eq:convolution-bound}, we obtain
\[
    G_t(z)\leq C(1+|z|)^{-d-2s}\ee^{-\kappa_st}\sum_{\ell=1}^\infty\frac{(\kappa_st)^\ell}{\ell!}\ell^{d+2s+1}.
\]
Let $N_\lambda$ be a Poisson random variable with mean $\lambda$. Choose an integer $m\geq d+2s+1$. Since $\ell^{d+2s+1}\leq\ell^m$ for $\ell\geq1$ and
\[
    \mathbb E[N_\lambda^m]
    =\sum_{k=1}^m S(m,k)\lambda^k
    \leq C_m(\lambda+\lambda^m),
\]
where $S(m,k)$ are the Stirling numbers of the second kind, the Poisson moment in the preceding display satisfies
\[
    \ee^{-\kappa_st}
    \sum_{\ell=1}^\infty
    \frac{(\kappa_st)^\ell}{\ell!}\ell^{d+2s+1}
    \leq C(t+t^m),
\]
which proves \eqref{eq:heat-upper}.
\end{proof}

\section{Preliminary estimates for the spreading analysis}\label{sec:barrier}

Throughout Sections~\ref{sec:barrier}--\ref{sec:upper}, assumptions
\emph{(F)} and \emph{(I)} are in force, and $u$ denotes the unique global
solution corresponding to $u_0$. This section collects the comparison and
profile estimates used below.

\begin{lemma}\label{lem:comparison}
Let $g:\R\to\R$ be locally Lipschitz, let $T>0$, and let
$v,w\in C^1([0,T];\ellinf)$ be bounded. Suppose
${\partial_tv+A^sv\leq g(v)}$ and
${\partial_tw+A^sw\geq g(w)}$
on $(0,T]\times\Z^d$, and $v(0,\cdot)\leq w(0,\cdot)$. Then
$v(t,x)\leq w(t,x)$ for $0\leq t\leq T$ and $x\in\Z^d$.
\end{lemma}

\begin{proof}
Choose a compact interval $J\subset\R$ containing the ranges of $v$ and
$w$, and choose $L>0$ at least as large as a Lipschitz constant of $g$ on
$J$. Then $H(r):=g(r)+Lr$ is nondecreasing on $J$. Let $M$ be a
Lipschitz constant of $H$ on $J$. The semigroup
$T_L(t):=\ee^{-t(A^s+LI)}=\ee^{-Lt}S(t)$
is positive and satisfies $\|T_L(t)\|_{\ell^\infty\to\ell^\infty}\leq1$; this follows from the compound-Poisson representation \eqref{eq:compound-poisson} and $S(t)\one=\one$.

Integrating the two differential inequalities by the variation-of-constants formula gives
\begin{align*}
    v(t)&\leq T_L(t)v(0)
    +\int_0^tT_L(t-r)H(v(r))\,dr,\\
    w(t)&\geq T_L(t)w(0)
    +\int_0^tT_L(t-r)H(w(r))\,dr.
\end{align*}
Set $D(t):=\|(v(t)-w(t))_+\|_{\ell^\infty}$. For every bounded sequence
$\psi$, positivity gives $(T_L(t)\psi)_+\leq T_L(t)\psi_+$. Using this,
the initial ordering, and $H(v)-H(w)\leq M(v-w)_+$, we obtain
$D(t)\leq M\int_0^tD(r)\,dr$.
Gronwall's inequality yields $D(t)=0$ on $[0,T]$, which proves the assertion.
\end{proof}

\begin{lemma}\label{lem:polynomial-seeding}
There exists $c_0>0$ such that
\begin{equation}\label{eq:polynomial-seed}
    u(1,x)
    \geq c_0(1+|x|^2)^{-(d+2s)/2},
    \qquad x\in\Z^d.
\end{equation}
\end{lemma}

\begin{proof}
Since $f(u)\geq0$ for $0\leq u\leq1$, Duhamel's formula gives
\[
    u(t)
    =S(t)u_0+
    \int_0^tS(t-r)f(u(r))\,dr
    \geq S(t)u_0.
\]
Choose $x_*\in\Z^d$ such that $u_0(x_*)>0$. Proposition~\ref{prop:heat-kernel} yields
\[
    u(1,x)\geq u_0(x_*)G_1(x-x_*)\geq c(1+|x-x_*|)^{-d-2s}\geq c_0(1+|x|^2)^{-(d+2s)/2}.
\]
\end{proof}

For $R\geq1$, introduce the algebraic profile
\begin{equation}\label{eq:profile-definition}
    \Phi_R(x)
    :=\left(1+\frac{|x|^2}{R^2}\right)^{-q/2}.
\end{equation}

\begin{lemma}\label{lem:weighted-profile}
There exists a constant $C_3>0$, depending only on $d$ and $s$, such that
\begin{equation}\label{eq:weighted-profile-estimate}
    |A^s\Phi_R(x)|
    \leq C_3R^{-2s}\Phi_R(x)
\end{equation}
for every $R\geq1$ and $x\in\Z^d$.
\end{lemma}

\begin{proof}
Extend $\Phi_R$ smoothly to $\R^d$. Symmetry of $K_s$ gives
\begin{equation}\label{eq:symmetrized-operator}
    A^s\Phi_R(x)
    =\frac12\sum_{z\neq0}K_s(z)
    \bigl(2\Phi_R(x)-\Phi_R(x+z)-\Phi_R(x-z)\bigr).
\end{equation}
Set $L:=R+|x|$. Direct differentiation shows that, whenever $|z|\leq L/4$,
\begin{equation}\label{eq:relative-hessian}
    \sup_{|\vartheta|\leq1}
    |D^2\Phi_R(x+\vartheta z)|
    \leq CL^{-2}\Phi_R(x).
\end{equation}
If $|x|\leq2R$, then $\Phi_R(x)$ is bounded below and $|D^2\Phi_R|\leq CR^{-2}$. If $|x|>2R$, the segment $x+\vartheta z$ stays at distance comparable to $|x|$, and differentiation of \eqref{eq:profile-definition} gives \eqref{eq:relative-hessian}.

Taylor's formula, Proposition~\ref{prop:jump-kernel}, and $s<1$ now imply
\begin{align}
    \sum_{0<|z|\leq L/4}K_s(z)|2\Phi_R(x)-\Phi_R(x+z)-\Phi_R(x-z)|
    &\leq CL^{-2}\Phi_R(x)\sum_{0<|z|\leq L/4}|z|^2K_s(z)\notag\\
    &\leq CL^{-2s}\Phi_R(x)\leq CR^{-2s}\Phi_R(x).
    \label{eq:profile-near-field}
\end{align}

For the far field, Proposition~\ref{prop:jump-kernel} gives
\begin{equation}\label{eq:kernel-tail}
    \sum_{|z|>L/4}K_s(z)\leq CL^{-2s}.
\end{equation}
Moreover, since $q>d$,
\begin{equation}\label{eq:profile-mass}
    \sum_{y\in\Z^d}\Phi_R(y)\leq CR^d.
\end{equation}
These estimates follow by decomposing the lattice into radial shells, or by
comparison with the corresponding radial integrals. If $|z|>L/4$, then
$K_s(z)\leq CL^{-q}$. Consequently,
\begin{align}
    \sum_{|z|>L/4}K_s(z)\Phi_R(x+z)
    &\leq CL^{-q}\sum_{y\in\Z^d}\Phi_R(y)
      \leq CR^dL^{-q},
      \label{eq:profile-far-convolution-plus}\\
    \sum_{|z|>L/4}K_s(z)\Phi_R(x-z)
    &\leq CR^dL^{-q}.
      \label{eq:profile-far-convolution-minus}
\end{align}
Since
\[
    \Phi_R(x)
    =R^q(R^2+|x|^2)^{-q/2}
    \geq R^qL^{-q}
\]
and $q=d+2s$, we have
\begin{equation}\label{eq:critical-exponent-identity}
    R^dL^{-q}\leq R^{-2s}\Phi_R(x).
\end{equation}
Therefore
\begin{align*}
 &\sum_{|z|>L/4}K_s(z)
 \left|2\Phi_R(x)-\Phi_R(x+z)-\Phi_R(x-z)\right|\\
 &\qquad\leq
 2\Phi_R(x)\sum_{|z|>L/4}K_s(z)
 +\sum_{|z|>L/4}K_s(z)\Phi_R(x+z)
 +\sum_{|z|>L/4}K_s(z)\Phi_R(x-z)\\
 &\qquad\leq CL^{-2s}\Phi_R(x)+CR^dL^{-q}
 \leq CR^{-2s}\Phi_R(x).
\end{align*}
Combining this estimate with \eqref{eq:symmetrized-operator} and
\eqref{eq:profile-near-field} proves
\eqref{eq:weighted-profile-estimate}.
\end{proof}

\begin{remark}\label{rem:no-cutoff}
The profile $\Phi_R$ is kept positive on the whole lattice. A pointwise truncation such as $\min\{\eta,C|x|^{-q}\}$ does not automatically preserve the subsolution inequality, because values outside the truncation region contribute to $A^s$ at every point.
\end{remark}

\section{Sharp lower spreading estimate}\label{sec:lower}

An expanding algebraic subsolution gives a positive lower bound on strictly
subcritical exponential balls. The compactness--Liouville argument below
then strengthens this bound to convergence to one.

\subsection{An expanding algebraic subsolution}

\begin{proposition}\label{prop:expanding-barrier}
For every $\gamma\in(0,a)$, there exist $R\geq1$ and $\eta\in(0,1)$ such that
\begin{equation}\label{eq:expanding-lower-barrier}
    u(1+t,x)
    \geq\eta
    \left(
    1+\frac{|x|^2}{R^2\exp(2\gamma t/q)}
    \right)^{-q/2}
\end{equation}
for all $t\geq0$ and $x\in\Z^d$.
\end{proposition}

\begin{proof}
Fix $\gamma\in(0,a)$ and put $\sigma:=(a-\gamma)/2$.
Since $f'(0)=a$, there exists $\eta_*>0$ such that
\begin{equation}\label{eq:reaction-lower-linearization}
    f(r)\geq(a-\sigma)r,
    \qquad 0\leq r\leq\eta_*.
\end{equation}
Choose $R\geq1$ sufficiently large that
\begin{equation}\label{eq:R-choice}
    C_3R^{-2s}
    \leq a-\sigma-\gamma
    =\frac{a-\gamma}{2},
\end{equation}
where $C_3$ is the constant in Lemma~\ref{lem:weighted-profile}. After $R$ has been fixed, choose
\begin{equation}\label{eq:eta-choice}
    0<\eta\leq
    \min\{\eta_*,c_0R^{-q}\},
\end{equation}
where $c_0$ is the constant in Lemma~\ref{lem:polynomial-seeding}.

Define $R(t):=R\exp(\gamma t/q)$ and
\begin{equation}\label{eq:subsolution-definition}
    \underline u(t,x):=\eta\Phi_{R(t)}(x).
\end{equation}
Writing $\zeta=|x|^2/R(t)^2$, we calculate
\begin{equation}\label{eq:subsolution-time-derivative}
    \partial_t\underline u(t,x)
    =\gamma\underline u(t,x)\frac{\zeta}{1+\zeta}
    \leq\gamma\underline u(t,x).
\end{equation}
Lemma~\ref{lem:weighted-profile} gives
\begin{equation}\label{eq:subsolution-diffusion}
    A^s\underline u(t,x)
    \leq C_3R(t)^{-2s}\underline u(t,x)
    \leq C_3R^{-2s}\underline u(t,x).
\end{equation}
Since $0<\Phi_{R(t)}\leq1$, the choice of $\eta$ gives
$0\leq\underline u(t,x)\leq\eta\leq\eta_*$ for $t\geq0$ and $x\in\Z^d$.
Hence \eqref{eq:reaction-lower-linearization} is applicable at every point. Combining it with \eqref{eq:R-choice}, \eqref{eq:subsolution-time-derivative}, and \eqref{eq:subsolution-diffusion}, we obtain
\[
    \partial_t\underline u+A^s\underline u\leq(\gamma+C_3R^{-2s})\underline u\leq(a-\sigma)\underline u\leq f(\underline u).
\]
Thus $\underline u$ is a global subsolution.

Since $\Phi_R(x)\leq R^q(1+|x|^2)^{-q/2}$,
the choices in \eqref{eq:eta-choice} and Lemma~\ref{lem:polynomial-seeding} imply
$\underline u(0,x)\leq c_0(1+|x|^2)^{-q/2}\leq u(1,x)$.
Lemma~\ref{lem:comparison} completes the proof.
\end{proof}

\begin{corollary}\label{cor:positive-plateau}
For every $b<a/q$, there exist $\eta_b>0$ and $T_b>0$ such that
\begin{equation}\label{eq:positive-plateau}
    u(t,x)\geq\eta_b
\end{equation}
whenever $t\geq T_b$ and $|x|\leq\ee^{bt}$.
\end{corollary}

\begin{proof}
Choose $\gamma$ so that $qb<\gamma<a$.
For $|x|\leq\ee^{bt}$, Proposition~\ref{prop:expanding-barrier} gives
\[
    \frac{|x|}{R\exp(\gamma(t-1)/q)}
    \leq C\exp\left(-\left(\frac{\gamma}{q}-b\right)t\right)
    \to0.
\]
The right-hand side of \eqref{eq:expanding-lower-barrier} therefore converges uniformly to $\eta$. Taking $\eta_b=\eta/2$ proves the result.
\end{proof}

\subsection{Compactness and rigidity of positive entire solutions}

\begin{lemma}\label{lem:liouville}
Let $U\in C^1(\R;\ellinf)$ be an entire solution of
$\partial_tU+A^sU=f(U)$ on $\R\times\Z^d$.
If there exists $\eta_0>0$ such that
\begin{equation}\label{eq:entire-bounds}
    \eta_0\leq U(t,x)\leq1
    \quad\text{for all }(t,x)\in\R\times\Z^d,
\end{equation}
then $U\equiv1$.
\end{lemma}

\begin{proof}
Let $v$ solve
\begin{equation}\label{eq:scalar-ode}
    v'(t)=f(v(t)),
    \qquad v(0)=\eta_0.
\end{equation}
By \eqref{eq:F-hypotheses}, $v$ is nondecreasing (strictly so unless
$\eta_0=1$), and
\begin{equation}\label{eq:scalar-limit}
    \lim_{t\to\infty}v(t)=1.
\end{equation}
Indeed, $v(t)\in[\eta_0,1]$ has a limit $\ell\in[\eta_0,1]$. If
$\ell<1$, continuity and positivity of $f$ on $(0,1)$ give
$v'(t)\geq c>0$ for all sufficiently large $t$, a contradiction. Hence
$\ell=1$.
For any $T>0$, the spatially constant function
\[
    \underline U_T(t,x):=v(t+T),
    \qquad -T\leq t\leq0,
\]
is a solution of the lattice equation and satisfies
\[
    \underline U_T(-T,x)=\eta_0\leq U(-T,x).
\]
The comparison principle yields
\[
    U(0,x)\geq v(T).
\]
Letting $T\to\infty$ and using \eqref{eq:scalar-limit}, we obtain $U(0,x)\geq1$. In view of \eqref{eq:entire-bounds}, $U(0,x)=1$. Applying the same argument after an arbitrary time translation proves $U\equiv1$.
\end{proof}

\begin{proposition}\label{prop:plateau-upgrade}
Let $b_0>\max\{b,0\}$. Suppose that there are $\eta_0,T_0>0$ such that
\begin{equation}\label{eq:faster-plateau}
    u(t,x)\geq\eta_0
    \quad\text{for }t\geq T_0,\quad |x|\leq\ee^{b_0t}.
\end{equation}
Then
\begin{equation}\label{eq:upgrade-conclusion}
    \lim_{t\to\infty}
    \inf_{\substack{x\in\Z^d\\|x|\leq\ee^{bt}}}u(t,x)=1.
\end{equation}
\end{proposition}

\begin{proof}
Suppose that \eqref{eq:upgrade-conclusion} is false. Since $u\leq1$, there exist $\delta>0$, a sequence $t_n\to\infty$, and sites $x_n\in\Z^d$ such that
\begin{equation}\label{eq:bad-sequence}
    |x_n|\leq\ee^{bt_n},
    \qquad
    u(t_n,x_n)\leq1-\delta.
\end{equation}
Define the translated solutions
\begin{equation}\label{eq:translated-solutions}
    U_n(\tau,z):=u(t_n+\tau,x_n+z),
    \qquad \tau>-t_n,\quad z\in\Z^d.
\end{equation}
By \eqref{eq:operator-bound} and $0\leq U_n\leq1$,
\begin{equation}\label{eq:uniform-time-lipschitz}
    |\partial_\tau U_n(\tau,z)|
    \leq2\kappa_s+\max_{0\leq r\leq1}|f(r)|.
\end{equation}
Fix a compact interval $I\subset\R$ and a finite set $E\subset\Z^d$. For
all sufficiently large $n$, $I\subset(-t_n,\infty)$, and
\eqref{eq:uniform-time-lipschitz} gives a common Lipschitz constant on $I$
for every $z\in E$. Since $0\leq U_n\leq1$, a diagonal
Arzel\`a--Ascoli argument yields a subsequence and a function
\[
    U:\R\times\Z^d\to[0,1]
\]
such that
\begin{equation}\label{eq:translated-convergence}
    U_n(\tau,z)\to U(\tau,z)
\end{equation}
locally uniformly in $\tau$, for every fixed $z$. Moreover, if $\tau_1,\tau_2\in\R$, then for every $z\in\Z^d$,
\[
    |U(\tau_1,z)-U(\tau_2,z)|
    \leq C|\tau_1-\tau_2|,
\]
where $C$ is the constant in \eqref{eq:uniform-time-lipschitz}. Taking the
supremum over $z$ shows that $\tau\mapsto U(\tau,\cdot)$ is globally
Lipschitz as an $\ell^\infty$-valued map.

It remains to pass to the limit in the nonlocal term. For fixed $z$ and
$M>0$, write
\begin{align*}
    A^sU_n(\tau,z)
    &=%
    \sum_{0<|y|\leq M}K_s(y)
    \bigl(U_n(\tau,z)-U_n(\tau,z+y)\bigr)\\
    &\quad+
    \sum_{|y|>M}K_s(y)
    \bigl(U_n(\tau,z)-U_n(\tau,z+y)\bigr).
\end{align*}
The first sum converges locally uniformly in $\tau$. The second sum is bounded, uniformly in $n$ and $\tau$, by
\[
    2\sum_{|y|>M}K_s(y),
\]
which tends to zero as $M\to\infty$. Consequently,
\[
    A^sU_n(\cdot,z)\to A^sU(\cdot,z)
\]
locally uniformly in time for each fixed $z$. For $\sigma<\tau$, the translated equation has the integral form
\[
    U_n(\tau,z)-U_n(\sigma,z)
    =\int_\sigma^\tau
    \bigl[-A^sU_n(r,z)+f(U_n(r,z))\bigr],dr.
\]
Passing to the limit in the integral identity gives, for every $z\in\Z^d$,
\[
    U(\tau,z)-U(\sigma,z)
    =\int_\sigma^\tau
    \bigl[-A^sU(r,z)+f(U(r,z))\bigr]dr.
\]
Set
\[
    F(r):=-A^sU(r)+f(U(r)).
\]
Since $U$ is continuous in the $\ell^\infty$ norm, $A^s$ is bounded, and
$f$ is Lipschitz on $[0,1]$, one has $F\in C(\R;\ell^\infty)$. Thus the
Bochner integral $\int_\sigma^\tau F(r)\,dr$ is well defined in
$\ell^\infty$. The coordinate functionals on $\ell^\infty$ separate points,
so the coordinatewise identities above yield
\[
    U(\tau)-U(\sigma)=\int_\sigma^\tau F(r)\,dr
    \quad\text{in }\ell^\infty.
\]
It follows that $U\in C^1(\R;\ellinf)$ and
\[
    \partial_\tau U+A^sU=f(U)
    \quad\text{in }\ell^\infty.
\]
Hence $U$ is an entire classical solution in the sense of Definition~\ref{def:classical-solution}.

We next show that the limit is bounded away from zero. Fix $T,N>0$. If
$|\tau|\leq T$ and $|z|\leq N$, then
\[
    |x_n+z|\leq\ee^{bt_n}+N.
\]
Because $b_0>b$,
\[
    \frac{\ee^{bt_n}+N}{\ee^{b_0(t_n-T)}}\to0.
\]
Also $t_n-T\geq T_0$ for all sufficiently large $n$. Hence, uniformly for $|\tau|\leq T$ and $|z|\leq N$,
\[
    |x_n+z|\leq\ee^{b_0(t_n+\tau)}.
\]
The persistent bound \eqref{eq:faster-plateau} gives
\[
    U_n(\tau,z)\geq\eta_0.
\]
Passing to the limit gives $U(\tau,z)\geq\eta_0$ on every fixed cylinder $[-T,T]\times\{z:|z|\leq N\}$. Since $T$ and $N$ are arbitrary,
\[
    \eta_0\leq U(\tau,z)\leq1
    \quad\text{on }\R\times\Z^d.
\]
Lemma~\ref{lem:liouville} implies $U\equiv1$. This contradicts \eqref{eq:bad-sequence}, since
\[
    U(0,0)
    =\lim_{n\to\infty}u(t_n,x_n)
    \leq1-\delta.
\]
\end{proof}

\medskip

\begin{theorem}\label{thm:lower-spreading}
For every $\varepsilon>0$,
\begin{equation}\label{eq:lower-spreading}
    \lim_{t\to\infty}
    \inf_{\substack{x\in\Z^d\\ |x|\leq
    \exp((a-\varepsilon)t/q)}}
    u(t,x)=1.
\end{equation}
\end{theorem}

\begin{proof}
First suppose that $0<\varepsilon<a$. Set
\[
    b:=\frac{a-\varepsilon}{q}
\]
and choose
\[
    b_0:=\frac{a-\varepsilon/2}{q}.
\]
Then $b_0>b$ and $b_0<a/q$. Corollary~\ref{cor:positive-plateau} gives a
persistent positive lower bound on $|x|\leq\ee^{b_0t}$, and
Proposition~\ref{prop:plateau-upgrade} yields \eqref{eq:lower-spreading}.

If $\varepsilon\geq a$, then
\[
    \exp\left(\frac{a-\varepsilon}{q}t\right)
    \leq\exp\left(\frac{a}{2q}t\right).
\]
The conclusion follows from the already proved case $\varepsilon=a/2$, because the requested ball is contained in a larger ball on which convergence to one is uniform.
\end{proof}

\section{Upper spreading estimate and the logarithmic level-set law}\label{sec:upper}

The KPP linearization and Proposition~\ref{prop:heat-kernel} give the upper
estimate. Combined with Theorem~\ref{thm:lower-spreading}, it completes the
proof of Theorem~\ref{thm:main}.

\begin{proposition}\label{prop:upper-spreading}
For every $\varepsilon>0$,
\begin{equation}\label{eq:upper-spreading}
    \lim_{t\to\infty}
    \sup_{\substack{x\in\Z^d\\ |x|\geq
    \exp((a+\varepsilon)t/q)}}
    u(t,x)=0.
\end{equation}
\end{proposition}

\begin{proof}
The KPP inequality \eqref{eq:KPP-inequality} gives
$\partial_tu+A^su=f(u)\leq au$. Set $v(t):=\ee^{-at}u(t)$. Then
$\partial_tv+A^sv=\ee^{-at}(f(u)-au)\leq0$.
The variation-of-constants formula and positivity of $S(t)$ therefore give
\begin{equation}\label{eq:linear-comparison}
    v(t)\leq S(t)u_0,
    \qquad\text{and hence}\qquad
    u(t)\leq\ee^{at}S(t)u_0.
\end{equation}
Choose $R_0>0$ such that $\supp u_0\subset B_{R_0}$.
If $|x|\geq2R_0+1$, then Proposition~\ref{prop:heat-kernel} gives
\begin{equation}\label{eq:linear-tail-bound}
    (S(t)u_0)(x)=\sum_yG_t(x-y)u_0(y)\leq C\|u_0\|_{\ell^1}(t+t^m)(1+|x|)^{-q}.
\end{equation}
Indeed, $|x-y|\geq|x|/2$ on the support of $u_0$, and
$\|u_0\|_{\ell^1}<\infty$ by assumption \emph{(I)}.
Combining \eqref{eq:linear-comparison} and \eqref{eq:linear-tail-bound}, we obtain
\begin{equation}\label{eq:nonlinear-tail-bound}
    u(t,x)
    \leq C\ee^{at}(t+t^m)(1+|x|)^{-q}.
\end{equation}
For $|x|\geq\exp((a+\varepsilon)t/q)$, the right-hand side of
\eqref{eq:nonlinear-tail-bound} is bounded by
$C(t+t^m)\ee^{-\varepsilon t}$,
which converges to zero. This proves \eqref{eq:upper-spreading}.
\end{proof}

\begin{proof}[Proof of Theorem~\ref{thm:main}]
The lower assertion \eqref{eq:main-lower} is Theorem~\ref{thm:lower-spreading}, and the upper assertion \eqref{eq:main-upper} is Proposition~\ref{prop:upper-spreading}.

Fix $0<\theta<1$ and $\varepsilon\in(0,a)$. The lower estimate and a
lattice point on a coordinate axis give
$\liminf_{t\to\infty}t^{-1}\log R_\theta(t)\geq(a-\varepsilon)/q$.
The upper estimate gives
$\limsup_{t\to\infty}t^{-1}\log R_\theta(t)\leq(a+\varepsilon)/q$.
Letting $\varepsilon\downarrow0$ and recalling $q=d+2s$ proves
$\lim_{t\to\infty}t^{-1}\log R_\theta(t)=f'(0)/(d+2s)$.
\end{proof}

\section{Critical-scale localization for logistic reactions}
\label{sec:logistic-localization}

The polynomial factor in \eqref{eq:nonlinear-tail-bound} prevents that
estimate from giving a bounded error at the critical scale. Under assumption
\emph{(L)}, a global algebraic profile with coupled amplitude and radius
removes this loss. The construction also applies to data satisfying
\eqref{eq:critical-tail-upper-bound}, with a different initialization. We do
not introduce a spatial cutoff.

For $R\geq1$, set
\begin{equation}\label{eq:logistic-profile}
    \Theta_R(x):=
    \frac{1}{1+(|x|/R)^q},
    \qquad x\in\Z^d.
\end{equation}

\begin{lemma}
\label{lem:logistic-weighted-profile}
There exists $C_4>0$, depending only on $d$ and $s$, such that
\begin{equation}\label{eq:logistic-weighted-estimate}
    |A^s\Theta_R(x)|
    \leq C_4R^{-2s}\Theta_R(x)
\end{equation}
for every $R\geq1$ and $x\in\Z^d$.
\end{lemma}

\begin{proof}
Choose
\begin{equation}\label{eq:holder-exponent-choice}
    2s<\alpha<\min\{q,2\}.
\end{equation}
Such a choice is possible because $q=d+2s>2s$ and $s<1$. Set
$L:=R+|x|$. For $0<|z|\leq L/4$,
\begin{equation}\label{eq:logistic-relative-difference}
    |2\Theta_R(x)-\Theta_R(x+z)-\Theta_R(x-z)|
    \leq
    C\Theta_R(x)\left(\frac{|z|}{L}\right)^\alpha.
\end{equation}
Set $F(y):=(1+|y|^q)^{-1}$, so that $\Theta_R(x)=F(x/R)$.
Since $q>1$ and $\alpha<\min\{q,2\}$, the function $F$ belongs to
$C^\alpha_{\rm loc}(\R^d)$, where for $\alpha>1$ this means
$C^{1,\alpha-1}_{\rm loc}$. Hence, on every fixed ball,
\[
    |2F(y)-F(y+h)-F(y-h)|\leq C|h|^\alpha.
\]
If $|x|\leq2R$ and $|z|\leq L/4$, then $x/R$ and
$(x\pm z)/R$ belong to a fixed ball, $L\asymp R$, and
$\Theta_R(x)\geq(1+2^q)^{-1}$. The preceding estimate therefore proves
\eqref{eq:logistic-relative-difference} in this case. If
$|x|>2R$, every point on either segment from $x$ to $x\pm z$ has distance
comparable to $|x|$. Direct differentiation away from the origin gives
\[
    |D^2\Theta_R(x+\vartheta z)|
    \leq CL^{-2}\Theta_R(x),
    \qquad |\vartheta|\leq1.
\]
Taylor's formula proves \eqref{eq:logistic-relative-difference}, because
$|z|/L\leq1/4$ and $\alpha<2$.

By symmetry of $K_s$, \eqref{eq:symmetrized-operator} holds with
$\Theta_R$ in place of $\Phi_R$. Proposition~\ref{prop:jump-kernel},
\eqref{eq:holder-exponent-choice}, and a decomposition into lattice shells
give
\[
    \sum_{0<|z|\leq L/4}|z|^\alpha K_s(z)
    \leq CL^{\alpha-2s}.
\]
Consequently, the contribution of $0<|z|\leq L/4$ is bounded by
\begin{equation}\label{eq:logistic-near-field}
    CL^{-2s}\Theta_R(x)
    \leq CR^{-2s}\Theta_R(x).
\end{equation}

For the complementary sum,
\[
    \sum_{|z|>L/4}K_s(z)\leq CL^{-2s},
    \qquad
    K_s(z)\leq CL^{-q}\quad\text{if }|z|>L/4.
\]
Since $q>d$,
\begin{equation}\label{eq:logistic-profile-mass}
    \sum_{y\in\Z^d}\Theta_R(y)\leq CR^d.
\end{equation}
It follows that
\begin{equation}\label{eq:logistic-far-convolution}
    \sum_{|z|>L/4}K_s(z)\Theta_R(x+z)\leq CL^{-q}\sum_{y\in\Z^d}\Theta_R(y)\leq CR^dL^{-q}.
\end{equation}
The same estimate holds with $x-z$ in place of $x+z$. Because $q>1$,
\[
    R^q+|x|^q\leq(R+|x|)^q=L^q,
\]
and hence
\begin{equation}\label{eq:logistic-critical-comparison}
    R^dL^{-q}
    \leq R^{-2s}\Theta_R(x).
\end{equation}
The term containing $\Theta_R(x)$ itself is at most
$CL^{-2s}\Theta_R(x)$. Combining
\eqref{eq:logistic-near-field}--\eqref{eq:logistic-critical-comparison}
proves the lemma.
\end{proof}

\begin{lemma}
\label{lem:coupled-logistic-barriers}
Fix $\mu>C_4$. There exists $\rho_*\geq1$ such that, for every
$\rho_0\geq\rho_*$, the systems
\begin{align}
    \alpha_+' &=
    a\alpha_+(1-\alpha_+)
    +\mu\rho_+^{-2s}\alpha_+,
    &
    q\frac{\rho_+'}{\rho_+}
    &=a\alpha_++\mu\rho_+^{-2s},
    \label{eq:plus-ode}\\
    \alpha_-' &=
    a\alpha_-(1-\alpha_-)
    -\mu\rho_-^{-2s}\alpha_-,
    &
    q\frac{\rho_-'}{\rho_-}
    &=a\alpha_--\mu\rho_-^{-2s},
    \label{eq:minus-ode}
\end{align}
with initial values
\begin{equation}\label{eq:coupled-initial-values}
    \alpha_+(0)=2,\qquad
    \alpha_-(0)=\frac12,\qquad
    \rho_+(0)=\rho_-(0)=\rho_0,
\end{equation}
have global solutions $\alpha_\pm,\rho_\pm\in C^1([0,\infty))$ with $\rho_\pm(t)>0$ for every $t\geq0$, satisfying
\begin{equation}\label{eq:amplitude-invariant-regions}
    1\leq\alpha_+(t)\leq2,
    \qquad
    \frac12\leq\alpha_-(t)\leq1.
\end{equation}
The functions
\begin{equation}\label{eq:coupled-barrier-functions}
    w_\pm(t,x):=\alpha_\pm(t)\Theta_{\rho_\pm(t)}(x)
\end{equation}
are, respectively, a supersolution and a subsolution of the logistic
equation. Moreover, there exist $\ell_\pm\in(0,\infty)$ such that
\begin{equation}\label{eq:coupled-ode-asymptotics}
    \alpha_\pm(t)\to1,
    \qquad
    \rho_\pm(t)\exp(-at/q)\to\ell_\pm
\end{equation}
as $t\to\infty$.
\end{lemma}

\begin{proof}
Set
\[
    \rho_*:=\max\left\{1,\left(\frac{4\mu}{a}\right)^{1/(2s)}\right\},
\]
and fix $\rho_0\geq\rho_*$. Then
\begin{equation}\label{eq:rho-zero-choice}
    \mu\rho_0^{-2s}\leq\frac{a}{4}.
\end{equation}
Consider a first exit from the region
\[
    \rho_\pm\geq\rho_0,\qquad
    1\leq\alpha_+\leq2,\qquad
    \frac12\leq\alpha_-\leq1.
\]
As long as the solution lies in this region,
\[
    q\frac{\rho_+'}{\rho_+}\geq a,
    \qquad
    q\frac{\rho_-'}{\rho_-}
    \geq\frac{a}{2}-\mu\rho_0^{-2s}
    \geq\frac{a}{4}.
\]
Neither radius can cross below $\rho_0$. On the four amplitude boundaries,
\begin{align*}
    \alpha_+=1&:\quad \alpha_+'=\mu\rho_+^{-2s}>0,\\
    \alpha_+=2&:\quad \alpha_+'=-2a+2\mu\rho_+^{-2s}\leq-\frac{3a}{2}<0,\\
    \alpha_-=\frac12&:\quad \alpha_-'
        =\frac{a}{4}-\frac{\mu}{2}\rho_-^{-2s}
        \geq\frac{a}{8}>0,\\
    \alpha_-=1&:\quad \alpha_-'=-\mu\rho_-^{-2s}<0.
\end{align*}
These inequalities rule out a first exit from the stated region.
Within this invariant region, the right-hand sides satisfy
$|\alpha_\pm'|\leq C$ and $|\rho_\pm'|\leq C\rho_\pm$. The continuation
criterion therefore makes both solutions global. In particular,
\begin{equation}\label{eq:rho-forcing-decay}
    \rho_+(t)^{-2s}\leq
    \rho_0^{-2s}\exp(-2sat/q),
    \qquad
    \rho_-(t)^{-2s}\leq
    \rho_0^{-2s}\exp(-sat/(2q)).
\end{equation}

Set $D_+:=\alpha_+-1$ and $D_-:=1-\alpha_-$. The amplitude equations and
\eqref{eq:amplitude-invariant-regions} yield
\begin{align*}
    D_+' &=-a\alpha_+D_++\mu\rho_+^{-2s}\alpha_+
    \leq-aD_++2\mu\rho_+^{-2s},\\
    D_-' &=-a\alpha_-D_-+\mu\rho_-^{-2s}\alpha_-
    \leq-\frac{a}{2}D_-+\mu\rho_-^{-2s}.
\end{align*}
The variation-of-constants formula and \eqref{eq:rho-forcing-decay} give an exponentially decaying bound for each $D_\pm$; hence
$D_\pm(t)\to0$ and
\begin{equation}\label{eq:amplitude-integrability}
    \int_0^\infty D_\pm(t)\,dt<\infty.
\end{equation}
Finally,
\begin{align*}
    \frac{d}{dt}\left(\log\rho_+-\frac{a}{q}t\right)
    &=\frac{aD_++\mu\rho_+^{-2s}}{q},\\
    \frac{d}{dt}\left(\log\rho_--\frac{a}{q}t\right)
    &=-\frac{aD_-+\mu\rho_-^{-2s}}{q}.
\end{align*}
The right-hand sides are integrable by
\eqref{eq:rho-forcing-decay} and \eqref{eq:amplitude-integrability}.
Both logarithmic differences therefore converge to finite real numbers,
which proves \eqref{eq:coupled-ode-asymptotics}.

To check the differential inequalities, differentiate
\eqref{eq:logistic-profile}:
\begin{equation}\label{eq:profile-radius-derivative}
    \partial_t\Theta_{\rho(t)}(x)
    =
    q\frac{\rho'(t)}{\rho(t)}
    \Theta_{\rho(t)}(x)
    \bigl(1-\Theta_{\rho(t)}(x)\bigr).
\end{equation}
Substitution of \eqref{eq:plus-ode} and \eqref{eq:minus-ode} gives the exact
identity
\begin{equation}\label{eq:coupled-barrier-identity}
    \partial_tw_\pm-aw_\pm(1-w_\pm)
    =
    \pm\mu\rho_\pm^{-2s}w_\pm
    \bigl(2-\Theta_{\rho_\pm}\bigr).
\end{equation}
Since $2-\Theta_{\rho_\pm}\geq1$, Lemma~\ref{lem:logistic-weighted-profile}
and $\mu>C_4$ give explicitly
\begin{align*}
    \partial_tw_++A^sw_+-aw_+(1-w_+)
    &\geq(\mu-C_4)\rho_+^{-2s}w_+\geq0,\\
    \partial_tw_-+A^sw_--aw_-(1-w_-)
    &\leq-(\mu-C_4)\rho_-^{-2s}w_-\leq0.
\end{align*}
This gives the asserted signs of $w_+$ and $w_-$.
\end{proof}

\begin{lemma}
\label{lem:logistic-barrier-initialization}
Assume \emph{(I)} and \emph{(L)}, and let $u$ be the corresponding global
solution.
For every fixed $\mu>C_4$, the initial radius $\rho_0$ and the corresponding
global coupled solution in Lemma~\ref{lem:coupled-logistic-barriers} may be
chosen simultaneously so that
\begin{equation}\label{eq:upper-critical-initialization}
    u_0(x)\leq2\Theta_{\rho_0}(x)
    \qquad (x\in\Z^d).
\end{equation}
For the same $\rho_0$, there exists $T_0>0$ such that
\begin{equation}\label{eq:lower-critical-initialization}
    u(T_0,x)\geq\frac12\Theta_{\rho_0}(x)
    \qquad (x\in\Z^d).
\end{equation}
\end{lemma}

\begin{proof}
Choose $\rho_0$ as in \eqref{eq:rho-zero-choice} and large enough that
$\supp u_0\subseteq B_{\rho_0}$. If $x\in B_{\rho_0}$, then
$2\Theta_{\rho_0}(x)\geq1\geq u_0(x)$; outside this ball, $u_0=0$.
This proves \eqref{eq:upper-critical-initialization}.

For the lower bound, Lemma~\ref{lem:polynomial-seeding} gives
\[
    u(1,x)\geq c_0(1+|x|^2)^{-q/2}.
\]
For $\rho_0\geq1$,
\begin{equation}\label{eq:profile-comparison-for-initialization}
    \Theta_{\rho_0}(x)
    \leq C\rho_0^q(1+|x|^2)^{-q/2}.
\end{equation}
This follows from $(1+r^2)^{q/2}\leq C(1+r^q)$ and
\[
    \frac{1+r^q}{1+r^q/\rho_0^q}\leq\rho_0^q.
\]
Let $C$ be the constant in \eqref{eq:profile-comparison-for-initialization}. Choose
\[
    0<\beta_0\leq
    \min\left\{\frac14,\frac{c_0}{C\rho_0^q}\right\}.
\]
Then \eqref{eq:profile-comparison-for-initialization} and the seed estimate imply
\[
    \beta_0\Theta_{\rho_0}(x)\leq u(1,x)
    \quad (x\in\Z^d).
\]
Since $\mu>C_4$, condition \eqref{eq:rho-zero-choice} already gives
$C_4\rho_0^{-2s}<a/2$. Let $\beta$ solve
\begin{equation}\label{eq:fixed-radius-amplitude}
    \beta'
    =
    a\beta(1-\beta)-C_4\rho_0^{-2s}\beta,
    \qquad
    \beta(0)=\beta_0.
\end{equation}
Its positive equilibrium is $1-C_4\rho_0^{-2s}/a>1/2$, so there is a
finite $\tau_0>0$ for which $\beta(\tau_0)=1/2$ and
$0<\beta(t)\leq1/2$ for $0\leq t\leq\tau_0$.

Set $z(t,x):=\beta(t)\Theta_{\rho_0}(x)$. By
Lemma~\ref{lem:logistic-weighted-profile} and
\eqref{eq:fixed-radius-amplitude},
\begin{align*}
    \partial_tz+A^sz-az(1-z)
    &\leq\bigl[a\beta(1-\beta)-C_4\rho_0^{-2s}\beta\bigr]\Theta_{\rho_0}+C_4\rho_0^{-2s}\beta\Theta_{\rho_0}-a\beta\Theta_{\rho_0}\bigl(1-\beta\Theta_{\rho_0}\bigr)\\
    &=-a\beta^2\Theta_{\rho_0}\bigl(1-\Theta_{\rho_0}\bigr)\leq0.
\end{align*}
Comparison on $[0,\tau_0]$, with the time origin shifted to $t=1$, gives
\[
    u(1+\tau_0,x)\geq z(\tau_0,x)
    =\frac12\Theta_{\rho_0}(x).
\]
Hence \eqref{eq:lower-critical-initialization} holds with
$T_0=1+\tau_0$.
\end{proof}

\begin{proof}[Proof of Theorem~\ref{thm:logistic-localization}]
Take $\rho_0$ as in Lemma~\ref{lem:logistic-barrier-initialization} and use
it in Lemma~\ref{lem:coupled-logistic-barriers}. Since the logistic
polynomial is defined on all of $\R$, the bounded-range comparison principle
applies also to the supersolution $w_+$, whose values lie in $[0,2]$.
Equations \eqref{eq:upper-critical-initialization} and
\eqref{eq:lower-critical-initialization} give
\begin{equation}\label{eq:critical-two-sided-comparison}
    w_-(t,x)\leq u(T_0+t,x),
    \qquad
    u(t,x)\leq w_+(t,x)
\end{equation}
for every $t\geq0$ and $x\in\Z^d$.

Fix $0<\theta<1$. By \eqref{eq:coupled-ode-asymptotics}, for all
sufficiently large $t$,
\[
    \alpha_-(t)\geq\frac{1+\theta}{2},
    \qquad
    c\exp(at/q)\leq\rho_\pm(t)\leq C\exp(at/q)
\]
for suitable positive constants $c$ and $C$. If
\[
    |x|\leq
    \left(\frac{1-\theta}{2\theta}\right)^{1/q}\rho_-(t),
\]
then $w_-(t,x)\geq\theta$. On the other hand, since
$\alpha_+(t)\leq2$, one has $w_+(t,x)<\theta$ whenever
\[
    |x|>
    \left(\frac{2}{\theta}-1\right)^{1/q}\rho_+(t).
\]
Apply these two bounds in \eqref{eq:critical-two-sided-comparison}. For the
lower inclusion at time $t\geq T_0$, use the subsolution at time $t-T_0$.
Since
\[
    \rho_-(t-T_0)\exp(-at/q)
    =\ee^{-aT_0/q}
      \rho_-(t-T_0)\exp\left(-\frac{a(t-T_0)}{q}\right)
    \to \ee^{-aT_0/q}\ell_->0,
\]
the fixed shift $T_0$ changes only the multiplicative constant. Lattice
rounding can be absorbed into the same constants. We obtain
\eqref{eq:logistic-level-inclusions}, and taking the supremum of $|x|$ over
the superlevel set gives \eqref{eq:logistic-radius-O-one}.
\end{proof}

\begin{lemma}
\label{lem:critical-tail-barrier-initialization}
Assume \emph{(L)}, \eqref{eq:critical-tail-initial-data}, and
\eqref{eq:critical-tail-upper-bound}. Fix $\mu>C_4$. Then $\rho_0$ in
Lemma~\ref{lem:coupled-logistic-barriers} can be chosen so that
\begin{equation}\label{eq:critical-tail-upper-initialization}
    u_0(x)\leq2\Theta_{\rho_0}(x)
    \qquad (x\in\Z^d).
\end{equation}
For the same $\rho_0$, there exists $T_0>0$ such that
\begin{equation}\label{eq:critical-tail-lower-initialization}
    u(T_0,x)\geq\frac12\Theta_{\rho_0}(x)
    \qquad (x\in\Z^d).
\end{equation}
\end{lemma}

\begin{proof}
Choose $\rho_0\geq1$ so large that
\begin{equation}\label{eq:critical-tail-rho-choice}
    \mu\rho_0^{-2s}\leq\frac{a}{4},
    \qquad
    \rho_0^q\geq C_0.
\end{equation}
This choice is admissible in Lemma~\ref{lem:coupled-logistic-barriers}.
If $|x|\leq\rho_0$, then
$2\Theta_{\rho_0}(x)\geq1\geq u_0(x)$. If $|x|>\rho_0$, then
\[
    2\Theta_{\rho_0}(x)
    =\frac{2\rho_0^q}{\rho_0^q+|x|^q}
    \geq\frac{\rho_0^q}{|x|^q}
    \geq C_0(1+|x|)^{-q}
    \geq u_0(x).
\]
This proves \eqref{eq:critical-tail-upper-initialization}.

For the lower barrier, choose $x_*\in\Z^d$ with
$u_0(x_*)>0$. Since $0\leq u\leq1$ and the logistic reaction is
nonnegative on $[0,1]$, Duhamel's formula and
Proposition~\ref{prop:heat-kernel} give
\[
    u(1,x)\geq u_0(x_*)G_1(x-x_*)\geq c_0(1+|x|^2)^{-q/2}
\]
for some $c_0>0$ and all $x\in\Z^d$. Using
\eqref{eq:profile-comparison-for-initialization}, choose
\begin{equation}\label{eq:critical-tail-beta-choice}
    0<\beta_0\leq
    \min\left\{\frac14,\frac{c_0}{C\rho_0^q}\right\},
\end{equation}
where $C$ is the constant in that estimate. Then
\[
    \beta_0\Theta_{\rho_0}(x)\leq u(1,x)
    \qquad (x\in\Z^d).
\]
By \eqref{eq:critical-tail-rho-choice} and $\mu>C_4$,
$C_4\rho_0^{-2s}<a/2$. Let $\beta$ solve
\begin{equation}\label{eq:critical-tail-fixed-radius-amplitude}
    \beta'
    =a\beta(1-\beta)-C_4\rho_0^{-2s}\beta,
    \qquad
    \beta(0)=\beta_0.
\end{equation}
Its positive equilibrium is $1-C_4\rho_0^{-2s}/a>1/2$; hence there exists
$\tau_0>0$ such that
\[
    \beta(\tau_0)=\frac12,
    \qquad
    0<\beta(t)\leq\frac12
    \quad (0\leq t\leq\tau_0).
\]
Set $z(t,x):=\beta(t)\Theta_{\rho_0}(x)$. By
Lemma~\ref{lem:logistic-weighted-profile} and
\eqref{eq:critical-tail-fixed-radius-amplitude},
\[
    \partial_tz+A^sz-az(1-z)\leq-a\beta^2\Theta_{\rho_0}\bigl(1-\Theta_{\rho_0}\bigr)\leq0.
\]
Hence $z$ is a subsolution. Comparison, with the time origin shifted to
$t=1$, yields
\[
    u(1+t,x)\geq z(t,x)
    \qquad (0\leq t\leq\tau_0).
\]
Taking $t=\tau_0$ proves
\eqref{eq:critical-tail-lower-initialization} with $T_0=1+\tau_0$.
\end{proof}

\begin{proof}[Proof of Theorem~\ref{thm:critical-tail-localization}]
Fix $\mu>C_4$ and choose $\rho_0$ as in
Lemma~\ref{lem:critical-tail-barrier-initialization}. Use this $\rho_0$ in
Lemma~\ref{lem:coupled-logistic-barriers}. Since the logistic polynomial is
defined on all of $\R$, the comparison principle applies to both $w_-$ and
$w_+$. Therefore
\begin{equation}\label{eq:critical-tail-two-sided-comparison}
    w_-(t,x)\leq u(T_0+t,x),
    \qquad
    u(t,x)\leq w_+(t,x)
\end{equation}
for all $t\geq0$ and $x\in\Z^d$.

Fix $0<\theta<1$. By \eqref{eq:coupled-ode-asymptotics}, for all sufficiently
large $t$,
\[
    \alpha_-(t)\geq\frac{1+\theta}{2},
    \qquad
    c\ee^{at/q}\leq\rho_\pm(t)\leq C\ee^{at/q}
\]
with positive constants $c$ and $C$. Consequently,
\[
    w_-(t,x)\geq\theta
    \quad\text{if}\quad
    |x|\leq
    \left(\frac{1-\theta}{2\theta}\right)^{1/q}\rho_-(t),
\]
whereas
\[
    w_+(t,x)<\theta
    \quad\text{if}\quad
    |x|>
    \left(\frac{2}{\theta}-1\right)^{1/q}\rho_+(t).
\]
Apply these estimates in \eqref{eq:critical-tail-two-sided-comparison}. For
the lower bound at time $t$, use $w_-(t-T_0,\cdot)$; the fixed shift only
changes the multiplicative constant because
\[
    \rho_-(t-T_0)\ee^{-at/q}
    \to \ee^{-aT_0/q}\ell_->0.
\]
After enlarging the constants to absorb lattice rounding, we obtain
\eqref{eq:critical-tail-level-inclusions}. Taking the outer radius and then
logarithms gives \eqref{eq:critical-tail-radius}.
\end{proof}

\begin{proof}[Proof of Corollary~\ref{cor:critical-tail-interface}]
The two inclusions in \eqref{eq:critical-tail-level-inclusions} imply
\[
    c_\theta\ee^{at/q}\leq r_\theta(t)
    \leq R_\theta(t)\leq C_\theta\ee^{at/q}
\]
for all sufficiently large $t$, up to an immaterial lattice-rounding
constant. This proves \eqref{eq:inner-outer-radius-asymptotics}.

Let $x\in\partial_{\rm d}E_\theta(t)$. The upper inclusion gives
$|x|\leq C_\theta\ee^{at/q}$. By definition of the discrete boundary, there
is $y\notin E_\theta(t)$ with $|x-y|_1=1$. The lower inclusion forces
$|y|>c_\theta\ee^{at/q}$, and therefore
\[
    |x|\geq |y|-|x-y|
    >c_\theta\ee^{at/q}-1.
\]
For all sufficiently large $t$, the last quantity is at least
$(c_\theta/2)\ee^{at/q}$. This proves
\eqref{eq:discrete-interface-annulus}.
\end{proof}

\begin{remark}
Theorem~\ref{thm:critical-tail-localization} may be compared directly with
Theorem~1.6 of Cabr\'e and Roquejoffre
\cite{CabreRoquejoffre2013}. In the special case $d=1$, $s=1/2$, and $a=1$,
one has $q=2$, and \eqref{eq:critical-tail-level-inclusions} gives the scale
$\ee^{t/2}$. The half-Laplacian in the continuum has an explicit Cauchy
kernel, whereas the lattice proof uses the scale-uniform weighted estimate
\eqref{eq:logistic-weighted-estimate}. A second difference is purely
discrete: the exact set $\{u(t,\cdot)=\theta\}$ may be empty on $\Z^d$.
Corollary~\ref{cor:critical-tail-interface} therefore records the analogous
interface statement through the nearest-neighbor boundary of the superlevel
set. The exact logistic algebra in \eqref{eq:coupled-barrier-identity} remains
essential; no bounded-error claim is made under assumption \emph{(F)} alone.
\end{remark}

\section{Nondecreasing initial data in one dimension}
\label{sec:monotone-data}

Throughout this section, $d=1$, assumption \emph{(F)} is in force, and $u$
denotes the solution generated by a nonzero nondecreasing initial datum with
values in $[0,1]$.

For compactly supported data in one dimension, the pointwise heat-kernel tail
gives the exponent $a/(1+2s)$. A nondecreasing datum occupies a right
half-lattice, so the cumulative tail enters in place of the pointwise tail
and changes the exponent to $a/(2s)$. This section proves
Theorem~\ref{thm:monotone-sharp-propagation}.

For $r\in\R$, write $r^-:=\max\{-r,0\}$. For $R\geq1$, define
\begin{equation}\label{eq:one-sided-profile}
    \Xi_R(r):=
    \left(1+\frac{(r^-)^2}{R^2}\right)^{-s}.
\end{equation}
Whenever $A^s\Xi_R(j)$ is written below, $\Xi_R$ denotes the restriction
of this continuous profile to the integer lattice $\Z$.
Thus $\Xi_R=1$ on $[0,\infty)$ and
$\Xi_R(r)\asymp R^{2s}|r|^{-2s}$ as $r\to-\infty$.

\begin{lemma}
\label{lem:cumulative-heat-kernel}
For $k\in\Z$, set $H_k(j):=\one_{\{j\geq k\}}$. There exist $c,C>0$
such that
\begin{equation}\label{eq:half-line-heat-lower}
    S(1)H_k(j)\geq c\Xi_1(j-k)
\end{equation}
for every $j,k\in\Z$. For the fixed admissible integer $m$ chosen after
Proposition~\ref{prop:heat-kernel}, one also has
\begin{equation}\label{eq:one-sided-semigroup-upper}
    S(t)\Xi_1(j)
    \leq C(1+t+t^m)\Xi_1(j)
\end{equation}
for every $t>0$ and $j\in\Z$.
\end{lemma}

\begin{proof}
By translation invariance,
\[
    S(1)H_k(j)=\sum_{\ell\geq k}G_1(j-\ell).
\]
Set $r=j-k$. If $r\geq0$, the term $\ell=j$ gives the lower bound
$G_1(0)>0$. If $r<0$, Proposition~\ref{prop:heat-kernel} gives
\[
    S(1)H_k(j)
    =\sum_{z\leq r}G_1(z)
    \geq c\sum_{n\geq|r|}(1+n)^{-1-2s}
    \geq c(1+|r|)^{-2s}.
\]
This proves \eqref{eq:half-line-heat-lower}.

For \eqref{eq:one-sided-semigroup-upper}, the Markov property gives
$S(t)\Xi_1(j)\leq1=\Xi_1(j)$ when $j\geq0$. Let $j<0$ and split
\[
    S(t)\Xi_1(j)
    =
    \sum_{\ell\leq j/2}G_t(j-\ell)\Xi_1(\ell)
    +
    \sum_{\ell>j/2}G_t(j-\ell)\Xi_1(\ell).
\]
In the first sum, $\Xi_1(\ell)\leq C(1+|j|)^{-2s}$, so the sum is at
most this quantity up to a constant. In the second, $z=j-\ell<j/2$.
The heat-kernel upper bound therefore yields
\[
    \sum_{\ell>j/2}G_t(j-\ell)\Xi_1(\ell)
    \leq
    C(t+t^m)\sum_{|z|\geq|j|/2}(1+|z|)^{-1-2s}
    \leq C(t+t^m)(1+|j|)^{-2s}.
\]
Since $\Xi_1(j)\asymp(1+|j|)^{-2s}$ for $j<0$, the proof is complete.
\end{proof}

\begin{lemma}
\label{lem:monotonicity-and-seeding}
The solution issued from \eqref{eq:monotone-initial-data} is nondecreasing:
\begin{equation}\label{eq:solution-monotonicity}
    u(t,j)\leq u(t,j+1)
\end{equation}
for every $t\geq0$ and $j\in\Z$. There also exist $k_*\in\Z$ and
$c_*>0$ such that
\begin{equation}\label{eq:one-sided-polynomial-seed}
    u(1,j)\geq c_*\Xi_1(j-k_*)
\end{equation}
for every $j\in\Z$.
\end{lemma}

\begin{proof}
Translation invariance shows that $v(t,j):=u(t,j+1)$ solves the same
equation as $u$. Since $v(0,j)\geq u_0(j)$, comparison proves
\eqref{eq:solution-monotonicity}.

Choose $k_*$ with $\mu_0:=u_0(k_*)>0$. Monotonicity of the datum gives
$u_0\geq\mu_0H_{k_*}$. Because $f(u)\geq0$ on the invariant interval,
Duhamel's formula, positivity of $S(t)$, and
Lemma~\ref{lem:cumulative-heat-kernel} give
\[
    u(1)\geq S(1)u_0
    \geq\mu_0S(1)H_{k_*}
    \geq c_*\Xi_1(\,\cdot-k_*).
\]
\end{proof}

\begin{lemma}
\label{lem:one-sided-weighted-profile}
There exists $C_5>0$ such that
\begin{equation}\label{eq:one-sided-weighted-estimate}
    |A^s\Xi_R(j)|
    \leq C_5R^{-2s}\Xi_R(j)
\end{equation}
for every $R\geq1$ and $j\in\Z$.
\end{lemma}

\begin{proof}
When $j\geq0$, $\Xi_R(j)=1$ and
\[
    A^s\Xi_R(j)
    =
    \sum_{z\neq0}K_s(z)\bigl(1-\Xi_R(j+z)\bigr)\geq0.
\]
If $j\geq R$, the summand vanishes unless $j+z<0$, and the kernel-tail
estimate gives $|A^s\Xi_R(j)|\leq Cj^{-2s}\leq CR^{-2s}$. If
$0\leq j<R$, then
\[
    1-\Xi_R(j+z)
    \leq C\min\{1,|z|^2/R^2\}.
\]
Consequently,
\[
    |A^s\Xi_R(j)|
    \leq
    CR^{-2}\sum_{0<|z|\leq R}|z|^2K_s(z)
    +C\sum_{|z|>R}K_s(z)
    \leq CR^{-2s}.
\]

Suppose that $j<0$ and set $L:=R+|j|$. The continuous extension in
\eqref{eq:one-sided-profile} belongs to $C^{1,1}(\R)$. We claim that
\begin{equation}\label{eq:one-sided-relative-hessian}
    \mathop{\rm ess\,sup}_{|\vartheta|\leq1}
    |\Xi_R''(j+\vartheta z)|
    \leq CL^{-2}\Xi_R(j)
\end{equation}
whenever $|z|\leq L/4$. If $|j|\leq2R$, then
$\Xi_R(j)\geq5^{-s}$, $L\leq3R$, and
$\|\Xi_R''\|_{L^\infty}\leq CR^{-2}\leq CL^{-2}\Xi_R(j)$.
If $|j|>2R$, the relevant segment stays on the negative half-line at
distance comparable to $|j|$ from the origin, and direct differentiation
again gives \eqref{eq:one-sided-relative-hessian}.

Symmetrization, Taylor's formula, and the jump-kernel bound now imply
\[
    \sum_{0<|z|\leq L/4}K_s(z)
    |2\Xi_R(j)-\Xi_R(j+z)-\Xi_R(j-z)|
    \leq CL^{-2s}\Xi_R(j).
\]
For $|z|>L/4$, use
\[
    \sum_{|z|>L/4}K_s(z)\leq CL^{-2s}
\]
and the critical comparison
\begin{equation}\label{eq:one-sided-critical-comparison}
    L^{-2s}\leq(R^2+j^2)^{-s}=R^{-2s}\Xi_R(j).
\end{equation}
Since $0<\Xi_R\leq1$, the absolute value of the complementary sum is at
most
\[
    C\Xi_R(j)L^{-2s}+CL^{-2s}
    \leq CR^{-2s}\Xi_R(j).
\]
This proves \eqref{eq:one-sided-weighted-estimate}.
\end{proof}

\begin{proposition}
\label{prop:one-sided-expanding-barrier}
For every $\gamma\in(0,a)$, there exist $R\geq1$ and $\eta\in(0,1)$ such
that
\begin{equation}\label{eq:one-sided-expanding-barrier}
    u(1+t,j)
    \geq
    \eta\,
    \Xi_{R\exp(\gamma t/(2s))}(j-k_*)
\end{equation}
for every $t\geq0$ and $j\in\Z$.
\end{proposition}

\begin{proof}
Set $\rho:=(a-\gamma)/2$. Since $f'(0)=a$, there exists $\eta_*>0$
such that $f(r)\geq(a-\rho)r$ for $0\leq r\leq\eta_*$. Choose $R$ so
large that
\[
    C_5R^{-2s}\leq a-\rho-\gamma=\frac{a-\gamma}{2},
\]
and then choose
\[
    0<\eta\leq\min\{\eta_*,c_*R^{-2s}\}.
\]
Let $R(t):=R\exp(\gamma t/(2s))$ and
\[
    \underline u(t,j):=\eta\Xi_{R(t)}(j-k_*).
\]
Direct differentiation gives
\[
    \partial_t\underline u
    =
    \gamma
    \frac{((j-k_*)^-)^2}
    {R(t)^2+((j-k_*)^-)^2}\,\underline u
    \leq\gamma\underline u.
\]
Lemma~\ref{lem:one-sided-weighted-profile} and $R(t)\geq R$ give
\[
    A^s\underline u(t,j)
    \leq C_5R(t)^{-2s}\underline u(t,j)
    \leq C_5R^{-2s}\underline u(t,j).
\]
Hence
\[
    \partial_t\underline u+A^s\underline u
    \leq(a-\rho)\underline u
    \leq f(\underline u).
\]
Finally,
\[
    \Xi_R(r)\leq R^{2s}\Xi_1(r)
\]
for every $r\in\R$. Lemma~\ref{lem:monotonicity-and-seeding} and the
choice of $\eta$ therefore give $\underline u(0,\cdot)\leq u(1,\cdot)$.
Comparison proves the proposition.
\end{proof}

\begin{corollary}
\label{cor:one-sided-positive-plateau}
For every $\beta<a/(2s)$, there exist $\eta_\beta,T_\beta>0$ such that
\begin{equation}\label{eq:one-sided-positive-plateau}
    u(t,j)\geq\eta_\beta
\end{equation}
whenever $t\geq T_\beta$ and $j\geq-\exp(\beta t)$.
\end{corollary}

\begin{proof}
Choose $\gamma\in(0,a)$ with $\beta<\gamma/(2s)$. In the stated region,
\[
    \frac{(j-k_*)^-}{R\exp(\gamma(t-1)/(2s))}
    \leq
    C\frac{\exp(\beta t)+1}{\exp(\gamma t/(2s))}
    \to0.
\]
The right-hand side of \eqref{eq:one-sided-expanding-barrier} therefore
converges uniformly to $\eta$ there. Taking $\eta_\beta=\eta/2$ for
large $t$ proves the assertion.
\end{proof}

\begin{lemma}
\label{lem:one-sided-plateau-upgrade}
Let $\beta_0>\max\{\beta,0\}$. If there exist $\eta_0,T_0>0$ such that
\begin{equation}\label{eq:one-sided-faster-plateau}
    u(t,j)\geq\eta_0
    \quad\text{when}\quad
    t\geq T_0,\quad j\geq-\exp(\beta_0t),
\end{equation}
then
\begin{equation}\label{eq:one-sided-upgrade-conclusion}
    \lim_{t\to\infty}
    \inf_{\substack{j\in\Z\\j\geq-\exp(\beta t)}}u(t,j)=1.
\end{equation}
\end{lemma}

\begin{proof}
Suppose otherwise. Then there exist $\delta>0$, $t_n\to\infty$, and
$j_n\geq-\exp(\beta t_n)$ such that
\[
    u(t_n,j_n)\leq1-\delta.
\]
Define $U_n(\tau,z):=u(t_n+\tau,j_n+z)$. The compactness argument in
Proposition~\ref{prop:plateau-upgrade} applies without change. The uniform
time-Lipschitz bound gives local compactness at each lattice site, and the
summability of $K_s$ permits passage to the limit in $A^s$. After taking a
subsequence, $U_n$ therefore converges to an entire classical solution $U$
on $\R\times\Z$ satisfying
\[
    0\leq U\leq1,
    \qquad U(0,0)\leq1-\delta.
\]

Fix $T,N>0$. If $|\tau|\leq T$ and $|z|\leq N$, then
\[
    j_n+z\geq-\exp(\beta t_n)-N.
\]
Since $\beta_0>\max\{\beta,0\}$,
\[
    \exp(\beta t_n)+N
    \leq\exp(\beta_0(t_n-T))
    \leq\exp(\beta_0(t_n+\tau))
\]
for all sufficiently large $n$. After also requiring
$t_n-T\geq T_0$,
\eqref{eq:one-sided-faster-plateau} gives
$U_n(\tau,z)\geq\eta_0$. Passing to the limit first on each fixed cylinder and then using the arbitrariness of $T$ and $N$ yields
$\eta_0\leq U\leq1$ on $\R\times\Z$.
Lemma~\ref{lem:liouville} therefore gives $U\equiv1$, contradicting
$U(0,0)\leq1-\delta$.
\end{proof}

\begin{lemma}
\label{lem:monotone-linear-upper}
Let $m\in\N$ be the fixed admissible exponent chosen after
Proposition~\ref{prop:heat-kernel}.  There exists $C>0$ such that
\begin{equation}\label{eq:monotone-linear-tail-bound}
    u(t,j)
    \leq C\exp(at)(1+t+t^m)\Xi_1(j)
\end{equation}
for every $t>0$ and $j\in\Z$.
\end{lemma}

\begin{proof}
Conditions \eqref{eq:monotone-initial-data} and
\eqref{eq:monotone-left-tail} imply $u_0\leq C\Xi_1$. As in
\eqref{eq:linear-comparison}, set $v=\exp(-at)u$ and use
$\partial_tv+A^sv\leq0$. Positivity of $S(t)$ and
Lemma~\ref{lem:cumulative-heat-kernel} give
\[
    u(t)\leq\exp(at)S(t)u_0
    \leq C\exp(at)S(t)\Xi_1
    \leq C\exp(at)(1+t+t^m)\Xi_1.
\]
\end{proof}

\begin{proof}[Proof of Theorem~\ref{thm:monotone-sharp-propagation}]
If $\sigma>a/(2s)$ and $j\leq-\exp(\sigma t)$, then
$\Xi_1(j)\leq C\exp(-2s\sigma t)$. Hence
Lemma~\ref{lem:monotone-linear-upper} gives
\[
    u(t,j)
    \leq C(1+t+t^m)\exp\bigl((a-2s\sigma)t\bigr),
\]
which tends to zero uniformly. This proves
\eqref{eq:monotone-outer-conclusion}.

Let $\sigma<a/(2s)$ and choose
\[
    \max\{\sigma,0\}<\beta_0<\frac{a}{2s}.
\]
Corollary~\ref{cor:one-sided-positive-plateau} supplies
\eqref{eq:one-sided-faster-plateau} at exponent $\beta_0$.
Lemma~\ref{lem:one-sided-plateau-upgrade}, with $\beta=\sigma$, proves
\eqref{eq:monotone-inner-conclusion}.

By \eqref{eq:solution-monotonicity}, every nonempty superlevel set is an
upper lattice ray. Fix $0<\varepsilon<a/(2s)$ and set
\[
    \sigma_-:=\frac{a}{2s}-\varepsilon,
    \qquad
    \sigma_+:=\frac{a}{2s}+\varepsilon.
\]
For all sufficiently large $t$, the inner estimate gives
$u(t,j)>\theta$ when $j\geq-\exp(\sigma_-t)$, while the outer estimate
gives $u(t,j)<\theta$ when $j\leq-\exp(\sigma_+t)$. Thus
$X_\theta(t)$ is finite and negative, and integer rounding gives
\[
    \exp(\sigma_-t)-1
    \leq -X_\theta(t)
    \leq \exp(\sigma_+t)+1.
\]
Divide the logarithms by $t$ and let first $t\to\infty$ and then
$\varepsilon\downarrow0$ to obtain \eqref{eq:monotone-interface-law}.
\end{proof}

\begin{remark}
The left-tail assumption \eqref{eq:monotone-left-tail} is used only for the
outer estimate. The inner convergence
\eqref{eq:monotone-inner-conclusion} holds for every nonzero nondecreasing
datum with values in $[0,1]$.
\end{remark}

\begin{remark}
The exponent $2s$ is forced by the half-lattice geometry:
\[
    S(1)H_0(j)
    =\sum_{z\leq j}G_1(z)
    \asymp |j|^{-2s}
    \quad\text{as }j\to-\infty.
\]
This is the lattice counterpart of Theorem~1.5 of
\cite{CabreRoquejoffre2013}. On a lattice, an exact level set
$\{j:u(t,j)=\theta\}$ may be empty, so the threshold
$X_\theta(t)$ in \eqref{eq:monotone-interface-definition} is the natural
interface position.
\end{remark}

\section{Directional spreading and nonexistence of planar fronts}
\label{sec:traveling-fronts}

Throughout this section, assumption \emph{(F)} is in force.

The exponential spreading law rules out planar fronts of finite constant
speed. The argument works in every direction and uses only comparison and
the spreading estimate; it requires neither monotonicity of the profile nor
an arithmetic condition on the direction.

Let
\[
    \mathbb S^{d-1}:=\{e\in\R^d:|e|=1\}.
\]
An entire classical solution is a function
\[
    W\in C^1(\R;\ellinf)
\]
with values in $[0,1]$ such that
\[
    \partial_tW(t)+A^sW(t)=f(W(t))
    \quad\text{in }\ellinf
\]
for every $t\in\R$.

\begin{definition}
\label{def:planar-traveling-front}
For $e\in\mathbb S^{d-1}$ and $c\in\R$, a
classical planar traveling front in direction $e$ with speed $c$ is an
entire classical solution of the form
\begin{equation}\label{eq:planar-front-form}
    W(t,x)=\phi(x\cdot e-ct),
    \qquad t\in\R,\quad x\in\Z^d,
\end{equation}
where $\phi\in C(\R;[0,1])$ satisfies
\begin{equation}\label{eq:planar-front-limits}
    \lim_{\rho\to-\infty}\phi(\rho)=1,
    \qquad
    \lim_{\rho\to+\infty}\phi(\rho)=0.
\end{equation}
No monotonicity of $\phi$ is assumed. With this convention, a positive
speed describes motion in the direction $e$.
\end{definition}

The associated profile operator is defined, for bounded
$\phi:\R\to\R$, by
\begin{equation}\label{eq:projected-profile-operator}
    (\mathcal L_{s,e}\phi)(\rho)
    :=
    \sum_{z\neq0}K_s(z)
    \bigl(\phi(\rho)-\phi(\rho+z\cdot e)\bigr).
\end{equation}
The series converges absolutely and uniformly in $\rho$ by
\eqref{eq:kappa-definition}. A continuous function satisfying
\eqref{eq:planar-front-limits} is uniformly continuous on $\R$; hence both
$\rho\mapsto\mathcal L_{s,e}\phi(\rho)$ and
$\rho\mapsto f(\phi(\rho))$ are continuous. If $c\neq0$, then
$t\mapsto W(t,0)=\phi(-ct)$ belongs to $C^1(\R)$, because point
evaluation is a bounded linear functional on $\ellinf$. Since
$t\mapsto-ct$ is a diffeomorphism of $\R$, it follows that
$\phi\in C^1(\R)$. Taking $x=0$ and $t=-\rho/c$ in the equation for
$W$ then gives
\begin{equation}\label{eq:traveling-profile-equation}
    -c\phi'(\rho)+\mathcal L_{s,e}\phi(\rho)
    =f(\phi(\rho)),
    \qquad \rho\in\R.
\end{equation}
When $c=0$, set
\[
    \Gamma_e:=\{x\cdot e:x\in\Z^d\}.
\]
The equation for $W$ gives
\[
    \mathcal L_{s,e}\phi(\rho)=f(\phi(\rho)),
    \qquad \rho\in\Gamma_e.
\]
If $e$ is not proportional to an integer vector, then $\Gamma_e$ is
dense in $\R$, and continuity of both sides extends this identity to every
$\rho\in\R$. Indeed, an additive subgroup of $\R$ is either dense or cyclic;
if $\Gamma_e=h\Z$ for some $h>0$, then every component of $e$ is an integer
multiple of $h$, so $e$ is proportional to an integer vector.
If $e$ is proportional to an integer vector, then $\Gamma_e$ is
discrete, and the values of the continuous interpolation $\phi$ outside
$\Gamma_e$ do not affect $W$. Conversely, if
$\phi\in C^1(\R;[0,1])$, $\phi'$ is bounded and uniformly continuous,
and \eqref{eq:traveling-profile-equation} holds for every $\rho$, then
\eqref{eq:planar-front-form} is an entire classical solution. Uniform
continuity of $\phi'$ gives strong $C^1(\R;\ellinf)$ differentiability,
while uniform convergence in \eqref{eq:projected-profile-operator}
verifies the equation in $\ellinf$.

For a fixed direction $e$, let
\begin{equation}\label{eq:traveling-speed-set}
    \mathcal C_{\mathrm{tw}}(e)
    :=
    \left\{
        c\in\R:
        \text{a classical planar traveling front of speed $c$
        exists in direction $e$}
    \right\}.
\end{equation}
We define the extended minimal traveling-front threshold by
\begin{equation}\label{eq:extended-minimal-wave-speed}
    c_{\mathrm{tw}}^*(e)
    :=
    \inf\bigl(\mathcal C_{\mathrm{tw}}(e)\cap[0,\infty)\bigr),
\end{equation}
where $\inf\varnothing:=+\infty$. Thus
$c_{\mathrm{tw}}^*(e)=+\infty$ may express that the finite-speed set is
empty; it does not assert the existence of a front with infinite speed.

Fix $\eta\in(0,1)$, and denote by $u^\eta$ the solution with single-site
initial datum
\begin{equation}\label{eq:single-site-seed}
    u^\eta(0,x)=\eta\delta_0(x).
\end{equation}
For $0<\theta<1$ and $e\in\mathbb S^{d-1}$, set
\begin{equation}\label{eq:directional-level-position}
    X_{\theta,e}^{\eta}(t)
    :=
    \sup\left\{
        x\cdot e:
        x\in\Z^d,\ u^\eta(t,x)\geq\theta
    \right\},
\end{equation}
with $\sup\varnothing:=-\infty$. The next lemma shows that this number is
positive and finite for all sufficiently large $t$. Accordingly,
$\log X_{\theta,e}^{\eta}(t)$ is used only on this eventual time interval.

\begin{lemma}
\label{lem:directional-exponential-spreading}
For every $\eta,\theta\in(0,1)$ and every
$e\in\mathbb S^{d-1}$,
\begin{equation}\label{eq:directional-logarithmic-law}
    \lim_{t\to\infty}
    \frac{1}{t}\log X_{\theta,e}^{\eta}(t)
    =
    \frac{a}{d+2s}.
\end{equation}
Consequently,
\begin{equation}\label{eq:infinite-directional-linear-speed}
    \lim_{t\to\infty}
    \frac{X_{\theta,e}^{\eta}(t)}{t}
    =+\infty.
\end{equation}
\end{lemma}

\begin{proof}
Recall that $q=d+2s$. Fix $0<\beta<a/q$. Choose an index $j$ for
which $e_j\neq0$, let $\mathbf e_j$ be the $j$th coordinate vector,
and put
\[
    v:=\operatorname{sgn}(e_j)\mathbf e_j,
    \qquad
    x_t:=\left\lfloor\ee^{\beta t}\right\rfloor v.
\]
Then $|x_t|\leq\ee^{\beta t}$ and
\[
    x_t\cdot e
    =
    |e_j|\left\lfloor\ee^{\beta t}\right\rfloor.
\]
Apply the lower estimate in Theorem~\ref{thm:main} to $u^\eta$ with
$\varepsilon=a-q\beta>0$. It gives $u^\eta(t,x_t)\to1$, so
\[
    X_{\theta,e}^{\eta}(t)
    \geq
    |e_j|\left\lfloor\ee^{\beta t}\right\rfloor
\]
for all sufficiently large $t$. Hence
\begin{equation}\label{eq:directional-log-liminf}
    \liminf_{t\to\infty}
    \frac{1}{t}\log X_{\theta,e}^{\eta}(t)
    \geq\beta.
\end{equation}
Letting $\beta\uparrow a/q$ gives the required lower bound.

For the upper bound, fix $\gamma>a/q$ and use the upper estimate in
Theorem~\ref{thm:main} with $\varepsilon=q\gamma-a>0$. For all
sufficiently large $t$, every $x$ satisfying
$u^\eta(t,x)\geq\theta$ must obey $|x|<\ee^{\gamma t}$. Since
$x\cdot e\leq|x|$, it follows that
\[
    X_{\theta,e}^{\eta}(t)\leq\ee^{\gamma t}.
\]
Therefore
\[
    \limsup_{t\to\infty}
    \frac{1}{t}\log X_{\theta,e}^{\eta}(t)
    \leq\gamma.
\]
Letting $\gamma\downarrow a/q$ and combining this inequality with
\eqref{eq:directional-log-liminf} proves
\eqref{eq:directional-logarithmic-law}. Finally, the lower estimate
with any fixed $\beta\in(0,a/q)$ gives
\[
    \frac{X_{\theta,e}^{\eta}(t)}{t}
    \geq
    |e_j|
    \frac{\lfloor\ee^{\beta t}\rfloor}{t}
    \to+\infty,
\]
which proves \eqref{eq:infinite-directional-linear-speed}.
\end{proof}

Lemma~\ref{lem:directional-exponential-spreading} defines the directional
linear spreading speed in the extended real line:
\begin{equation}\label{eq:directional-linear-spreading-speed}
    c_{\mathrm{sp}}(e)
    :=
    \lim_{t\to\infty}
    \frac{X_{\theta,e}^{\eta}(t)}{t}
    =+\infty.
\end{equation}
The value is independent of $\eta,\theta\in(0,1)$. The corresponding
finite logarithmic spreading rate is
\begin{equation}\label{eq:directional-logarithmic-speed}
    \lambda_{\mathrm{sp}}(e)
    :=
    \lim_{t\to\infty}
    \frac{1}{t}\log X_{\theta,e}^{\eta}(t)
    =
    \frac{a}{d+2s}.
\end{equation}
The same two limits hold for the directional level position of the solution
issued from any initial datum satisfying assumption \emph{(I)}. Indeed, the
lower and upper estimates in Theorem~\ref{thm:main} apply directly to that
solution. Repeating the proof of
Lemma~\ref{lem:directional-exponential-spreading}, with the same coordinate
points $x_t$, gives both directional limits.

\begin{lemma}
\label{lem:front-linear-upper-bound}
Suppose that a classical planar traveling front in direction $e$ has a
finite speed $c\in\R$. Then, for every
$\eta,\theta\in(0,1)$, there is a constant $C$ such that
\begin{equation}\label{eq:front-induced-level-bound}
    X_{\theta,e}^{\eta}(t)\leq ct+C,
    \qquad t\geq0.
\end{equation}
In particular,
\begin{equation}\label{eq:front-spreading-inequality}
    \limsup_{t\to\infty}
    \frac{X_{\theta,e}^{\eta}(t)}{t}
    \leq c.
\end{equation}
\end{lemma}

\begin{proof}
Write the front as $W(t,x)=\phi(x\cdot e-ct)$. Choose $j$ such that
$e_j\neq0$, set
\[
    v:=\operatorname{sgn}(e_j)\mathbf e_j,
    \qquad z_N:=Nv,
\]
and observe that $z_N\cdot e=N|e_j|\to+\infty$. By the first limit in
\eqref{eq:planar-front-limits}, one may choose $N$ so large that
\[
    \phi(-z_N\cdot e)\geq\eta.
\]
The integer translate
\[
    \widetilde W(t,x)
    :=
    W(t,x-z_N)
    =
    \phi(x\cdot e-ct-z_N\cdot e)
\]
is again an entire classical solution, since translation invariance follows
directly from \eqref{eq:jump-representation}. At time zero,
\[
    u^\eta(0,\cdot)\leq\widetilde W(0,\cdot).
\]
At $x=0$, this inequality follows from the choice of $N$; at $x\neq0$,
the left-hand side is zero and $\widetilde W(0,x)\geq0$. Applying
Lemma~\ref{lem:comparison} on an arbitrary finite time interval and then
letting its endpoint increase yields
\begin{equation}\label{eq:single-site-below-front}
    u^\eta(t,x)\leq\widetilde W(t,x),
    \qquad t\geq0,\quad x\in\Z^d.
\end{equation}

The number
\[
    M_\theta:=\sup\{\rho\in\R:\phi(\rho)\geq\theta\}
\]
is finite, by the two limits in \eqref{eq:planar-front-limits}. If
$u^\eta(t,x)\geq\theta$, then \eqref{eq:single-site-below-front}
implies
\[
    x\cdot e-ct-z_N\cdot e\leq M_\theta.
\]
Taking the supremum over the $\theta$-superlevel set gives
\[
    X_{\theta,e}^{\eta}(t)
    \leq ct+z_N\cdot e+M_\theta.
\]
This proves \eqref{eq:front-induced-level-bound}; division by $t$ gives
\eqref{eq:front-spreading-inequality}.
\end{proof}

\begin{theorem}
\label{thm:no-finite-speed-planar-front}
Assume \emph{(F)}. For every $e\in\mathbb S^{d-1}$, equation
\eqref{eq:intro-equation} admits no classical planar traveling front
with any finite speed $c\in\R$. Consequently,
\begin{equation}\label{eq:minimal-wave-equals-spreading-speed}
    \mathcal C_{\mathrm{tw}}(e)=\varnothing,
    \qquad
    c_{\mathrm{tw}}^*(e)
    =
    c_{\mathrm{sp}}(e)
    =
    +\infty.
\end{equation}
The corresponding finite directional logarithmic spreading rate is
\begin{equation}\label{eq:finite-logarithmic-spreading-rate}
    \lambda_{\mathrm{sp}}(e)
    =
    \frac{f'(0)}{d+2s}.
\end{equation}
\end{theorem}

\begin{proof}
Suppose, to the contrary, that a planar traveling front with a finite speed $c$ exists in direction $e$. Lemma~\ref{lem:front-linear-upper-bound} gives
\[
    \limsup_{t\to\infty}
    \frac{X_{\theta,e}^{\eta}(t)}{t}
    \leq c,
\]
whereas Lemma~\ref{lem:directional-exponential-spreading} gives
\[
    \lim_{t\to\infty}
    \frac{X_{\theta,e}^{\eta}(t)}{t}
    =+\infty.
\]
This contradiction excludes every $c\in\R$, including $c=0$ and
$c<0$. Hence $\mathcal C_{\mathrm{tw}}(e)=\varnothing$, and
\eqref{eq:extended-minimal-wave-speed} gives
$c_{\mathrm{tw}}^*(e)=+\infty$. Equations
\eqref{eq:directional-linear-spreading-speed} and
\eqref{eq:directional-logarithmic-speed} prove the remaining
assertions.
\end{proof}

\begin{remark}
\label{rem:meaning-infinite-threshold}
Equality \eqref{eq:minimal-wave-equals-spreading-speed} holds in the
extended real line. It records the failure of a constant-speed
traveling-front description and does not mean that a front exists at a
speed called $+\infty$. A profile with the reversed end states is also
excluded: replacing $(e,c,\phi(\rho))$ by
$(-e,-c,\phi(-\rho))$ reduces it to
Definition~\ref{def:planar-traveling-front}.
\end{remark}

\begin{remark}
\label{rem:no-exponential-moment}
The jump kernel has no positive exponential moment in any direction.
Indeed, for every $e\in\mathbb S^{d-1}$ and $\lambda>0$,
\begin{equation}\label{eq:no-exponential-moment}
    \sum_{z\neq0}K_s(z)\exp(\lambda z\cdot e)=+\infty.
\end{equation}
To see this, choose $j$ with $e_j\neq0$ and take
$z_n=n\operatorname{sgn}(e_j)\mathbf e_j$. Then
\eqref{eq:jump-kernel-bounds} gives
\[
    K_s(z_n)\exp(\lambda z_n\cdot e)
    \geq
    c(1+n)^{-d-2s}\exp(\lambda n|e_j|),
\]
so the corresponding subseries diverges. Thus the usual finite-valued
real exponential dispersion relation used for exponentially integrable
KPP kernels is unavailable here. This observation is consistent with,
but is not needed for, the comparison proof above.
\end{remark}

\section{The local regime}\label{tr:sec:local}

The boundedness of the lattice operator gives a direct comparison between
the fractional and local dynamics.

\begin{lemma}\label{tr:lem:operator-convergence}
For $0<s<1$,
\begin{equation}\label{tr:eq:operator-norm}
  \|A^s-A\|_{\ell^\infty\to\ell^\infty}\leq2(1-s).
\end{equation}
If $u_s$ and $u_1$ solve \eqref{eq:intro-equation} and
\eqref{tr:eq:local-equation}, respectively, with the same initial datum, then
\begin{equation}\label{tr:eq:solution-stability}
  \|u_s(t)-u_1(t)\|_\infty
  \leq\frac{2(1-s)}{a}\bigl(\ee^{at}-1\bigr),
  \qquad t\geq0.
\end{equation}
\end{lemma}

\begin{proof}
The binomial series converges in operator norm and gives
\begin{equation}\label{tr:eq:binomial}
  (I-P)^s=I-sP-\sum_{n=2}^\infty b_{n,s}P^n,
  \qquad
  b_{n,s}:=-(-1)^n\binom{s}{n}>0.
\end{equation}
Since
\[
  b_{n,s}=\frac{s\Gamma(n-s)}{\Gamma(1-s)\Gamma(n+1)},
\]
the series is summable. Evaluating its scalar counterpart at $1$ yields
$\sum_{n\geq2}b_{n,s}=1-s$. Hence
\[
  \|A^s-A\|
  \leq(1-s)\|P\|+\sum_{n=2}^\infty b_{n,s}\|P^n\|
  \leq2(1-s).
\]

Set $w=u_s-u_1$. The semigroup generated by $-A^s$ is a positive
contraction on $\ell^\infty$. Variation of constants and the fact that
$r\mapsto ar(1-r)$ is $a$-Lipschitz on $[0,1]$ give
\[
  \|w(t)\|_\infty
  \leq2(1-s)t+a\int_0^t\|w(r)\|_\infty\,dr.
\]
The integral form of Gronwall's inequality proves
\eqref{tr:eq:solution-stability}.
\end{proof}

We use the following local spreading result. Its formulation in terms of
$\mathcal W_a$ makes the later comparison transparent.

\begin{lemma}\label{tr:lem:wulff-geometry}
The set $\mathcal W_a$ is a compact convex neighborhood of the origin and
\begin{equation}\label{tr:eq:wulff-half-space-representation}
  \mathcal W_a
  =\bigcap_{e\in\mathbb S^{d-1}}
  \{v\in\R^d:v\cdot e\leq c_a(e)\}.
\end{equation}
Moreover, each supporting value is attained:
\begin{equation}\label{tr:eq:wulff-support-value}
  \max_{v\in\mathcal W_a}v\cdot e=c_a(e),
  \qquad e\in\mathbb S^{d-1}.
\end{equation}
If $h_t$ is the convolution kernel of $\ee^{-tA}$, then
\begin{equation}\label{tr:eq:local-walk-mgf}
  \sum_{x\in\Z^d}h_t(x)\ee^{\xi\cdot x}
  =\exp\left[t\left(-1+\frac1d\sum_{j=1}^d\cosh\xi_j\right)\right].
\end{equation}
Consequently, for every compact set
$F\subset\R^d\setminus\mathcal W_a$ there is $c_F>0$ such that
\begin{equation}\label{tr:eq:local-linear-large-deviation}
  \sup_{\substack{x\in\Z^d\\ x/t\in F}}\ee^{at}h_t(x)
  \leq C_F\ee^{-c_Ft},
  \qquad t\geq1.
\end{equation}
\end{lemma}

\begin{proof}
The function $\Lambda_a$ is smooth, even, and strictly convex. It tends to
infinity faster than linearly, while $\Lambda_a(0)=a$. Its Legendre
transform $I_a$ is therefore strictly convex and coercive. Since
$I_a(0)=-a<0$, the zero sublevel set is compact and contains the origin in
its interior.

By definition, $I_a(v)\leq0$ precisely when
\[
  \xi\cdot v\leq\Lambda_a(\xi)
  \qquad\text{for every }\xi\in\R^d.
\]
Writing a nonzero $\xi$ as $\lambda e$, with $\lambda>0$ and
$e\in\mathbb S^{d-1}$, this condition becomes
\[
  v\cdot e\leq
  \inf_{\lambda>0}\frac{\Lambda_a(\lambda e)}\lambda=c_a(e),
\]
which proves \eqref{tr:eq:wulff-half-space-representation}.

To prove \eqref{tr:eq:wulff-support-value}, set
\[
  \psi(\xi):=-1+\frac1d\sum_{j=1}^d\cosh\xi_j,
  \qquad \Lambda_a=a+\psi.
\]
For fixed $e\in\mathbb S^{d-1}$, the function
$\lambda\mapsto\Lambda_a(\lambda e)/\lambda$ tends to infinity both as
$\lambda\downarrow0$ and as $\lambda\to\infty$. It therefore has a minimizer
$\lambda_e>0$. With $v_e:=\nabla\psi(\lambda_e e)$, the critical-point
identity is
\[
  \lambda_e e\cdot v_e=\Lambda_a(\lambda_e e).
\]
Since $v_e=\nabla\Lambda_a(\lambda_e e)$, Legendre duality gives
\[
  I_a(v_e)=\lambda_e e\cdot v_e-\Lambda_a(\lambda_e e)=0.
\]
Thus $v_e\in\mathcal W_a$ and
$v_e\cdot e=\Lambda_a(\lambda_e e)/\lambda_e=c_a(e)$. The reverse inequality
follows from \eqref{tr:eq:wulff-half-space-representation}, proving
\eqref{tr:eq:wulff-support-value}.

For \eqref{tr:eq:local-walk-mgf}, use
$\ee^{-tA}=\ee^{-t}\ee^{tP}$. A single jump has moment generating
function $d^{-1}\sum_j\cosh\xi_j$, and the number of jumps is Poisson with
mean $t$. Chernoff's bound now gives, for every $\xi\in\R^d$,
\begin{equation}\label{tr:eq:local-chernoff}
  \ee^{at}h_t(x)
  \leq\exp\bigl(t\Lambda_a(\xi)-\xi\cdot x\bigr).
\end{equation}
If $v\notin\mathcal W_a$, some $\xi$ satisfies
$\xi\cdot v-\Lambda_a(\xi)>0$. Compactness of $F$ provides finitely many
such $\xi$ and a common positive gap. Applying
\eqref{tr:eq:local-chernoff} on the resulting finite cover proves
\eqref{tr:eq:local-linear-large-deviation}.
\end{proof}

\begin{proposition}\label{tr:prop:local-spreading}
Let $u_1$ solve \eqref{tr:eq:local-equation}.

If $u_0$ is nonzero and finitely supported, then, for every
$\delta\in(0,1)$,
\begin{align}
  \lim_{t\to\infty}
  \inf_{x\in(1-\delta)t\mathcal W_a\cap\Z^d}u_1(t,x)&=1,
  \label{tr:eq:local-wulff-inner}\\
  \lim_{t\to\infty}
  \sup_{\substack{x\in\Z^d\\
  x\notin(1+\delta)t\mathcal W_a}}u_1(t,x)&=0.
  \label{tr:eq:local-wulff-outer}
\end{align}
If \eqref{tr:eq:front-data} holds, then, for every
$\delta\in(0,c_a(e))$,
\begin{align}
  \lim_{t\to\infty}
  \sup_{\substack{x\in\Z^d\\ x\cdot e\leq-(c_a(e)+\delta)t}}
  u_1(t,x)&=0,
  \label{tr:eq:local-front-outer}\\
  \lim_{t\to\infty}
  \inf_{\substack{x\in\Z^d\\ x\cdot e\geq-(c_a(e)-\delta)t}}
  u_1(t,x)&=1.
  \label{tr:eq:local-front-inner}
\end{align}
\end{proposition}

\begin{proof}
Let $Q$ denote the time-one solution map on the order interval
$[0,1]^{\Z^d}$, equipped with the topology of uniform convergence on finite
subsets of $\Z^d$. The comparison principle makes $Q$ order preserving, and
translation invariance of the equation implies that $Q$ commutes with every
lattice translation. Continuity in the local topology follows by writing the
variation-of-constants formula, first restricting the nearest-neighbor heat
kernel to a large finite set and then using that its remaining mass tends to
zero uniformly on bounded time intervals. The compactness condition in this
topology is automatic on the discrete habitat: every bounded sequence has a
subsequence converging on each finite subset, by a diagonal argument. The
nearest-neighbor heat kernel is positive at every lattice point for positive
time, so a nonzero nonnegative datum is mapped to a strictly positive
function. The concavity of $r\mapsto ar(1-r)$ also gives
\begin{equation}\label{tr:eq:local-map-subhomogeneity}
  Q(\theta\phi)\geq\theta Q(\phi),
  \qquad 0\leq\theta\leq1,\qquad 0\leq\phi\leq1.
\end{equation}
Indeed, if $u$ starts from $\phi$, then $\theta u$ is a subsolution with
initial value $\theta\phi$. Finally, the constant states are governed by
$q'=aq(1-q)$, so that $0$ is unstable and every constant initial value in
$(0,1]$ converges to $1$.

These facts verify the hypotheses of the multidimensional spreading theorem
for translation-invariant monotone recursions in \cite{Weinberger1982}; the
monotone-semiflow formulation is given in \cite{LiangZhao2010}. It remains to
identify the associated linear speed. The derivative of $Q$ at zero is
$\ee^{a-A}$, and
\[
  \ee^{a-A}\bigl(\ee^{-\lambda x\cdot e}\bigr)
  =\exp\bigl(\Lambda_a(\lambda e)\bigr)
   \ee^{-\lambda x\cdot e}.
\]
The KPP inequalities
\[
  ar(1-r)\leq ar,
  \qquad
  ar(1-r)\geq(a-\eta)r
\]
on a sufficiently small interval give the matching upper and lower
linearizations. The directional speed is therefore
$\inf_{\lambda>0}\Lambda_a(\lambda e)/\lambda=c_a(e)$. The multidirectional
part of the same spreading theorem identifies the invasion set as the
intersection of the corresponding directional half-spaces. By
Lemma~\ref{tr:lem:wulff-geometry}, this intersection is
$\mathcal W_a=\{I_a\leq0\}$. This proves
\eqref{tr:eq:local-wulff-inner}--\eqref{tr:eq:local-wulff-outer}.

We derive the front-like limits from these compact-support spreading
estimates. Since $ar(1-r)\leq ar$ for $r\in[0,1]$, comparison with the
linearized equation and the upper bound in \eqref{tr:eq:front-data} give
\[
  u_1(t,x)\leq \ee^{at}\bigl(h_t*H_{e,\kappa_-}\bigr)(x).
\]
For every $\lambda>0$, symmetry of $h_t$ and
\eqref{tr:eq:local-walk-mgf} imply
\[
  \ee^{at}\bigl(h_t*H_{e,\kappa_-}\bigr)(x)
  \leq \exp\bigl(t\Lambda_a(\lambda e)
       +\lambda(x\cdot e-\kappa_-)\bigr).
\]
Choose $\lambda>0$ such that
$\Lambda_a(\lambda e)/\lambda<c_a(e)+\delta/2$. If
$x\cdot e\leq-(c_a(e)+\delta)t$, the last display is bounded by
$\exp(-\lambda\delta t/2-\lambda\kappa_-)$, which proves
\eqref{tr:eq:local-front-outer}.

For the lower bound, let $v_e\in\mathcal W_a$ be the point constructed in
the proof of Lemma~\ref{tr:lem:wulff-geometry}. By symmetry,
$-v_e\in\mathcal W_a$, and $(-v_e)\cdot e=-c_a(e)$. Choose
$\eta\in(0,1)$ so small that $\eta c_a(e)<\delta/2$, and set
$v:=-(1-\eta)v_e$. Since the origin is an interior point of
$\mathcal W_a$, we have $v\in\operatorname{int}\mathcal W_a$. Choose
$k_t\in\Z^d$ with $|k_t-tv|\leq\sqrt d$. There is $\rho\in(0,1)$,
independent of $t$, such that
\[
  k_t\in(1-\rho)t\mathcal W_a
\]
for all sufficiently large $t$.

Let $w$ solve \eqref{tr:eq:local-equation} with initial datum
$w(0)=\mu\delta_0$. The first part of the proposition yields
$w(t,k_t)\to1$. If
$x\cdot e\geq-(c_a(e)-\delta)t$, put $y=x-k_t$. The choice of $v$ gives
\[
  y\cdot e
  \geq \bigl(\delta-\eta c_a(e)\bigr)t-\sqrt d
  \geq \frac\delta2t-\sqrt d.
\]
Hence $y\cdot e\geq\kappa_+$ for all sufficiently large $t$, and
$\mu\delta_y\leq u_0$ by \eqref{tr:eq:front-data}. Translation invariance and
comparison now give
\[
  u_1(t,x)\geq w(t,x-y)=w(t,k_t).
\]
The right-hand side is independent of $x$ and tends to one. This proves
\eqref{tr:eq:local-front-inner}.
\end{proof}

The same estimate gives convergence on a growing time interval, a form that
is useful when the observation time is not fixed in advance.

\begin{corollary}\label{tr:cor:uniform-local-window}
Let $T_s\geq0$ satisfy
\begin{equation}\label{tr:eq:uniform-local-window-condition}
  \tau_s-aT_s\to\infty
  \qquad\text{as }s\uparrow1.
\end{equation}
For any initial datum $0\leq u_0\leq1$ independent of $s$,
\begin{equation}\label{tr:eq:uniform-local-window-convergence}
  \sup_{0\leq t\leq T_s}
  \|u_s(t)-u_1(t)\|_\infty\to0.
\end{equation}
\end{corollary}

\begin{proof}
The right-hand side of \eqref{tr:eq:solution-stability} is increasing in $t$.
At $t=T_s$ it is bounded by
\[
  \frac2a\exp(-\tau_s+aT_s),
\]
which tends to zero by \eqref{tr:eq:uniform-local-window-condition}.
\end{proof}

\section{Estimates uniform near \texorpdfstring{$s=1$}{s=1}}
\label{tr:sec:uniform}

Fix once and for all $s_*\in(1/2,1)$. Unless otherwise stated, the constants
in this section are uniform for $s\in[s_*,1)$.

This section isolates the estimates needed after the local regime. All
constants are uniform near $s=1$, and the factor $1-s$ is kept explicit in
every term arising from a long jump.

Let $p_n(z)$ be the convolution kernel of $P^n$. Formula
\eqref{tr:eq:binomial} gives, for $z\neq0$,
\begin{equation}\label{tr:eq:jump-decomposition}
  K_s(z)=s p_1(z)+\widetilde K_s(z),
  \qquad
  \widetilde K_s(z):=\sum_{n=2}^\infty b_{n,s}p_n(z),
\end{equation}
Let
\[
  S_s(t):=\ee^{-tA^s},
  \qquad
  (S_s(t)\phi)(x)=\sum_{y\in\Z^d}G_{s,t}(x-y)\phi(y).
\]

For $R\geq1$, $e\in\mathbb S^{d-1}$, and $\kappa\in\R$, define
\begin{align}
  \Theta_{R,s}(x)
  &:=\frac1{1+(|x|/R)^{d+2s}},
  \label{tr:eq:radial-profile}\\
  \Psi_{R,e,\kappa,s}(x)
  &:=\left(1+\frac{((x\cdot e-\kappa)^-)^2}{R^2}\right)^{-s},
  \qquad r^-:=\max\{-r,0\}.
  \label{tr:eq:front-profile}
\end{align}

The transition argument requires estimates that remain uniform in $s$.
In particular, the factor $1-s$ cannot be absorbed into an unspecified
constant, because it determines the transition time.

\subsection{The long-jump kernel}

Lemma~\ref{lem:one-dimensional-walk} and \eqref{eq:product-kernel} provide
the required nearest-neighbor estimates. We refine the subordinated kernel
bound so that its constants remain uniform as $s\uparrow1$.

The binomial decomposition \eqref{tr:eq:jump-decomposition} is especially
convenient for short jumps.  For long jumps, the subordinated form
\begin{equation}\label{tr:eq:subordinated-long-kernel}
  K_s(z)=c_s\int_0^\infty h_r(z)\frac{dr}{r^{1+s}},
  \qquad c_s=\frac{s}{\Gamma(1-s)},
  \qquad z\neq0,
\end{equation}
makes the coefficient $1-s$ visible.  Notice that
\begin{equation}\label{tr:eq:cs-comparison}
  c(1-s)\leq c_s\leq C(1-s),
  \qquad s_*\leq s<1.
\end{equation}

\begin{lemma}\label{tr:lem:uniform-long-jumps}
For $s\in[s_*,1)$ and $z\neq0$,
\begin{equation}\label{tr:eq:long-jump-two-sided}
  c(1-s)(1+|z|)^{-d-2s}
  \leq\widetilde K_s(z)
  \leq C(1-s)(1+|z|)^{-d-2s}.
\end{equation}
Moreover, if $\kappa_s=\sum_{z\neq0}K_s(z)$, then
\begin{equation}\label{tr:eq:jump-rate-bounds}
  0<c\leq\kappa_s\leq C<\infty.
\end{equation}
\end{lemma}

\begin{proof}
First suppose $n=|z|_1\geq2$.  By
Lemma~\ref{lem:one-dimensional-walk} and
\eqref{eq:product-kernel},
\[
  h_r(z)\geq cr^{-d/2}
  \quad\text{for}\quad C(1+n^2)\leq r\leq2C(1+n^2).
\]
Substitution in \eqref{tr:eq:subordinated-long-kernel}, together with
\eqref{tr:eq:cs-comparison}, gives
\begin{equation}\label{tr:eq:long-kernel-lower-calculation}
  K_s(z)\geq c(1-s)
  \int_{C(1+n^2)}^{2C(1+n^2)}r^{-1-s-d/2}\,dr
  \geq c(1-s)(1+n)^{-d-2s}.
\end{equation}
Since $p_1(z)=0$ for $n\geq2$, this is the required lower bound for
$\widetilde K_s$.

For the upper bound, choose $\alpha>0$ with $\ee\alpha<1$.  When
$0<r\leq\alpha n$, the walk needs at least $n$ jumps to reach $z$, so the
Poisson tail estimate gives
\[
  h_r(z)\leq\mathbb P\{\operatorname{Poisson}(r)\geq n\}
  \leq\left(\frac{\ee r}{n}\right)^n.
\]
Consequently,
\begin{equation}\label{tr:eq:small-time-kernel-integral}
  c_s\int_0^{\alpha n}h_r(z)\frac{dr}{r^{1+s}}
  \leq C(1-s)(\ee\alpha)^n n^{-1-s}
  \leq C(1-s)(1+n)^{-d-2s}.
\end{equation}
For $r\geq\alpha n$, equations \eqref{eq:one-dimensional-upper} and
\eqref{eq:product-kernel} imply
\[
  h_r(z)\leq Cr^{-d/2}\exp(-cn^2/r).
\]
After the change of variables $\rho=n^2/r$,
\begin{align}
  c_s\int_{\alpha n}^\infty h_r(z)\frac{dr}{r^{1+s}}
  &\leq C(1-s)n^{-d-2s}
  \int_0^{n/\alpha}\rho^{s+d/2-1}\ee^{-c\rho}\,d\rho \notag\\
  &\leq C(1-s)n^{-d-2s}.
  \label{tr:eq:large-time-kernel-integral}
\end{align}
This proves the upper bound when $|z|_1\geq2$.

It remains to treat the finitely many nearest-neighbor points.  The upper
bound follows from $\sum_{n\geq2}b_{n,s}=1-s$.  For the lower bound, a
nearest-neighbor point can be reached in three steps, and hence
\[
  \widetilde K_s(z)\geq b_{3,s}p_3(z)\geq c(1-s)
  \qquad (|z|_1=1).
\]
This completes \eqref{tr:eq:long-jump-two-sided}.  Finally,
$K_s=s p_1+\widetilde K_s$, the total mass of $p_1$ is one, and the total
mass of $\widetilde K_s$ is at most $1-s$.  These facts give
\eqref{tr:eq:jump-rate-bounds}.
\end{proof}

\subsection{A two-scale decomposition before the transition time}

The full jump rate stays bounded away from zero as $s\uparrow1$, since its
nearest-neighbor part converges to the original walk. Only the nonlocal part
becomes rare. Separating these two components replaces the global
operator-norm estimate near the endpoint $at=\tau_s$.

Set
\begin{equation}\label{tr:eq:long-jump-rate}
  \widetilde\kappa_s:=\sum_{z\neq0}\widetilde K_s(z),
  \qquad
  \widetilde j_s(z):=\widetilde\kappa_s^{-1}\widetilde K_s(z).
\end{equation}
Lemma~\ref{tr:lem:uniform-long-jumps} and
$\sum_{n\geq2}b_{n,s}=1-s$ imply
\begin{equation}\label{tr:eq:long-jump-rate-comparison}
  c(1-s)\leq\widetilde\kappa_s\leq1-s,
  \qquad
  \widetilde j_s(z)\leq C(1+|z|)^{-d-2s}.
\end{equation}
Let $H_{s,t}(z):=h_{st}(z)$, the heat kernel of the nearest-neighbor walk
with total rate $s$, and let
\begin{equation}\label{tr:eq:rare-jump-kernel}
  Q_{s,t}:=\ee^{-\widetilde\kappa_st}
  \sum_{\ell=0}^\infty
  \frac{(\widetilde\kappa_st)^\ell}{\ell!}
  \widetilde j_s^{*\ell}.
\end{equation}

\begin{lemma}\label{tr:lem:two-scale-kernel}
For $s\in[s_*,1)$ and $t>0$,
\begin{equation}\label{tr:eq:two-scale-factorization}
  G_{s,t}=H_{s,t}*Q_{s,t}.
\end{equation}
If $\widetilde\kappa_st\leq1$, then, with $q_s=d+2s$,
\begin{align}
  Q_{s,t}(z)&\leq C(1-s)t(1+|z|)^{-q_s},
  &&z\neq0,
  \label{tr:eq:rare-jump-pointwise}\\
  \sum_{|z|\geq R}Q_{s,t}(z)
  &\leq C(1-s)tR^{-2s},
  &&R\geq1,
  \label{tr:eq:rare-jump-radial-tail}\\
  \sum_{z\cdot e\leq-R}Q_{s,t}(z)
  &\leq C_e(1-s)tR^{-2s},
  &&R\geq1.
  \label{tr:eq:rare-jump-half-space-tail}
\end{align}
Moreover,
\begin{equation}\label{tr:eq:rare-jump-probability}
  \sum_{z\neq0}Q_{s,t}(z)\leq\widetilde\kappa_st.
\end{equation}
\end{lemma}

\begin{proof}
The jump representation and \eqref{tr:eq:jump-decomposition} give
\[
  -A^s=s(P-I)+
  \sum_{z\neq0}\widetilde K_s(z)(T_z-I),
  \qquad (T_z\phi)(x)=\phi(x+z).
\]
The two convolution generators commute. Exponentiating them proves
\eqref{tr:eq:two-scale-factorization}.

To estimate the convolution powers, consider a decomposition
$z=z_1+\cdots+z_\ell$. At least one increment satisfies
$|z_i|\geq |z|/\ell$. We use
$\widetilde j_s(y)\leq C(1+|y|)^{-q_s}$ for that increment, sum over its
position, and then sum the remaining $\ell-1$ variables, whose total mass is
one. This gives
\begin{equation}\label{tr:eq:rare-jump-convolution}
  \widetilde j_s^{*\ell}(z)
  \leq C\ell^{q_s+1}(1+|z|)^{-q_s}.
\end{equation}
Write $\lambda=\widetilde\kappa_st$. For $0<\lambda\leq1$, the Poisson
moment in \eqref{tr:eq:rare-jump-kernel} satisfies
\[
  \ee^{-\lambda}\sum_{\ell\geq1}
  \frac{\lambda^\ell}{\ell!}\ell^{q_s+1}
  \leq C\lambda,
\]
uniformly for $q_s\in[d+2s_*,d+2]$. Equations
\eqref{tr:eq:long-jump-rate-comparison} and
\eqref{tr:eq:rare-jump-convolution} prove
\eqref{tr:eq:rare-jump-pointwise}. Summing this estimate over radial shells
or over a lattice half-space gives
\eqref{tr:eq:rare-jump-radial-tail} and
\eqref{tr:eq:rare-jump-half-space-tail}. Finally, the probability that the
compound-Poisson process makes at least one long jump is
$1-\ee^{-\lambda}\leq\lambda$. This proves
\eqref{tr:eq:rare-jump-probability}; returns to the origin can only decrease
the left-hand side.
\end{proof}

The nearest-neighbor factor has the same large-deviation geometry as the
local equation. The only change is its jump rate $s$, which converges to
one.

\begin{lemma}\label{tr:lem:uniform-nearest-neighbor-deviation}
Fix $\delta>0$. There are $s_\delta<1$ and $c_\delta,C_\delta>0$ such
that, for $s\in[s_\delta,1)$ and $t\geq1$,
\begin{align}
  \sup_{\substack{z\in\Z^d\\
  z\notin(1+\delta)t\mathcal W_a}}
  \ee^{at}H_{s,t}(z)&\leq C_\delta\ee^{-c_\delta t},
  \label{tr:eq:uniform-nn-wulff-tail}\\
  \ee^{at}\sum_{z\cdot e\leq-(c_a(e)+\delta)t}H_{s,t}(z)
  &\leq C_{\delta,e}\ee^{-c_{\delta,e}t}.
  \label{tr:eq:uniform-nn-directional-tail}
\end{align}
For every $b,T_0>0$ there are $c,C>0$ such that
\begin{equation}\label{tr:eq:uniform-nn-ballistic-concentration}
  \sum_{|z|\geq bt}H_{s,\ell}(z)\leq C\ee^{-ct}
\end{equation}
whenever $s\in[s_*,1)$, $t\geq T_0$, and $0\leq\ell\leq T_0^{-1}t$.
The same conclusion holds with $|z|\geq bt$ replaced by
$|z\cdot e|\geq bt$.
\end{lemma}

\begin{proof}
The moment generating function of $H_{s,t}$ is
\begin{equation}\label{tr:eq:rate-s-mgf}
  \sum_{z\in\Z^d}H_{s,t}(z)\ee^{\xi\cdot z}
  =\exp\left[st\left(-1+\frac1d
  \sum_{j=1}^d\cosh\xi_j\right)\right].
\end{equation}
After multiplication by $\ee^{at}$, its exponent is
\[
  \Lambda_{a,s}(\xi)
  :=a+s\left(-1+\frac1d\sum_{j=1}^d\cosh\xi_j\right).
\]
This function converges locally uniformly to $\Lambda_a$. On a fixed ball
in the velocity variable, the finite-cover argument in
Lemma~\ref{tr:lem:wulff-geometry} is therefore uniform for $s$ close to one.
Outside a sufficiently large ball, choose a coordinate $j$ with
$|z_j|\geq|z|/\sqrt d$ and apply \eqref{tr:eq:rate-s-mgf} with
$\xi=\lambda\operatorname{sgn}(z_j)\mathbf e_j$, where $\lambda>0$ is
fixed. The term $-\xi\cdot z$ then dominates both $at$ and the
moment-generating term. These two regions prove
\eqref{tr:eq:uniform-nn-wulff-tail}. For the directional estimate, choose
$\lambda>0$ such that
\[
  \Lambda_a(\lambda e)-\lambda(c_a(e)+\delta)<0.
\]
Exponential Markov inequality applied to $-z\cdot e$, followed by local
uniform convergence of $\Lambda_{a,s}$, gives
\eqref{tr:eq:uniform-nn-directional-tail}.

For \eqref{tr:eq:uniform-nn-ballistic-concentration}, apply exponential
Markov inequality to each coordinate with a fixed sufficiently small
parameter. Since $\ell\leq T_0^{-1}t$, the negative contribution from the
threshold $bt$ dominates the moment-generating term after the parameter is
chosen in terms of $b$ and $T_0$. A union bound over the coordinates proves
the radial estimate; the projected estimate follows directly from
\eqref{tr:eq:rate-s-mgf}.
\end{proof}

\begin{proposition}\label{tr:prop:early-linear-estimates}
Let $s\uparrow1$ and $t\to\infty$ under the restriction $at<\tau_s$.
If $\phi\geq0$ is finitely supported, then, for every $\delta>0$,
\begin{equation}\label{tr:eq:early-linear-localized-upper}
  \sup_{\substack{x\in\Z^d\\
  x\notin(1+\delta)t\mathcal W_a}}
  \ee^{at}(S_s(t)\phi)(x)\to0.
\end{equation}
For every $e\in\mathbb S^{d-1}$, $\kappa\in\R$, and $\delta>0$,
\begin{equation}\label{tr:eq:early-linear-front-upper}
  \sup_{\substack{x\in\Z^d\\ x\cdot e\leq-(c_a(e)+\delta)t}}
  \ee^{at}S_s(t)H_{e,\kappa}(x)\to0.
\end{equation}
\end{proposition}

\begin{proof}
Since $at<\tau_s$,
\begin{equation}\label{tr:eq:pretransition-small-intensity}
  (1-s)\ee^{at}\leq1,
  \qquad
  (1-s)t\leq a^{-1}(1-s)\tau_s\to0.
\end{equation}
In particular, $\widetilde\kappa_st\leq1$ for $s$ close to one.

It suffices first to take $\phi=\delta_0$. By
\eqref{tr:eq:two-scale-factorization},
\[
  G_{s,t}(x)=\sum_{y\in\Z^d}H_{s,t}(y)Q_{s,t}(x-y).
\]
Split the sum according to whether
$y\notin(1+\delta/2)t\mathcal W_a$. The first part, multiplied by
$\ee^{at}$, is bounded by $C\ee^{-ct}$ by
\eqref{tr:eq:uniform-nn-wulff-tail}. Since $\mathcal W_a$ contains a
neighborhood of the origin, homothety and convexity give
\[
  \operatorname{dist}\bigl(
  \R^d\setminus(1+\delta)t\mathcal W_a,
  (1+\delta/2)t\mathcal W_a\bigr)\geq c_\delta t.
\]
For the remaining terms, \eqref{tr:eq:rare-jump-pointwise} yields
\[
  \ee^{at}Q_{s,t}(x-y)
  \leq C\ee^{at}(1-s)t^{1-d-2s}
  \leq Ct^{1-d-2s}\to0.
\]
Summation against $H_{s,t}$ proves the assertion for $\delta_0$. A finite
number of translations handles every finitely supported $\phi$.

For the half-space datum, write the convolution probabilistically as the
sum of an independent nearest-neighbor displacement $Y$ and a long-jump
displacement $Z$. A fixed shift $\kappa$ is negligible on the scale $t$.
If $Y+Z$ has projection at most $-(c_a(e)+\delta)t$, then either
\[
  Y\cdot e\leq-(c_a(e)+\delta/2)t
  \quad\text{or}\quad
  Z\cdot e\leq-\delta t/2
\]
for all large $t$. The first event is controlled by
\eqref{tr:eq:uniform-nn-directional-tail}. By
\eqref{tr:eq:rare-jump-half-space-tail}, the second contributes at most
\[
  C\ee^{at}(1-s)t(\delta t)^{-2s}
  \leq C_\delta t^{1-2s}\to0.
\]
This proves \eqref{tr:eq:early-linear-front-upper}.
\end{proof}

\begin{proof}[Proof of the early parts of
Theorems~\ref{tr:thm:localized-transition} and \ref{tr:thm:front-transition}]
The KPP inequality $au(1-u)\leq au$ gives
\begin{equation}\label{tr:eq:early-linear-comparison}
  u_s(t)\leq\ee^{at}S_s(t)u_0.
\end{equation}
For a finitely supported datum, the outer estimate
\eqref{tr:eq:early-localized-outer} follows from
\eqref{tr:eq:early-linear-localized-upper}. Under
\eqref{tr:eq:front-data}, comparison with $H_{e,\kappa_-}$ and
\eqref{tr:eq:early-linear-front-upper} proves
\eqref{tr:eq:early-front-outer}.

We next prove the localized inner estimate. Fix $\delta\in(0,1)$ and
choose $\gamma,\eta\in(0,1)$ so small that
\begin{equation}\label{tr:eq:localized-intermediate-choice}
  1-\frac\delta2<(1-\eta)(1-\gamma).
\end{equation}
Put $t_0=(1-\gamma)t$ and
\[
  K_t=(1-\delta/2)t\mathcal W_a\cap\Z^d.
\]
By \eqref{tr:eq:localized-intermediate-choice}, $K_t$ is contained in
$(1-\eta)t_0\mathcal W_a$ for all $t$. The local spreading theorem gives
\begin{equation}\label{tr:eq:local-intermediate-plateau}
  \inf_{x\in K_t}u_1(t_0,x)\to1.
\end{equation}
Although the global stability estimate is not small at the final time, it
is small at $t_0$. Indeed, $at<\tau_s$ implies
\begin{equation}\label{tr:eq:intermediate-time-stability}
  \|u_s(t_0)-u_1(t_0)\|_\infty
  \leq\frac2a\exp(-\tau_s+at_0)
  \leq\frac2a\exp(-\gamma\tau_s)\to0.
\end{equation}
Thus the left-hand side of \eqref{tr:eq:local-intermediate-plateau} remains
unchanged if $u_1$ is replaced by $u_s$.

Since $\mathcal W_a$ contains a Euclidean ball centered at the origin,
there is $b_\delta>0$ such that
\begin{equation}\label{tr:eq:wulff-interior-distance}
  \operatorname{dist}\bigl((1-\delta)t\mathcal W_a,
  \R^d\setminus(1-\delta/2)t\mathcal W_a\bigr)
  \geq b_\delta t.
\end{equation}
Let $\ell=t-t_0=\gamma t$. Because the reaction is nonnegative on
$[0,1]$, Duhamel's formula yields
\begin{equation}\label{tr:eq:retention-semigroup-lower}
  u_s(t)\geq S_s(\ell)u_s(t_0).
\end{equation}
The probability of at least one long jump during this interval is at most
\[
  \widetilde\kappa_s\ell\leq(1-s)t\to0
\]
by \eqref{tr:eq:pretransition-small-intensity}. If no long jump occurs,
\eqref{tr:eq:uniform-nn-ballistic-concentration} and
\eqref{tr:eq:wulff-interior-distance} show that a displacement large enough
to leave $K_t$ has probability $O(\ee^{-ct})$. Combining this with
\eqref{tr:eq:local-intermediate-plateau}--
\eqref{tr:eq:retention-semigroup-lower} gives
\[
  \inf_{x\in(1-\delta)t\mathcal W_a\cap\Z^d}u_s(t,x)\to1,
\]
which is \eqref{tr:eq:early-localized-inner}.

For front-like data, fix $\delta\in(0,c_a(e))$ and choose
$\gamma>0$ and $\eta\in(0,\delta/2)$ so that
\begin{equation}\label{tr:eq:front-intermediate-choice}
  c_a(e)-\frac\delta2
  <(c_a(e)-\eta)(1-\gamma).
\end{equation}
At time $t_0=(1-\gamma)t$, the local front estimate and
\eqref{tr:eq:intermediate-time-stability} imply
\begin{equation}\label{tr:eq:front-intermediate-plateau}
  \inf_{\substack{x\in\Z^d\\
  x\cdot e\geq-(c_a(e)-\delta/2)t}}u_s(t_0,x)\to1.
\end{equation}
The target half-space
$\{x:x\cdot e\geq-(c_a(e)-\delta)t\}$ lies at distance
$\delta t/2$ from the complementary boundary in
\eqref{tr:eq:front-intermediate-plateau}. During the remaining time
$\gamma t$, the probability of a long jump tends to zero, while
\eqref{tr:eq:uniform-nn-ballistic-concentration} makes the probability of a
nearest-neighbor displacement of projected length $\delta t/2$
exponentially small. Applying \eqref{tr:eq:retention-semigroup-lower} once
more proves \eqref{tr:eq:early-front-inner}.
\end{proof}

\begin{remark}\label{tr:rem:why-two-scale}
The intermediate time $t_0=(1-\gamma)t$ is needed only for the inner
estimates. At the final time, Lemma~\ref{tr:lem:operator-convergence} gives an
error of order $(1-s)\ee^{at}$, which need not vanish when
$at=\tau_s-O(1)$. At time $t_0$ the same error is at most
$C\ee^{-\gamma\tau_s}$. The remaining interval is handled by the
retention estimate \eqref{tr:eq:retention-semigroup-lower}, where the relevant
quantity is the unamplified long-jump intensity $(1-s)t\to0$. This is the
step that extends the early regime from $\tau_s-at\to\infty$ to the full
range $at<\tau_s$.
\end{remark}

\subsection{The heat kernel}

The semigroup generated by $-A^s$ is that of a compound-Poisson process. We
retain the proof in order to track both the coefficient of the one-jump term
and the uniformity of the convolution estimate.

\begin{lemma}\label{tr:lem:uniform-heat-kernel}
There are $c,C>0$ and an integer $m\geq d+4$ such that, for
$s\in[s_*,1)$,
\begin{align}
  G_{s,1}(z)&\geq c(1-s)(1+|z|)^{-d-2s},
  &&z\in\Z^d,
  \label{tr:eq:heat-one-jump-lower}\\
  G_{s,t}(z)&\leq C(t+t^m)(1+|z|)^{-d-2s},
  &&t>0,\quad z\neq0.
  \label{tr:eq:heat-polynomial-upper}
\end{align}
\end{lemma}

\begin{proof}
Put $j_s(0)=0$ and $j_s(z)=\kappa_s^{-1}K_s(z)$ for $z\neq0$.  Then
\begin{equation}\label{tr:eq:compound-poisson-detailed}
  G_{s,t}=\ee^{-\kappa_st}
  \sum_{\ell=0}^\infty\frac{(\kappa_st)^\ell}{\ell!}j_s^{*\ell}.
\end{equation}
For $z\neq0$, the term $\ell=1$ and
Lemma~\ref{tr:lem:uniform-long-jumps} prove
\eqref{tr:eq:heat-one-jump-lower}.  At $z=0$, the zero-jump term is bounded
below by a positive constant, which is stronger than the stated estimate.

We claim that, with $q_s=d+2s$,
\begin{equation}\label{tr:eq:uniform-convolution-bound}
  j_s^{*\ell}(z)
  \leq C\ell^{q_s+1}(1+|z|)^{-q_s},
  \qquad \ell\geq1,\quad z\neq0.
\end{equation}
Indeed, in every decomposition $z=z_1+\cdots+z_\ell$, at least one
increment has length at least $|z|/\ell$.  After choosing this increment,
the other $\ell-1$ variables determine it.  Since
$j_s(y)\leq C(1+|y|)^{-q_s}$ uniformly and the other factors have total
mass one,
\[
  j_s^{*\ell}(z)
  \leq C\ell(1+|z|/\ell)^{-q_s}
  \leq C\ell^{q_s+1}(1+|z|)^{-q_s}.
\]
The constants remain uniform because $q_s\in[d+2s_*,d+2]$.

Choose an integer $m\geq d+4$.  Since $q_s+1\leq d+3\leq m$, insertion
of \eqref{tr:eq:uniform-convolution-bound} into
\eqref{tr:eq:compound-poisson-detailed} gives
\[
  G_{s,t}(z)\leq C(1+|z|)^{-q_s}
  \mathbb E[N_{\kappa_st}^{m}],
\]
where $N_\lambda$ is Poisson with mean $\lambda$.  Its $m$th moment is a
polynomial in $\lambda$ without constant term.  The bounds
\eqref{tr:eq:jump-rate-bounds} therefore imply
$\mathbb E[N_{\kappa_st}^{m}]\leq C(t+t^m)$, proving
\eqref{tr:eq:heat-polynomial-upper}.
\end{proof}

\subsection{Uniform lattice tails}

Both radial and one-sided estimates follow from the same elementary lattice
sum. They are stated separately because the one-sided estimate accounts for
the absence of $d$ from the front-like exponent.

\begin{lemma}\label{tr:lem:uniform-lattice-sums}
There are constants $c,C>0$, uniform for $s\in[s_*,1)$, such that
\begin{align}
  \sum_{|z|\geq R}(1+|z|)^{-d-2s}
  &\leq CR^{-2s},
  \label{tr:eq:uniform-radial-tail-sum}\\
  \sum_{1\leq|z|\leq R}|z|^{2-d-2s}
  &\leq C\frac{R^{2-2s}}{1-s},
  \label{tr:eq:uniform-second-moment-sum}
\end{align}
for $R\geq1$.  For every $e\in\mathbb S^{d-1}$ there are
$c_e,C_e>0$ such that
\begin{equation}\label{tr:eq:uniform-half-space-tail}
  c_eR^{-2s}
  \leq\sum_{z\cdot e\leq-R}(1+|z|)^{-d-2s}
  \leq C_eR^{-2s},
  \qquad R\geq1.
\end{equation}
\end{lemma}

\begin{proof}
The first two bounds follow by decomposing $\Z^d$ into shells
$n\leq|z|<n+1$ and comparing the resulting sums with radial integrals.
For \eqref{tr:eq:uniform-second-moment-sum}, the integral is
$(R^{2-2s}-1)/(2-2s)$; this is the only apparent singularity as
$s\uparrow1$.

The upper bound in \eqref{tr:eq:uniform-half-space-tail} follows from
$z\cdot e\leq-R\Rightarrow |z|\geq R$.  For the reverse inequality,
choose $j$ with $|e_j|\geq d^{-1/2}$ and put
$v=-\operatorname{sgn}(e_j)\mathbf e_j$.  For each large $R$, choose an
integer $n\asymp R/|e_j|$ with $n|e_j|\geq2R$.  A lattice cube of side
$c_en$ centered at $nv$ lies in $\{z:z\cdot e\leq-R\}$, contains at
least $c_en^d$ points, and satisfies $|z|\leq C_en$ on that cube.  Its
contribution is bounded below by
$c_en^d(1+n)^{-d-2s}\geq c_eR^{-2s}$.  Adjusting the constants covers the
bounded range of $R$.
\end{proof}

\subsection{Algebraic comparison profiles}

The radial estimate is critical in a useful sense: the mass of
$\Theta_{R,s}$ is of order $R^d$, while the long-jump kernel has exponent
$d+2s$.  Their product supplies precisely the factor $R^{-2s}$.

\begin{lemma}\label{tr:lem:uniform-radial-profile}
There is $C>0$ such that
\begin{equation}\label{tr:eq:radial-profile-bound-detailed}
  |A^s\Theta_{R,s}(x)|
  \leq CR^{-2s}\Theta_{R,s}(x)
\end{equation}
for $s\in[s_*,1)$, $R\geq1$, and $x\in\Z^d$.
\end{lemma}

\begin{proof}
Write $q_s=d+2s$ and $L=R+|x|$.  Symmetry of $K_s$ gives
\begin{equation}\label{tr:eq:symmetrized-radial-operator}
  A^s\Theta_{R,s}(x)=\frac12\sum_{z\neq0}K_s(z)
  \bigl(2\Theta_{R,s}(x)-\Theta_{R,s}(x+z)
  -\Theta_{R,s}(x-z)\bigr).
\end{equation}
Direct differentiation of the continuous extension shows that
\begin{equation}\label{tr:eq:radial-relative-hessian}
  \sup_{|\vartheta|\leq1}
  |D^2\Theta_{R,s}(x+\vartheta z)|
  \leq CL^{-2}\Theta_{R,s}(x)
\end{equation}
whenever $|z|\leq L/4$.  The constant is uniform because $q_s$ ranges in
a compact interval.  For $|x|\leq2R$ this follows from the scale $R$;
for $|x|>2R$ the segment in \eqref{tr:eq:radial-relative-hessian} stays at
distance comparable with $|x|$ from the origin.

Separate the nearest-neighbor part in
\eqref{tr:eq:jump-decomposition}.  Taylor's formula,
Lemma~\ref{tr:lem:uniform-long-jumps}, and
\eqref{tr:eq:uniform-second-moment-sum} show that the contribution of
$0<|z|\leq L/4$ is at most
\begin{align}
  &CL^{-2}\Theta_{R,s}(x)
  \left(1+(1-s)
  \sum_{1\leq|z|\leq L/4}|z|^{2-d-2s}\right) \notag\\
  &\hspace{35mm}\leq CL^{-2s}\Theta_{R,s}(x).
  \label{tr:eq:radial-near-field-uniform}
\end{align}
The cancellation of $(1-s)$ against $(2-2s)^{-1}$ is what makes the last
constant uniform.

For the complementary sum,
Lemma~\ref{tr:lem:uniform-long-jumps} gives
\begin{equation}\label{tr:eq:radial-far-tail-uniform}
  \sum_{|z|>L/4}\widetilde K_s(z)\leq C(1-s)L^{-2s}.
\end{equation}
The nearest-neighbor part is absent when $L>4$; the bounded range $L\leq4$
follows directly from $\|A^s\|\leq2$.  Since $q_s>d$,
\begin{equation}\label{tr:eq:radial-profile-mass-uniform}
  \sum_{y\in\Z^d}\Theta_{R,s}(y)\leq CR^d.
\end{equation}
If $|z|>L/4$, then
$\widetilde K_s(z)\leq C(1-s)L^{-q_s}$, and hence
\[
  \sum_{|z|>L/4}\widetilde K_s(z)\Theta_{R,s}(x\pm z)
  \leq C(1-s)R^dL^{-q_s}.
\]
Finally,
\begin{equation}\label{tr:eq:radial-critical-comparison-uniform}
  R^dL^{-q_s}\leq CR^{-2s}\Theta_{R,s}(x).
\end{equation}
Combining \eqref{tr:eq:radial-near-field-uniform}--
\eqref{tr:eq:radial-critical-comparison-uniform} with
\eqref{tr:eq:symmetrized-radial-operator} proves the lemma.
\end{proof}

\begin{lemma}\label{tr:lem:uniform-directional-profile}
For each $e\in\mathbb S^{d-1}$ there is $C_e>0$ such that
\begin{equation}\label{tr:eq:directional-profile-bound-detailed}
  |A^s\Psi_{R,e,\kappa,s}(x)|
  \leq C_eR^{-2s}\Psi_{R,e,\kappa,s}(x)
\end{equation}
for $s\in[s_*,1)$, $R\geq1$, $\kappa\in\R$, and $x\in\Z^d$.
Moreover,
\begin{align}
  S_s(1)H_{e,\kappa}(x)
  &\geq c_e(1-s)\Psi_{1,e,\kappa,s}(x),
  \label{tr:eq:cumulative-one-jump-lower}\\
  S_s(t)H_{e,\kappa}(x)
  &\leq C_e(1+t+t^m)\Psi_{1,e,\kappa,s}(x).
  \label{tr:eq:cumulative-heat-upper-detailed}
\end{align}
\end{lemma}

\begin{proof}
Put $r=x\cdot e-\kappa$.  If $r\geq R$, the summand in the jump
representation vanishes unless $r+z\cdot e<0$.  The nearest-neighbor term
either vanishes or is covered by enlarging the constant when $R$ is
bounded.  Equations \eqref{tr:eq:long-jump-two-sided} and
\eqref{tr:eq:uniform-half-space-tail} then give
\[
  |A^s\Psi_{R,e,\kappa,s}(x)|
  \leq C_e(1-s)r^{-2s}\leq C_eR^{-2s}.
\]
If $0\leq r<R$, the continuous one-dimensional profile belongs to
$C^{1,1}$ and its second derivative is bounded by $CR^{-2}$.  Splitting at
$|z|=R$ and using Lemma~\ref{tr:lem:uniform-lattice-sums} gives
\[
  |A^s\Psi_{R,e,\kappa,s}(x)|
  \leq CR^{-2}\left(1+(1-s)
  \sum_{1\leq|z|\leq R}|z|^{2-d-2s}\right)
  +C(1-s)R^{-2s}\leq CR^{-2s}.
\]
In these two cases the profile equals one.

Suppose $r<0$ and let $L=R+|r|$.  For $|z|\leq L/4$, the same
relative Hessian argument as in \eqref{tr:eq:radial-relative-hessian} gives
\[
  |2\Psi(x)-\Psi(x+z)-\Psi(x-z)|
  \leq C\Psi(x)\frac{|z|^2}{L^2},
\]
where $\Psi=\Psi_{R,e,\kappa,s}$.  Thus the near part is bounded by
$CL^{-2s}\Psi(x)$.  On the far part, the tail mass contributes
$C(1-s)L^{-2s}\Psi(x)$.  Since the shifted profiles are at most one,
their contribution is bounded by the same tail mass without the factor
$\Psi(x)$, namely $C(1-s)L^{-2s}$.  Since
\begin{equation}\label{tr:eq:directional-critical-comparison}
  L^{-2s}\leq C_eR^{-2s}\Psi_{R,e,\kappa,s}(x),
\end{equation}
we obtain \eqref{tr:eq:directional-profile-bound-detailed}.

For the semigroup estimates, observe that
\[
  S_s(t)H_{e,\kappa}(x)
  =\sum_{z\cdot e\leq r}G_{s,t}(z).
\]
If $r\geq0$, the zero-jump term proves the lower bound, and the Markov
property proves the upper bound.  If $r<0$, combine
\eqref{tr:eq:heat-one-jump-lower} or \eqref{tr:eq:heat-polynomial-upper} with
\eqref{tr:eq:uniform-half-space-tail}.  For bounded negative $r$ the
one-jump estimate at a fixed admissible lattice vector supplies the same
lower bound. More precisely, for $-1<r<0$, choose $z_e\in\Z^d$ with
$z_e\cdot e\leq-1$ and use the term $G_{s,1}(z_e)$ in the cumulative sum.
Since
$\Psi_{1,e,\kappa,s}(x)\asymp(1+|r|)^{-2s}$ for $r<0$, this proves
\eqref{tr:eq:cumulative-one-jump-lower} and
\eqref{tr:eq:cumulative-heat-upper-detailed}.
\end{proof}

\begin{proposition}\label{tr:prop:uniform-estimates}
Fix $s_*\in(1/2,1)$. There exist $c,C>0$ and an integer $m\geq d+4$
such that, for every $s\in[s_*,1)$, the following estimates hold:
\begin{align}
  c(1-s)(1+|z|)^{-d-2s}
  &\leq\widetilde K_s(z)
  \leq C(1-s)(1+|z|)^{-d-2s},
  &&z\neq0,
  \label{tr:eq:uniform-long-kernel}\\
  G_{s,1}(z)
  &\geq c(1-s)(1+|z|)^{-d-2s},
  &&z\in\Z^d,
  \label{tr:eq:uniform-fixed-time-lower}\\
  G_{s,t}(z)
  &\leq C(t+t^m)(1+|z|)^{-d-2s},
  &&t>0,\quad z\neq0,
  \label{tr:eq:uniform-heat-upper}\\
  |A^s\Theta_{R,s}(x)|
  &\leq CR^{-2s}\Theta_{R,s}(x),
  &&R\geq1,
  \label{tr:eq:uniform-radial-weight}\\
  |A^s\Psi_{R,e,\kappa,s}(x)|
  &\leq C_eR^{-2s}\Psi_{R,e,\kappa,s}(x),
  &&R\geq1.
  \label{tr:eq:uniform-front-weight}
\end{align}
In addition,
\begin{equation}\label{tr:eq:uniform-cumulative-upper}
  S_s(t)H_{e,\kappa}(x)
  \leq C_e(1+t+t^m)\Psi_{1,e,\kappa,s}(x)
\end{equation}
for $t>0$ and $x\in\Z^d$.
\end{proposition}

\begin{proof}
The kernel estimate is Lemma~\ref{tr:lem:uniform-long-jumps}.  The two
heat-kernel bounds follow from Lemma~\ref{tr:lem:uniform-heat-kernel}, and the
profile estimates are Lemmas~\ref{tr:lem:uniform-radial-profile} and
\ref{tr:lem:uniform-directional-profile}.  Equation
\eqref{tr:eq:uniform-cumulative-upper} is
\eqref{tr:eq:cumulative-heat-upper-detailed}.
\end{proof}

\section{The accelerating regime}\label{tr:sec:acceleration}

The initial algebraic tail has amplitude $1-s$. It reaches a fixed height
after a time comparable with $\tau_s$; from then on, the profiles in
\eqref{tr:eq:radial-profile} and \eqref{tr:eq:front-profile} expand with
constants independent of $s$. The scalar calculation below applies to both
radial and front-like data and raises the amplitude without changing the
spatial scale.

\begin{lemma}\label{tr:lem:fixed-scale-amplification}
Let $0<\Phi\leq1$ and suppose that
\begin{equation}\label{tr:eq:abstract-weighted-bound}
  |A^s\Phi(x)|\leq\delta\Phi(x)
  \qquad (x\in\Z^d)
\end{equation}
for some $0<\delta\leq a/4$.  If $0<\beta_0<1/2$ and
\begin{equation}\label{tr:eq:abstract-logistic-ode}
  \beta'=(a-\delta)\beta-a\beta^2,
  \qquad \beta(0)=\beta_0,
\end{equation}
then $\beta(t)\Phi$ is a subsolution of \eqref{eq:intro-equation}.  The first
time $T$ at which $\beta(T)=1/2$ satisfies
\begin{equation}\label{tr:eq:abstract-amplification-time}
  T\leq C_a(1+|\log\beta_0|).
\end{equation}
The constant is independent of $s$, $\delta$, and $\Phi$.
\end{lemma}

\begin{proof}
Since $0<\Phi\leq1$, equations \eqref{tr:eq:abstract-weighted-bound} and
\eqref{tr:eq:abstract-logistic-ode} give
\begin{align*}
  \partial_t(\beta\Phi)+A^s(\beta\Phi)
  &\leq \bigl((a-\delta)\beta-a\beta^2+\delta\beta\bigr)\Phi\\
  &=a\beta\Phi(1-\beta)
  \leq a\beta\Phi(1-\beta\Phi).
\end{align*}
The positive equilibrium in \eqref{tr:eq:abstract-logistic-ode} is
$b_*=(a-\delta)/a\geq3/4$, and
\begin{equation}\label{tr:eq:explicit-logistic-solution}
  \beta(t)=
  \frac{b_*}{1+(b_*/\beta_0-1)\ee^{-(a-\delta)t}}.
\end{equation}
Solving $\beta(T)=1/2$ in this formula and using
$a-\delta\geq3a/4$ proves \eqref{tr:eq:abstract-amplification-time}.
\end{proof}

\begin{lemma}\label{tr:lem:localized-amplification}
Assume the hypotheses of Theorem~\ref{tr:thm:localized-transition}. There are
$C_0>0$, $\rho_0\geq1$, and, for each $s\in[s_*,1)$, a time
\begin{equation}\label{tr:eq:start-time-bound}
  T_s\leq C_0(1+\tau_s)
\end{equation}
such that
\begin{equation}\label{tr:eq:radial-start}
  u_s(T_s,x)\geq\frac12\Theta_{\rho_0,s}(x),
  \qquad x\in\Z^d.
\end{equation}
For every $\eta\in(0,a)$ and $\theta\in(0,1)$ there are
$c_{\eta,\theta},T_{\eta,\theta}>0$, independent of $s$, such that
\begin{equation}\label{tr:eq:uniform-radial-expansion}
  u_s(T_s+t,x)\geq\theta
\end{equation}
whenever
\begin{equation}\label{tr:eq:uniform-radial-region}
  t\geq T_{\eta,\theta},
  \qquad
  |x|\leq c_{\eta,\theta}
  \exp\left(\frac{(a-\eta)t}{d+2s}\right).
\end{equation}
\end{lemma}

\begin{proof}
Choose $y_0$ with $u_0(y_0)>0$. Since the reaction is nonnegative on
$[0,1]$, comparison with the linear diffusion equation and
\eqref{tr:eq:uniform-fixed-time-lower} give
\begin{equation}\label{tr:eq:one-jump-seed}
  u_s(1,x)
  \geq u_0(y_0)G_{s,1}(x-y_0)
  \geq c(1-s)(1+|x|)^{-d-2s}.
\end{equation}
Let $C$ be the constant in \eqref{tr:eq:uniform-radial-weight}, and fix
$M>C$. Choose $\rho_0$ so large that
\begin{equation}\label{tr:eq:rho-choice}
  M\rho_0^{-2s}\leq\frac a4
\end{equation}
for $s\in[s_*,1)$. Since
\[
  \Theta_{\rho_0,s}(x)
  \leq C\rho_0^{d+2}(1+|x|)^{-d-2s},
\]
equation \eqref{tr:eq:one-jump-seed} implies
\[
  \beta_{0,s}\Theta_{\rho_0,s}(x)\leq u_s(1,x),
  \qquad
  \beta_{0,s}=c_0(1-s),
\]
after decreasing $c_0$.

Let
\begin{equation}\label{tr:eq:fixed-radius-ode}
  \beta'=a\beta(1-\beta)-C\rho_0^{-2s}\beta,
  \qquad \beta(0)=\beta_{0,s}.
\end{equation}
By \eqref{tr:eq:uniform-radial-weight},
$\beta(t)\Theta_{\rho_0,s}$ is a subsolution. The positive equilibrium of
\eqref{tr:eq:fixed-radius-ode} is at least $3/4$. Its explicit logistic formula
shows that the first time at which $\beta=1/2$ is bounded by
$C(1+|\log\beta_{0,s}|)$. This proves
\eqref{tr:eq:start-time-bound} and \eqref{tr:eq:radial-start}.

Expansion is obtained by allowing the radius to vary. With the choices of
$M$ and $\rho_0$ already fixed, let
$(\alpha,\rho)$ solve
\begin{align}
  \alpha'&=a\alpha(1-\alpha)-M\rho^{-2s}\alpha,
  \label{tr:eq:coupled-alpha}\\
  (d+2s)\frac{\rho'}\rho&=a\alpha-M\rho^{-2s},
  \label{tr:eq:coupled-radius}
\end{align}
with $\alpha(0)=1/2$ and $\rho(0)=\rho_0$. Then
$1/2\leq\alpha<1$ and $\rho$ is increasing. Indeed, on the two boundaries
of the amplitude interval,
\[
  \alpha'=\frac a4-\frac M2\rho^{-2s}\geq\frac a8
  \quad\text{at }\alpha=\frac12,
  \qquad
  \alpha'=-M\rho^{-2s}<0
  \quad\text{at }\alpha=1.
\]
Within this interval,
\begin{equation}\label{tr:eq:uniform-radius-first-growth}
  (d+2s)\frac{\rho'}\rho
  \geq\frac a2-M\rho_0^{-2s}\geq\frac a4.
\end{equation}
There can therefore be no first exit from
$[1/2,1]\times[\rho_0,\infty)$.  The vector field has at most linear
growth in $\rho$, so the solution is global.

For the barrier calculation, differentiation of
\eqref{tr:eq:radial-profile} gives
\begin{equation}\label{tr:eq:radial-profile-time-derivative}
  \partial_t\Theta_{\rho(t),s}(x)
  =(d+2s)\frac{\rho'}\rho
  \Theta_{\rho(t),s}(x)
  \bigl(1-\Theta_{\rho(t),s}(x)\bigr).
\end{equation}
Substitution of \eqref{tr:eq:coupled-alpha} and
\eqref{tr:eq:coupled-radius} gives the exact identity
\begin{equation}\label{tr:eq:radial-coupled-identity}
  \partial_t(\alpha\Theta_{\rho,s})
  -a\alpha\Theta_{\rho,s}(1-\alpha\Theta_{\rho,s})
  =-M\rho^{-2s}\alpha\Theta_{\rho,s}
  (2-\Theta_{\rho,s}).
\end{equation}
Since $2-\Theta_{\rho,s}\geq1$, the weighted estimate
\eqref{tr:eq:radial-profile-bound-detailed} and $M>C$ imply
\[
  \partial_t(\alpha\Theta_{\rho,s})
  +A^s(\alpha\Theta_{\rho,s})
  -a\alpha\Theta_{\rho,s}(1-\alpha\Theta_{\rho,s})
  \leq-(M-C)\rho^{-2s}\alpha\Theta_{\rho,s}.
\]
Thus $\alpha(t)\Theta_{\rho(t),s}$ is a subsolution.

It remains to quantify the radius.  From
\eqref{tr:eq:uniform-radius-first-growth},
\begin{equation}\label{tr:eq:uniform-forcing-decay}
  \rho(t)^{-2s}
  \leq\rho_0^{-2s}
  \exp\left(-\frac{as}{2(d+2s)}t\right)
  \leq C\ee^{-c t},
\end{equation}
where $c>0$ is independent of $s\in[s_*,1)$.  With
$D=1-\alpha$, the amplitude equation gives
\[
  D'=-a\alpha D+M\rho^{-2s}\alpha
  \leq-\frac a2D+M\rho^{-2s}.
\]
Variation of constants and \eqref{tr:eq:uniform-forcing-decay} show that
$D(t)\leq C\ee^{-c_1t}$, with uniform constants.  In particular,
$\alpha(t)\to1$ uniformly and
$\int_0^\infty(1-\alpha(t))\,dt<\infty$ uniformly in $s$.

Given $\eta>0$, choose $t_\eta$ independently of $s$ so that
$a(1-\alpha(t))+M\rho(t)^{-2s}\leq\eta$ for $t\geq t_\eta$.
The radius equation then yields
\[
  \qquad
  \rho(t)\geq \rho(t_\eta)
  \exp\left(\frac{(a-\eta)(t-t_\eta)}{d+2s}\right).
\]
After absorbing the fixed shift $t_\eta$ into a constant, this becomes
\begin{equation}\label{tr:eq:uniform-radius-final-growth}
  \rho(t)\geq c_\eta
  \exp\left(\frac{(a-\eta)t}{d+2s}\right).
\end{equation}
Comparison with \eqref{tr:eq:radial-start} gives
$u_s(T_s+t,x)\geq\alpha(t)\Theta_{\rho(t),s}(x)$.  Once
$\alpha(t)\geq(1+\theta)/2$, the product is at least $\theta$ whenever
\[
  |x|\leq\rho(t)
  \left(\frac{1-\theta}{2\theta}\right)^{1/(d+2s)}.
\]
Together with \eqref{tr:eq:uniform-radius-final-growth}, this proves
\eqref{tr:eq:uniform-radial-expansion}--\eqref{tr:eq:uniform-radial-region}.
\end{proof}

\begin{proof}[Proof of the late part of
Theorem~\ref{tr:thm:localized-transition}]
The KPP upper bound gives
\[
  u_s(t)\leq\ee^{at}S_s(t)u_0.
\]
Since $u_0$ has finite support, \eqref{tr:eq:uniform-heat-upper} yields
\begin{equation}\label{tr:eq:localized-upper}
  u_s(t,x)
  \leq C\ee^{at}(t+t^m)(1+|x|)^{-d-2s}
\end{equation}
outside a fixed ball. If
$|x|\geq\exp(\sigma_+t)$, the right-hand side is bounded by
\[
  C(t+t^m)
  \exp\bigl(-((d+2s)\sigma_+-a)t\bigr).
\]
Since $\sigma_+>a/(d+2)$, there are $\nu_+>0$ and $s_+<1$ such that
$(d+2s)\sigma_+-a\geq\nu_+$ for $s\in[s_+,1)$. This proves
\eqref{tr:eq:late-localized-outer}.

For the lower bound, choose $\eta\in(0,a)$ so that
\[
  \nu_-:=a-\eta-(d+2)\sigma_->0.
\]
Fix $\theta\in(0,1)$ and apply
Lemma~\ref{tr:lem:localized-amplification} with this $\eta$ and $\theta$.
For $s\in[s_*,1)$,
\[
  a-\eta-(d+2s)\sigma_-\geq\nu_-.
\]
Since $T_s\leq C_0(1+\tau_s)$, we may choose
$C_{\sigma_-,\sigma_+}>0$, independent of $s$ and $\theta$, so that
\begin{equation}\label{tr:eq:localized-rate-absorption}
  (a-\eta)(t-T_s)
  \geq(d+2s)\sigma_-t+\frac{\nu_-}{2}t
\end{equation}
whenever $t\geq C_{\sigma_-,\sigma_+}\tau_s$ and $s$ is sufficiently
close to one. The fixed factor $c_{\eta,\theta}$ in
\eqref{tr:eq:uniform-radial-region} is absorbed by the last term in
\eqref{tr:eq:localized-rate-absorption}. Hence the region in
\eqref{tr:eq:uniform-radial-region}, with time $t-T_s$, contains
$\{|x|\leq\ee^{\sigma_-t}\}$ for all sufficiently large $s$ and $t$.
The lower limit in \eqref{tr:eq:late-localized-inner} is therefore at least
$\theta$. Letting $\theta\uparrow1$ proves the assertion.
\end{proof}

We turn to front-like data. Summation over the transverse variables changes
the pointwise exponent $d+2s$ into the cumulative exponent $2s$.

\begin{lemma}\label{tr:lem:front-amplification}
Assume \eqref{tr:eq:front-data}. There are $C_1>0$ and, for each
$s\in[s_*,1)$, a time $\widehat T_s\leq C_1(1+\tau_s)$ such that, for every
$\gamma\in(0,a)$ and $\theta\in(0,1)$, there exist
$c_{\gamma,\theta},T_{\gamma,\theta}>0$, independent of $s$, for which
\begin{equation}\label{tr:eq:uniform-front-expansion}
  u_s(\widehat T_s+t,x)\geq\theta
\end{equation}
whenever
\begin{equation}\label{tr:eq:uniform-front-region}
  t\geq T_{\gamma,\theta},
  \qquad
  x\cdot e\geq-c_{\gamma,\theta}
  \exp\left(\frac{\gamma t}{2s}\right).
\end{equation}
\end{lemma}

\begin{proof}
The lower inequality in \eqref{tr:eq:front-data}, the one-jump term, and
\eqref{tr:eq:uniform-long-kernel} give
\begin{equation}\label{tr:eq:front-seed}
  u_s(1,x)
  \geq c(1-s)\Psi_{1,e,\kappa_+,s}(x).
\end{equation}
Here we used the uniform lattice half-space estimate
\[
  c_eR^{-2s}
  \leq\sum_{z\cdot e\leq-R}(1+|z|)^{-d-2s}
  \leq C_eR^{-2s},
  \qquad R\geq1,
\]
which follows by summing over lattice shells for the upper bound and over a
fixed cone about $-e$ for the lower bound.
Choose a fixed $R_*$ so large that
$C_eR_*^{-2s}\leq a/4$. The fixed-radius scalar subsolution used in
\eqref{tr:eq:fixed-radius-ode}, now multiplied by
$\Psi_{R_*,e,\kappa_+,s}$, raises the amplitude in
\eqref{tr:eq:front-seed} to a fixed number $\eta_0>0$ in at most
$C(1+\tau_s)$ time.  More precisely,
\[
  \Psi_{R_*,e,\kappa_+,s}
  \leq R_*^{2s}\Psi_{1,e,\kappa_+,s},
\]
so the scalar equation may be started at
$c(1-s)R_*^{-2s}$.  Its positive equilibrium is bounded away from zero
uniformly in $s$, and the explicit logistic formula gives the asserted
time bound.  We have therefore obtained
\begin{equation}\label{tr:eq:front-fixed-amplitude-start}
  u_s(\widehat T_s,x)
  \geq\eta_0\Psi_{R_*,e,\kappa_+,s}(x).
\end{equation}

For the expanding profile, fix $\gamma\in(0,a)$ and put
$\delta=(a-\gamma)/3$. Choose $R_\gamma\geq R_*$ so large that
\begin{equation}\label{tr:eq:directional-radius-choice}
  C_eR_\gamma^{-2s}\leq a-\delta-\gamma
  \qquad (s_*\leq s<1).
\end{equation}
Choose $\eta_\gamma>0$ so that
\[
  a\eta_\gamma\leq\delta,
  \qquad
  \eta_\gamma\leq
  \eta_0\left(\frac{R_*}{R_\gamma}\right)^2.
\]
Since
\[
  \Psi_{R_\gamma,e,\kappa_+,s}
  \leq\left(\frac{R_\gamma}{R_*}\right)^{2s}
  \Psi_{R_*,e,\kappa_+,s}
  \leq\left(\frac{R_\gamma}{R_*}\right)^2
  \Psi_{R_*,e,\kappa_+,s},
\]
the new profile is ordered below the right-hand side of
\eqref{tr:eq:front-fixed-amplitude-start} at time $\widehat T_s$.
Set
\[
  R(t)=R_\gamma\exp\left(\frac{\gamma t}{2s}\right),
  \qquad
  \underline u(t,x)=
  \eta_\gamma\Psi_{R(t),e,\kappa_+,s}(x).
\]
Writing $r=x\cdot e-\kappa_+$, direct differentiation gives
\begin{equation}\label{tr:eq:directional-profile-time-derivative}
  \partial_t\underline u
  =\gamma\frac{(r^-)^2}{R(t)^2+(r^-)^2}\,\underline u
  \leq\gamma\underline u.
\end{equation}
By \eqref{tr:eq:directional-profile-bound-detailed} and
\eqref{tr:eq:directional-radius-choice},
\[
  \partial_t\underline u+A^s\underline u
  \leq(\gamma+C_eR_\gamma^{-2s})\underline u
  \leq(a-\delta)\underline u.
\]
On the other hand, $0\leq\underline u\leq\eta_\gamma$ and therefore
$a\underline u(1-\underline u)\geq(a-\delta)\underline u$.
Thus $\underline u$ is a subsolution.  Its initial value lies below the
right-hand side of \eqref{tr:eq:front-fixed-amplitude-start}, so comparison
gives
\begin{equation}\label{tr:eq:expanding-directional-subsolution}
  u_s(\widehat T_s+t,x)
  \geq\eta_\gamma\Psi_{R(t),e,\kappa_+,s}(x),
  \qquad t\geq0.
\end{equation}
In particular, since $2^{-s}\geq1/2$, the solution is at least
$\eta_\gamma/2$ whenever
$x\cdot e\geq\kappa_+-R(t)$.

We finally raise this plateau to the prescribed level $\theta$. Choose
$L_{\gamma,\theta}>0$ so that the solution of $b'=ab(1-b)$ starting at
$\eta_\gamma$ exceeds $(1+\theta)/2$ at time $L_{\gamma,\theta}$. If $t$ is
large enough, set
\[
  \widehat R=R(t-L_{\gamma,\theta}).
\]
The loss $C_e\widehat R^{-2s}$ is then so small, uniformly in $s$, that
the solution of
\[
  \beta'=a\beta(1-\beta)-C_e\widehat R^{-2s}\beta,
  \qquad \beta(0)=\eta_\gamma,
\]
satisfies $\beta(L_{\gamma,\theta})\geq(1+\theta)/2$. Starting this
fixed-radius subsolution at time
$\widehat T_s+t-L_{\gamma,\theta}$ is allowed by
\eqref{tr:eq:expanding-directional-subsolution}.  Hence at time
$\widehat T_s+t$ the solution is at least $\theta$ wherever
\[
  \Psi_{\widehat R,e,\kappa_+,s}(x)
  \geq\frac{2\theta}{1+\theta}.
\]
This region contains
\[
  x\cdot e\geq\kappa_+-c_\theta\widehat R
  \geq-c_{\gamma,\theta}
  \exp\left(\frac{\gamma t}{2s}\right),
\]
after a uniform enlargement of the lower time bound.  This proves
\eqref{tr:eq:uniform-front-expansion}--\eqref{tr:eq:uniform-front-region}.
\end{proof}

\begin{proof}[Proof of the late part of
Theorem~\ref{tr:thm:front-transition}]
The upper half-space in \eqref{tr:eq:front-data}, linear comparison, and
\eqref{tr:eq:uniform-cumulative-upper} give
\begin{equation}\label{tr:eq:front-upper}
  u_s(t,x)
  \leq C\ee^{at}(1+t+t^m)\Psi_{1,e,\kappa_-,s}(x).
\end{equation}
For $x\cdot e\leq-\exp(\sigma_+t)$, the right-hand side of
\eqref{tr:eq:front-upper} is at most
\[
  C(1+t+t^m)\exp\bigl(-(2s\sigma_+-a)t\bigr).
\]
Since $\sigma_+>a/2$, the exponent is bounded above by $-\nu_+t$ for
some $\nu_+>0$ and all $s$ sufficiently close to one. This proves
\eqref{tr:eq:late-front-outer}.

Choose $\gamma$ so that $2\sigma_-<\gamma<a$, and set
$\nu_-:=\gamma-2\sigma_->0$. Fix $\theta\in(0,1)$ and apply
Lemma~\ref{tr:lem:front-amplification} with this $\gamma$ and $\theta$.
Since $\widehat T_s\leq C_1(1+\tau_s)$ and
$\gamma-2s\sigma_-\geq\nu_-$, one can choose
$\widehat C_{\sigma_-,\sigma_+}>0$, independent of $s$ and $\theta$, so
that
\[
  \gamma(t-\widehat T_s)
  \geq2s\sigma_-t+\frac{\nu_-}{2}t
\]
whenever $t\geq\widehat C_{\sigma_-,\sigma_+}\tau_s$ and $s$ is
sufficiently close to one. The fixed factor $c_{\gamma,\theta}$ in
\eqref{tr:eq:uniform-front-region} is absorbed by the remaining exponential
factor. Thus the half-space in \eqref{tr:eq:uniform-front-region}, with time
$t-\widehat T_s$, contains
$\{x:x\cdot e\geq-\ee^{\sigma_-t}\}$ for all sufficiently large $s$ and
$t$. The lower limit in \eqref{tr:eq:late-front-inner} is at least $\theta$.
Letting $\theta\uparrow1$ completes the proof.
\end{proof}

\section{Level-set formulation and scale separation}
\label{tr:sec:level-set-formulation}

The transition theorems are stated as uniform estimates on expanding
regions. Their equivalent logarithmic formulation is recorded here.

For localized data and $\theta\in(0,1)$, define
\begin{align}
  R^{\mathrm{in}}_{\theta,s}(t)
  &:=\sup\left\{R\geq1:
  \inf_{\substack{x\in\Z^d\\ |x|\leq R}}u_s(t,x)\geq\theta\right\},
  \label{tr:eq:inner-level-radius}\\
  R^{\mathrm{out}}_{\theta,s}(t)
  &:=\inf\left\{R\geq1:
  \sup_{\substack{x\in\Z^d\\ |x|\geq R}}u_s(t,x)\leq\theta\right\}.
  \label{tr:eq:outer-level-radius}
\end{align}
The usual conventions are used when one of the defining sets is empty.
For front-like data, put
\begin{align}
  Q^{\mathrm{in}}_{\theta,s}(t)
  &:=\sup\left\{R\geq1:
  \inf_{\substack{x\in\Z^d\\ x\cdot e\geq-R}}
  u_s(t,x)\geq\theta\right\},
  \label{tr:eq:inner-front-radius}\\
  Q^{\mathrm{out}}_{\theta,s}(t)
  &:=\inf\left\{R\geq1:
  \sup_{\substack{x\in\Z^d\\ x\cdot e\leq-R}}
  u_s(t,x)\leq\theta\right\}.
  \label{tr:eq:outer-front-radius}
\end{align}
The estimates in Theorems~\ref{tr:thm:localized-transition} and
\ref{tr:thm:front-transition} show that, in the limits considered below, these
radii are finite and at least one for all sufficiently large times.

\begin{corollary}\label{tr:cor:late-logarithmic-radii}
Let $s\uparrow1$ and $t\to\infty$, with
\begin{equation}\label{tr:eq:well-after-transition}
  \frac{t}{\tau_s}\to\infty.
\end{equation}
For every $\theta\in(0,1)$ and every nonzero finitely supported datum,
\begin{align}
  \frac1t\log R^{\mathrm{in}}_{\theta,s}(t)
  -\frac{a}{d+2s}&\to0,
  \label{tr:eq:inner-radius-log-limit}\\
  \frac1t\log R^{\mathrm{out}}_{\theta,s}(t)
  -\frac{a}{d+2s}&\to0.
  \label{tr:eq:outer-radius-log-limit}
\end{align}
Under \eqref{tr:eq:front-data},
\begin{align}
  \frac1t\log Q^{\mathrm{in}}_{\theta,s}(t)
  -\frac{a}{2s}&\to0,
  \label{tr:eq:inner-front-log-limit}\\
  \frac1t\log Q^{\mathrm{out}}_{\theta,s}(t)
  -\frac{a}{2s}&\to0.
  \label{tr:eq:outer-front-log-limit}
\end{align}
\end{corollary}

\begin{proof}
Fix $\delta>0$ sufficiently small and put
\[
  \sigma_-:=\frac{a}{d+2}-\delta,
  \qquad
  \sigma_+:=\frac{a}{d+2}+\delta.
\]
Condition \eqref{tr:eq:well-after-transition} implies
$t\geq C_{\sigma_-,\sigma_+}\tau_s$ in the stated limit. Theorem
\ref{tr:thm:localized-transition} shows that the solution is larger than
$\theta$ on the ball of radius $\ee^{\sigma_-t}$ and smaller than
$\theta$ outside the ball of radius $\ee^{\sigma_+t}$. Hence both radii
in \eqref{tr:eq:inner-level-radius}--\eqref{tr:eq:outer-level-radius} lie
between these two scales, up to an immaterial lattice error. Divide their
logarithms by $t$ and let $\delta\downarrow0$. This gives convergence to
$a/(d+2)$. Since $a/(d+2s)\to a/(d+2)$, equations
\eqref{tr:eq:inner-radius-log-limit} and
\eqref{tr:eq:outer-radius-log-limit} follow.

For the front-like data, take
$\sigma_-=a/2-\delta$ and $\sigma_+=a/2+\delta$ and apply
Theorem~\ref{tr:thm:front-transition}. Letting $\delta\downarrow0$ gives the
limit $a/2$, which differs from $a/(2s)$ by a quantity tending to zero.
This proves \eqref{tr:eq:inner-front-log-limit} and
\eqref{tr:eq:outer-front-log-limit}.
\end{proof}

The factor $1-s$ can also be read directly from the linearized equation.
The one-jump term in \eqref{tr:eq:compound-poisson-detailed} shows that a
localized seed creates, after one unit of time, a tail comparable with
\[
  (1-s)(1+|x|)^{-d-2s}.
\]
Linear growth multiplies its amplitude by $\ee^{at}$. At bounded spatial
distance this contribution becomes order one when $at\approx\tau_s$. At a
large radius $R$, the same balance reads
\begin{equation}\label{tr:eq:radial-balance-heuristic}
  (1-s)\ee^{at}R^{-d-2s}\asymp1.
\end{equation}
Once the factor $1-s$ has been absorbed during the initialization time,
\eqref{tr:eq:radial-balance-heuristic} gives
$\log R\sim at/(d+2s)$. Summation over an occupied half-space replaces
$R^{-d-2s}$ by $R^{-2s}$ through
\eqref{tr:eq:uniform-half-space-tail}; this gives the front-like exponent
$a/(2s)$. These balances do not replace the comparison proof, but they
explain why the two exponents and the same transition time occur together.

\begin{remark}\label{tr:rem:intermediate-window}
The proof separates the ranges $at<\tau_s$ and $t\geq C\tau_s$. It does
not determine the level-set geometry between these two time scales. Near
$at=\tau_s$, the amplified coefficient of the long-jump tail is no longer
small in operator norm, although its spatial decay is still relevant on the
ballistic scale. Determining the level-set geometry in the remaining interval
would require matching the local large-deviation profile with the emerging
algebraic tail and then following the latter until the uniform expanding
barrier reaches its asymptotic rate.
\end{remark}

\section{Further remarks}

\begin{remark}
The exponent $d+2s$ is determined by the tail of the L\'evy kernel and not by a diffusive scaling argument. The lower estimate uses the one-jump contribution to create a tail of order $|x|^{-d-2s}$, while the upper estimate shows that repeated jumps introduce only a polynomial factor in time. Therefore neither side changes the leading balance
\[
    \ee^{at}|x|^{-d-2s}\asymp1.
\]
\end{remark}

\begin{remark}
In Proposition~\ref{prop:plateau-upgrade}, the positive plateau must be available at an exponent $b_0$ strictly larger than the target exponent $b$. This guarantees that a fixed space--time neighborhood of a point satisfying $|x_n|\leq\ee^{bt_n}$ remains inside the plateau after translation, including at negative translated times. An estimate only at the same exponent does not provide this uniform inclusion near the boundary of the exponential ball.
\end{remark}

\begin{remark}
The proof uses four properties of the jump generator: summability and symmetry
of the kernel, the tail $|z|^{-d-\alpha}$, a fixed-time lower bound of the
same order, and a weighted estimate for a slowly varying algebraic profile.
The same argument applies to translation-invariant lattice jump generators
for which these estimates and the comparison principle hold. For regularly
varying kernels, see \cite{BinghamGoldieTeugels1987}; the corresponding
propagation exponent is $f'(0)/(d+\alpha)$.
\end{remark}

\end{document}